\documentclass[10pt]{amsart}

\usepackage{latexsym,exscale,enumerate,amsfonts,amssymb,xparse, mathtools}
\usepackage[normalem]{ulem}
\usepackage{amsmath,amsthm,amsfonts,amssymb,amscd, stmaryrd,textcomp}
\usepackage[bookmarks=false]{hyperref}
\usepackage{ulem}
\usepackage{bbold}
\usepackage[usenames,dvipsnames]{xcolor}
\usepackage{enumitem}
\usepackage{verbatim}

\numberwithin{equation}{section}

\theoremstyle{definition}
\newtheorem{thm}{Theorem}[section]
\newtheorem{cor}[thm]{Corollary}
\newtheorem{conj}[thm]{Conjecture}
\newtheorem{lem}[thm]{Lemma}
\newtheorem{rem}[thm]{Remark}

\newtheorem{prop}[thm]{Proposition}
\newtheorem{defn}[thm]{Definition}
\newtheorem{example}[thm]{Example}

\newtheorem*{ack}{Acknowledgements}

\usepackage[all]{xy}
\SelectTips{cm}{}

\usepackage{graphicx}

\usepackage{tikz}
\usetikzlibrary{cd}
\usetikzlibrary{decorations.markings}
\usetikzlibrary{decorations.pathreplacing}
\usetikzlibrary{arrows,shapes,positioning}
\usetikzlibrary{decorations.pathmorphing}
\tikzset{anchorbase/.style={baseline={([yshift=-0.5ex]current bounding box.center)}},
tinynodes/.style={font=\tiny,text height=0.75ex,text depth=0.15ex},
smallnodes/.style={font=\scriptsize,text height=0.75ex,text depth=0.15ex},
0label/.style={decorate, densely dashed},
1label/.style={very thick, black},
3label/.style={very thick, green},
overcross/.style={line width=5pt,color=white},
thinovercross/.style={line width=4pt,color=white},
2label/.style={very thick, ourblue},
glabel/.style={ultra thick, gray},
}

\colorlet{green}{black!30!green}
\definecolor{ourblue}{RGB}{109, 156, 179}

\newcommand{\sln}[1][n]{\mathfrak{sl}_{#1}}
\newcommand{\gln}[1][n]{\mathfrak{gl}_{#1}}
\newcommand{\spn}[1][2n]{\mathfrak{sp}_{#1}}
\newcommand{\gspn}[1][2n]{\mathfrak{gsp}_{#1}}
\newcommand{\largewedge}{\mbox{\Large $\wedge$}}
\newcommand{\Tr}{\operatorname{Tr}}
\newcommand{\tr}{\operatorname{tr}}
\newcommand{\Rep}{\mathbf{Rep}}
\newcommand{\FRep}{\mathbf{FundRep}}
\newcommand{\Web}{\mathbf{Web}}
\newcommand{\Lad}{\mathbf{Lad}}

\newcommand{\Hom}{{\rm Hom}}

\newcommand{\End}{{\rm End}}

\renewcommand{\to}{\rightarrow}

\newcommand{\id}{{\rm id}}

\newcommand{\BMW}{\mathrm{BMW}}
\newcommand{\Tot}{\operatorname{Tot}}
\let\hat=\widehat
\let\tilde=\widetilde

\usepackage{bbm}
\def\C{{\mathbbm C}}
\def\N{{\mathbbm N}}

\def\Z{{\mathbbm Z}}

\def\K{{\mathbbm K}}

\begin{document}
%

\title[]{On webs in quantum type $C$}

\author{David E. V. Rose}
\address{Department of Mathematics, University of North Carolina, 
Phillips Hall CB \#3250, UNC-CH, Chapel Hill, NC 27599-3250, USA}
\email{davidrose@unc.edu}

\author{Logan Tatham}
\address{Department of Mathematics, University of North Carolina, 
Phillips Hall CB \#3250, UNC-CH, Chapel Hill, NC 27599-3250, USA}
\email{ltatham@live.unc.edu}

\begin{abstract}
We study webs in quantum type $C$, focusing on the rank three case. 
We define a linear pivotal category $\Web(\spn[6])$ diagrammatically by generators and relations, 
and conjecture that it is equivalent to the category $\FRep(U_q(\spn[6]))$
of quantum $\spn[6]$ representations generated by the fundamental representations, 
for generic values of the parameter $q$.
We prove a number of results in support of this conjecture, 
most notably that there is a full, essentially surjective functor 
$\Web(\spn[6]) \to \FRep(U_q(\spn[6]))$, 
that all $\Hom$-spaces in $\Web(\spn[6])$ are finite-dimensional, 
and that the endomorphism algebra of the monoidal unit in $\Web(\spn[6])$ is $1$-dimensional. 
The latter corresponds to the statement that all closed webs can be evaluated to 
scalars using local relations; as such, we obtain a new approach to the quantum 
$\spn[6]$ link invariants, akin to the Kauffman bracket description of the Jones polynomial.

\end{abstract}

\maketitle

\section{Introduction}
Given a classical algebraic object, such as a group or ring, 
one of the basic questions one can ask is for a presentation of the object 
via generators and relations.
The advent of quantum invariants in low-dimensional topology suggested the 
investigation of a related problem, one ``categorical dimension'' higher: 
is it possible to find a presentation of a given monoidal category via generators and relations?
Indeed, work of Reshetikhin-Turaev \cite{RT1,RT2} shows that
suitable monoidal categories lead to invariants of links and $3$-manifolds, 
and, amongst other applications,
such presentations elucidate various properties of these invariants and related structures.

One of the first results in this direction is a folklore theorem 
that describes the category $\Rep(U_q(\sln[2]))$
of finite-dimensional representations of the quantum group 
$U_q(\sln[2])$ (see \cite{RTW} for a very early incarnation). 
Specifically, the full subcategory $\FRep(U_q(\sln[2]))$
of $\Rep(U_q(\sln[2]))$ monoidally 
generated by the fundamental representation 
is equivalent to the $\C(q)$-linear pivotal category freely generated by a single object, 
modulo a single relation in the endomorphism algebra of the monoidal unit.
Translated via the diagrammatic formalism for monoidal categories, 
this equivalently states that $\FRep(U_q(\sln[2]))$ is equivalent to the 
Temperley-Lieb category, wherein objects are non-negative integers, 
and morphisms $n \to m$ are $\C(q)$-linear combinations 
of planar $(m,n)$-tangles, 
modulo isotopy and the local relation
\[
\xy
(0,0)*{
\begin{tikzpicture}[scale =.75, smallnodes]
	\draw[1label] (0,0) circle (.5);
\end{tikzpicture}
};
\endxy
= -q-q^{-1}
\]
Note that this in effect gives a diagrammatic description 
of the entire category $\Rep(U_q(\sln[2]))$ of finite-dimensional representations, 
as this category can be recovered from $\FRep(U_q(\sln[2]))$ via 
idempotent completion.
Further, $\Rep(U_q(\sln[2]))$ is braided, 
and explicit diagrammatic formulae for the braiding in $\FRep(U_q(\sln[2]))$
recover the Kauffman bracket formulation \cite{Kauff} 
of the Jones polynomial \cite{Jones}.
This diagrammatic, generators and relations description of $\FRep(U_q(\sln[2]))$
thus serves as both the basis for applications of the Jones polynomial 
to classic problems in knot theory \cite{Kauff,Mur,Thistle1}, 
and for the development of Khovanov homology \cite{Kh1} 
and related constructions in categorical representation theory \cite{BSIII,BSIV}.

Pioneering work of Kuperberg \cite{Kup} poses the problem of 
finding a similar description for the category $\FRep(U_q(\mathfrak{g}))$ 
tensor-generated by the fundamental representations of the quantum group 
$U_q(\mathfrak{g})$, where $\mathfrak{g}$ is a complex simple Lie algebra. 
He answers this question in the case that $\mathfrak{g}$ is rank $2$, 
showing that $\FRep(U_q(\mathfrak{g}))$ admits a description as the
$\C(q)$-linear pivotal category freely 
generated by a finite set of morphisms depicted graphically as trivalent vertices, 
modulo a finite list of local relations. 
It follows that morphisms therein are given by $\C(q)$-linear combinations of webs, 
certain labeled trivalent graphs, modulo planar isotopy and local relations.
Again, diagrammatic formulae for the braiding on these categories give 
an explicit, local description for the $U_q(\mathfrak{g})$ link invariants, 
that again serve as the basis for further study of these invariants 
and their categorified analogues.

The problem of extending Kuperberg's results to higher rank 
resisted various attempts \cite{Kim, Mor}, 
although related descriptions were obtained for the corresponding link polynomials \cite{MOY,MOSpn}.
However, in breakthrough work \cite{CKM},
Cautis-Kamnitzer-Morrison found a web-based description of $\FRep(U_q(\sln))$ 
using a quantized version of skew Howe duality. 
This technique was adapted to give new diagrammatic 
descriptions of various other categories of representations of quantum groups 
\cite{RTub,QS2,Grant,TVW,BDKuj}. 
Nevertheless, it remains an open problem to give such a description of 
$\FRep(U_q(\mathfrak{g}))$ for simple complex $\mathfrak{g}$ of rank $\geq 3$ 
outside type $A$.
Complicating matters, the work of Sartori-Tubbenhauer \cite{TubSar} interestingly 
suggests that skew Howe duality 
cannot be used in a straightforward way to solve this problem in quantum types $BCD$, 
thus new ideas are required.

In this paper, 
we take the first steps towards solving this problem for the 
quantum group $U_q(\spn)$, focusing on the rank $3$ case.
Specifically, in Definition \ref{def:Web} below, 
we define the category $\Web(\spn[6])$ via generators and relations, 
and conjecture that it is equivalent to $\FRep(U_q(\spn[6]))$.
Although our conjecture remains open at the moment, 
we provide ample evidence for its validity: 
amongst other results, we prove that there is a full, essentially surjective 
functor $\Psi: \Web(\spn[6]) \to \FRep(U_q(\spn[6]))$ of ribbon categories, 
that all $\Hom$-spaces in $\Web(\spn[6])$ are finite-dimensional,
and that the endomorphism algebra of the monoidal unit in 
$\Web(\spn[6])$ is isomorphic to $\C(q)$. 
The latter result equivalently states that all closed webs in $\Web(\spn[6])$ 
can be evaluated to scalars, and our proof provides an explicit algorithm; 
this result pairs with diagrammatic formulae for the braiding to give an explicit, 
local, diagrammatic description of the (colored) $U_q(\spn[6])$ link invariant, 
\`{a} la the Kauffman bracket description of the Jones polynomial.
As such, we expect the category $\Web(\spn[6])$ to form the basis for an 
explicit construction of the $\spn$ link homologies, 
in the spirit of \cite{Kh3,BN2,MSV,QR1}.
In follow-up work \cite{BELR}, 
we plan to investigate faithfulness of the functor $\Psi$ using the techniques from \cite{EliasLL}, 
as well as higher rank and implications for link homology.

In Section \ref{sec:Background},
we recall the relevant background on quantum groups and web categories, 
and collect the main results of the present work. 
Section \ref{sec:FullAndBMW} is devoted to the definition and study of the functor 
$\Psi: \Web(\spn[6]) \to \FRep(U_q(\spn[6]))$. 
In Section \ref{sec:Closed}, we prove 
the finite-dimensionality of $\Hom$-spaces, and results about closed webs in the plane 
and the annulus; algebraically, these latter results correspond to the aforementioned result 
about the endomorphism algebra of the monoidal unit, 
and to a computation of the ``trace decategorification'' of $\Web(\spn[6])$.
To do so, we introduce a category of ladder-like webs as a technical tool. 
We believe this category will play an important role in the eventual resolution of our conjecture.
Finally, in Section \ref{sec:Links} we briefly discuss our approach to 
the quantum $\spn[6]$ link invariant, 
which gives a fresh perspective on the $n=3$ case of the invariants in \cite{MOSpn}

\begin{ack}
Some of the results in this paper also appear, in different form, 
in the PhD thesis \cite{LoganThesis} of the second named author, 
supervised by the first named author at UNC Chapel Hill. 
We thank UNC for excellent working conditions over the past several years.
Special thanks to Tomotada Ohtsuki, 
for extremely helpful correspondence regarding the choice of normalization 
for the generators of $\Web(\spn[6])$; 
his notes \cite{Ohtsuki} suggested the normalization used here. 
Also to Daniel Tubbenhauer, for his careful reading 
of a preliminary version of this manuscript, 
and for his very useful remarks and suggestions.
Both authors also thank Ben Elias and Jiuzu Hong for helpful discussion, 
and for their interest in this work.
This work was supported in part by Simons Collaboration Grant 523992: 
\emph{Research on knot invariants, representation theory, and categorification}.
\end{ack}

\section{Background and statement of results}\label{sec:Background}

In this section, we begin by reviewing quantum groups and their representation theory, 
along the way establishing notation and conventions. 
We then recall know results about webs in type $C_2$, 
and finally state the main results of the present work.

\subsection{Quantum groups and their representation theory} 
We begin with some background on quantum groups and their representations; 
almost all of this material is standard, and can be found e.g. in \cite{CP}, 
except when noted otherwise.

Let $\mathfrak{g}$ be a simple complex Lie algebra of rank $n$ with corresponding 
Cartan matrix $(a_{ij})$, and let $q$ be a (generic) indeterminant. 
Recall that the \emph{quantum group} $U_q(\mathfrak{g})$ associated to $\mathfrak{g}$ is the unital 
$\C(q)$-algebra generated by $X_i^+,X_i^-,K_i,K_i^{-1}$ ($1\leq i\leq n$) modulo the relations:
\begin{itemize}
\item $K_iK_i^{-1}=1$, $K_i^{-1}K_i=1$
\item $K_iK_j=K_jK_i$
\item $K_iX_j^{\pm}K_i^{-1}=q_i^{\pm a_{ij}}X_i^{\pm}$
\item $X_i^+X_j^- - X_j^-X_i^+ =\delta_{ij}\frac{K_i-K_i^{-1}}{q_i-q_i^{-1}}$
\item For $i\neq j$, 
$\sum_{m=0}^{1-a_{ij}}(-1)^m
{1-a_{ij} \brack m}_{q_i} (X_i^\pm)^{1-a_{ij}-m}X_j^\pm(X_i^\pm)^m = 0$.
\end{itemize}
Here, $q_i := q^{d_i}$ where $\{d_i\}_{i=1}^n$ are the relatively prime 
positive integers such that $(d_i a_{ij})$ is symmetric, 
and ${n \brack k}_{q_i} = \frac{[n]_{q_i}!}{[n-k]_{q_i}! [k]_{q_i}!}$ is the quantum binomial coefficient, 
defined using the quantum integers $[n]_{q_i} = \frac{q_i^n - q_i^{-n}}{q_i-q_i^{-1}}$.

Recall that $U_q(\mathfrak{g})$ is a Hopf algebra with comultiplication, 
counit, and antipode defined by
\begin{itemize}
\item $\Delta(K_i)=K_i\otimes K_i$, $\Delta(X_i^+)=X_i^+\otimes K_i + 1\otimes X_i^+$, $\Delta(X_i^-)=X_i^-\otimes1+K_i^{-1}\otimes X_i^-$
\item $\epsilon(K_i)=1$, $\epsilon(X_i^+)=\epsilon(X_i^-)=0$
\item $S(K_i)=K_i^{-1}$, $S(X_i^+)=-X_i^+K_i^{-1}$, $S(X_i^-)=-K_iX_i^-$
\end{itemize}
The quantum group can be understood as a Hopf algebra deformation of the enveloping algebra $U(\mathfrak{g})$. 
Further, there exists a universal $R$-matrix $R \in \widetilde{U_q(\mathfrak{g}) \otimes U_q(\mathfrak{g})}$ and 
ribbon element $\nu \in \tilde{U_q(\mathfrak{g})}$ so that the triple $(U_q(\mathfrak{g}),R,\nu)$ is a (topological)
ribbon Hopf algebra.

The representation theory of $U_q(\mathfrak{g})$ closely parallels that of $\mathfrak{g}$. 
Specifically, let $\Rep(U_q(\mathfrak{g}))$ be the category of (type I) finite-dimensional representations 
of $U_q(\mathfrak{g})$, then, as for $\mathfrak{g}$, $\Rep(U_q(\mathfrak{g}))$ is semi-simple, 
and there is a unique simple object $\Gamma_\lambda$ corresponding to each dominant weight $\lambda$ for $\mathfrak{g}$.
The decomposition of tensor products of irreducible representations into direct sums of irreducibles in 
$\Rep(U_q(\mathfrak{g}))$ (i.e. the fusion rules) are identical to those for $\mathfrak{g}$.
Further, the ribbon Hopf structure on $U_q(\mathfrak{g})$ endows $\Rep(U_q(\mathfrak{g}))$ 
with the structure of a ribbon (also called tortile) category. 
In particular, $\Rep(U_q(\mathfrak{g}))$ is a braided pivotal category.

To expand on the latter, the ribbon Hopf algebra structure produces a group-like element 
$g := \nu^{-1} u \in \tilde{U_q(\mathfrak{g})}$ (here $u$ is an element satisfying $\nu^2 = u S(u)$) 
and for $\Gamma \in \Rep(U_q(\mathfrak{g}))$ the pivotal isomorphism $\Gamma \cong \Gamma^{**}$ 
is given by $v \mapsto (f \mapsto f(gv))$ for $v \in \Gamma$ and $f \in \Gamma^*$.
This element $g$ is also relevant as it determines the \emph{quantum dimension} 
of $\Gamma \in \Rep(U_q(\mathfrak{g}))$, defined explicitly by
\[
\dim_q(\Gamma) = \Tr(g|_{\Gamma})
\]

\begin{rem}\label{rem:Ribbon}
Work of Snyder-Tingley classifies the choices of ribbon element $\nu$ for $U_q(\mathfrak{g})$; 
see \cite{ST} for complete details. 
To summarize, they show that there exists a ``half-ribbon'' element $X \in \tilde{U_q(\mathfrak{g})}$ so that 
all ribbon elements are given by $\nu = s(\phi) X^{-2}$. 
Here, $\phi$ is a character of the weight lattice modulo the root lattice with $|\phi| \leq 2$, 
and $s(\phi)$ is the corresponding element in $\tilde{U_q(\mathfrak{g})}$.

There is a standard choice of ribbon element $\nu = C$ for $U_q(\mathfrak{g})$, 
where $C \in \tilde{U_q(\mathfrak{g})}$ is the so-called quantum Casimir element, 
which acts on the irreducible representation $\Gamma_\lambda$ by $q^{-(\lambda,\lambda+2\rho)}$.
Here, $\rho$ is the half sum of the positive roots (or equivalently the sum of the fundamental weights), 
and $(\cdot , \cdot)$ is the standard symmetric bilinear form.
This ribbon element gives a pivotal structure on $\Rep(U_q(\mathfrak{g}))$ for which the 
Frobenius-Schur indicators agree with those for $\mathfrak{g}$. 

However, for the choice of ribbon element $\nu = X^{-2}$ (and corresponding pivotal structure), 
Snyder-Tingley show that all non-zero Frobenius-Schur indicators equal one. 
For $\mathfrak{g}$ of type $C$, this implies that this choice of $\nu$ differs from the standard choice. 
In this paper, \textbf{we work with this latter, non-standard, ribbon element}, 
hence the corresponding pivotal structure.
\end{rem}

We now further detail the above in the case $\mathfrak{g}=\spn[6]$ of most interest in this work. 
We have 
\[
(a_{ij})=\begin{pmatrix}
2 & -1 & 0 \\
-1 & 2 & -2 \\
0 & -1 & 2
\end{pmatrix}
\]
and $d_1=d_2=1$, $d_3=2$. 
Fixing elements $\{\epsilon_i\}_{i=1}^3$ in the dual of the Cartan that satisfy $(\epsilon_i,\epsilon_j) = \delta_{i,j}$, 
the simple roots are given by 
$\alpha_1 = \epsilon_1 - \epsilon_2, \alpha_2 = \epsilon_2 - \epsilon_3 , \alpha_3 = 2\epsilon_3$,
and the fundamental weights are
$\omega_1 = \epsilon_1, \omega_2 = \epsilon_1 + \epsilon_2, \omega_3 = \epsilon_1 + \epsilon_2 + \epsilon_3$. 
It follows that $\rho = 3\epsilon_1 + 2\epsilon_2 + \epsilon_3$.

All (type I) finite-dimensional irreducible representations of $U_q(\spn[6])$ take the form 
$\Gamma_{a,b,c} := \Gamma_{a\omega_1 + b\omega_2 + c\omega_3}$ for $a,b,c \in \N$, 
and we abbreviate
\[
V := \Gamma_{\omega_1} \;\; , \;\; W := \Gamma_{\omega_2} \;\; U := \Gamma_{\omega_3} 
\]
We record here that
\begin{equation}\label{eq:TensorDecomp}
\begin{aligned}
V \otimes V &\cong \Gamma_{2,0,0} \oplus W \oplus \C(q) \\
V \otimes W &\cong \Gamma_{1,1,0} \oplus U \oplus V \\
V\otimes U &\cong \Gamma_{1,0,1} \oplus W
\end{aligned}
\end{equation}
which in particular imply that every irreducible representation appears as a summand 
of $V^{\otimes n}$ for some $n$, and the parity of this $n$ is uniquely determined by the irreducible representation.

By Remark \ref{rem:Ribbon}, 
the non-standard ribbon element $X^{-2}$ acts on the irreducible representation 
$\Gamma_{\lambda}$ by $(-1)^n q^{-(\lambda,\lambda+2\rho)}$ where $\Gamma_\lambda \subset V^{\otimes n}$. 
The corresponding grouplike element $g$ acts on $\Gamma_{\lambda}$ by 
$(-1)^n K^{2 \rho} := (-1)^n K_1^6K_2^{10}K_3^6$, with $n$ as before 
(here, we use that $2\rho=6\alpha_1+10\alpha_2+6\alpha_3$).
Since $K_i$ acts on the $\mu$-weight space of $\Gamma_{\lambda}$ as $q^{(\alpha_i,\mu)}$,
this implies that
\begin{equation}\label{eq:qDim}
\begin{aligned}
\dim_q(V) &= -(q^6+q^4+q^2+q^{-2}+q^{-4}+q^{-6}) = -\frac{[3][8]}{[4]} \\
\dim_q(W) &= q^{10}+q^8+q^6+q^4+2q^2+2+2q^{-2}+q^{-4}+q^{-6}+q^{-8}+q^{-10} = \frac{[7][8]}{[4]} \\
\dim_q(W) &= -(q^{12}+q^8+q^6+2q^4+q^2+2+q^{-2}+2q^{-4}+q^{-6}+q^{-8}+q^{-12}) = -\frac{[6][7][8]}{[2][3][4]}
\end{aligned}
\end{equation}
Note that these formulae can also be obtained via the quantum Weyl dimension formula, 
see e.g. \cite{CP}.

\subsection{Web categories}

Recall from above that Kuperberg \cite{Kup} extended the Temperley-Lieb description of 
$\Rep(U_q(\sln[2]))$ to $U_q(\mathfrak{g})$ for $\mathfrak{g}$ of rank $2$. 
We now summarize his construction in the case $\mathfrak{g}= \spn[4]$, 
as it is the most relevant to the present work. 
To be precise, we actually work with a category equivalent to Kuperberg's, 
in which his trivalent vertex has been rescaled by $\sqrt{-1}$.

The type $C_2$ spider, denoted here by $\Web(\spn[4])$, 
is obtained from the $\C(q)$-linear ribbon category
pivotally generated by self-dual objects $\{1,2\}$ 
and the morphism
\[
\xy
(0,0)*{
\begin{tikzpicture}[scale =.5, smallnodes]
	\draw[1label] (0,0) node[below,color=black]{$1$} to [out=90,in=210] (.5,.75);
	\draw[1label] (1,0) node[below,color=black]{$1$} to [out=90,in=330] (.5,.75);
	\draw[2label] (.5,.75) to (.5,1.5) node[above,color=black]{$2$};
\end{tikzpicture}
};
\endxy 
\in \Hom_{\Web(\spn[4])}(1 \otimes 1, 2)
\]
by taking the quotient by the tensor-ideal generated by the following relations:
\[
\begin{gathered}
\xy
(0,0)*{
\begin{tikzpicture}[scale =.75, smallnodes]
	\draw[1label] (0,0) circle (.5);
\end{tikzpicture}
};
\endxy
= -\frac{[2][6]}{[3]} 
\;\; , \;\;
\xy
(0,0)*{
\begin{tikzpicture}[scale =.75, smallnodes]
	\draw[2label] (0,0) circle (.5);
\end{tikzpicture}
};
\endxy
= \frac{[5][6]}{[2][3]} 
\;\; , \;\;
\xy
(0,0)*{
\begin{tikzpicture}[scale=.2]
	\draw [1label] (0,-2.75) to [out=30,in=0] (0,.75);
	\draw [1label] (0,-2.75) to [out=150,in=180] (0,.75);
	\draw [2label] (0,-4.5) to (0,-2.75);
\end{tikzpicture}
};
\endxy
= 0
\;\; , \;\;
\xy
(0,0)*{
\begin{tikzpicture}[scale=.2]
	\draw [2label] (0,.75) to (0,2.5);
	\draw [1label] (0,-2.75) to [out=30,in=330] (0,.75);
	\draw [1label] (0,-2.75) to [out=150,in=210] (0,.75);
	\draw [2label] (0,-4.5) to (0,-2.75);
\end{tikzpicture}
};
\endxy
= [2]^2 \;
\xy
(0,0)*{
\begin{tikzpicture}[scale=.2]
	\draw [2label] (0,-4.5) to (0,2.5);
\end{tikzpicture}
};
\endxy \\
\xy
(0,0)*{
\begin{tikzpicture}[scale=.3]
	\draw[1label] (-1,0) to (0,1);
	\draw[1label] (1,0) to (0,1);
	\draw[1label] (0,2.5) to (-1,3.5);
	\draw[1label] (0,2.5) to (1,3.5);
	\draw[2label] (0,1) to (0,2.5);
\end{tikzpicture}
};
\endxy
-
\xy
(0,0)*{
\begin{tikzpicture}[scale=.3, rotate=90]
	\draw[1label] (-1,0) to (0,1);
	\draw[1label] (1,0) to (0,1);
	\draw[1label] (0,2.5) to (-1,3.5);
	\draw[1label] (0,2.5) to (1,3.5);
	\draw[2label] (0,1) to (0,2.5);
\end{tikzpicture}
};
\endxy
=
\xy
(0,0)*{
\begin{tikzpicture}[scale=.3]
	\draw[1label] (-1,0) to [out=45,in=-45] (-1,2);
	\draw[1label] (1,0) to [out=135,in=-135] (1,2);
\end{tikzpicture}
};
\endxy
-
\xy
(0,0)*{
\begin{tikzpicture}[scale=.3]
	\draw[1label] (-1,0) to [out=45,in=135] (1,0);
	\draw[1label] (-1,2) to [out=-45,in=-135] (1,2);
\end{tikzpicture}
};
\endxy
\;\; , \;\;
\xy
(0,0)*{
\begin{tikzpicture}[scale=.25]
	\draw[1label] (-1,0) to (1,0);
	\draw[1label] (-1,0) to (0,1.732);
	\draw[1label] (1,0) to (0,1.732);
	\draw[2label] (0,1.732) to (0,3.232);
	\draw[2label] (-2.3,-.75) to (-1,0);
	\draw[2label] (2.3,-.75) to (1,0);
\end{tikzpicture}
};
\endxy
= 0.
\end{gathered}
\]
As in these relations, we'll follow the conventions that we won't label the (co)domain 
in our morphisms, electing instead to color our web edges. 
For the duration, \textbf{black} denotes $1$-labeled edges 
and \textcolor{ourblue}{\textbf{blue}} denotes $2$-labeled edges. 
Further, as can be inferred from the above, 
we read all webs as mapping from bottom to top.

Denote the full subcategory of $\Rep(U_q(\spn[4]))$ tensor-generated by the fundamental representations 
by $\FRep(U_q(\spn[4]))$. Kuperberg's results then give the following.

\begin{thm}\label{thm:sp4}
There is an equivalence of pivotal categories $\Web(\spn[4]) \cong \FRep(U_q(\spn[4]))$.
\end{thm}

Indeed, this is the $\mathfrak{g}=\spn[4]$ case of his more-general result, 
while holds for all rank $2$ Lie algebras 
(i.e. additionally for $\sln[3]$ and $\mathfrak{g}_2$).
Further, he extends this result to an equivalence of ribbon categories by giving 
explicitly formulae for the braidings in the web categories.
As a consequence, Kuperberg obtains an explicit, 
construction of the colored $U_q(\mathfrak{g})$ link 
invariant in the spirit of the Kauffman bracket formulation of the Jones polynomial, 
which is thus amenable to combinatorial study.

\begin{rem}
In the next section, we'll add \textcolor{green}{\textbf{green}} (and later \textcolor{gray}{\textbf{gray}}) edges to our color palette. 
Those reading/printing in grayscale, or with color vision deficiency, can download the tex code, 
comment out the commands 
\begin{verbatim} 
2label/.style={very thick, ourblue}, 
glabel/.style={ultra thick, gray},
\end{verbatim} 
and comment in the commands
\begin{verbatim}
%2label/.style={double},
%glabel/.style={line width=3pt, gray},
\end{verbatim}
which will render all \textcolor{ourblue}{\textbf{blue}} edges as doubled, 
and thicken all \textcolor{gray}{\textbf{gray}} edges, to help distinguish them from the \textcolor{green}{\textbf{green}} edges.
\end{rem}

\subsection{Type $\spn[6]$ webs and statement of results}\label{sec:Theorems}

We now take the first steps towards giving a web-based description 
of $\Rep(U_q(\mathfrak{g}))$ in type $C$ and higher rank. 
We focus on the case $\mathfrak{g}=\spn[6]$, but believe our results should extend to higher rank as well. 
Specifically, we define a category $\Web(\spn[6])$, 
and conjecture that it gives a description of the category of quantum $\spn[6]$ representations
(see Conjecture \ref{conj} below).
At the time of writing, this remains a conjecture; however, we provide ample evidence for its validity in 
Theorems \ref{thmFull}, \ref{thmBMW}, \ref{thmEmpty}, \ref{thmTr}, and \ref{thmRes} below.

\begin{defn}\label{def:Web}
Let $\Web(\spn[6])$ be the strict pivotal $\C(q)$-linear category generated 
by the self-dual objects $\{1,2,3\}$ and with morphisms generated by:
\[
\xy
(0,0)*{
\begin{tikzpicture}[scale =.5, smallnodes]
	\draw[1label] (0,0) node[below]{$1$} to [out=90,in=210] (.5,.75);
	\draw[1label] (1,0) node[below]{$1$} to [out=90,in=330] (.5,.75);
	\draw[2label] (.5,.75) to (.5,1.5) node[above]{$2$};
\end{tikzpicture}
};
\endxy 
\in \Hom_{\Web(\spn[6])}(1 \otimes 1, 2)
\;\; , \;\;
\xy
(0,0)*{
\begin{tikzpicture}[scale =.5, smallnodes]
	\draw[1label] (0,0) node[below]{$1$} to [out=90,in=210] (.5,.75);
	\draw[2label] (1,0) node[below]{$2$} to [out=90,in=330] (.5,.75);
	\draw[3label] (.5,.75) to (.5,1.5) node[above]{$3$};
\end{tikzpicture}
};
\endxy 
\in \Hom_{\Web(\spn[6])}(1 \otimes 2, 3)
\]
modulo the tensor-ideal generated by the following relations:
\begin{equation}\label{eq:spRel}
\begin{gathered}
\xy
(0,0)*{
\begin{tikzpicture}[scale =.75, smallnodes]
	\draw[1label] (0,0) circle (.5);
\end{tikzpicture}
};
\endxy
= -\frac{[3][8]}{[4]} 
\;\; , \;\;
\xy
(0,0)*{
\begin{tikzpicture}[scale=.2]
	\draw [1label] (0,-2.75) to [out=30,in=0] (0,.75);
	\draw [1label] (0,-2.75) to [out=150,in=180] (0,.75);
	\draw [2label] (0,-4.5) to (0,-2.75);
\end{tikzpicture}
};
\endxy
= 0
\;\; , \;\;
\xy
(0,0)*{
\begin{tikzpicture}[scale=.2]
	\draw [2label] (0,.75) to (0,2.5);
	\draw [1label] (0,-2.75) to [out=30,in=330] (0,.75);
	\draw [1label] (0,-2.75) to [out=150,in=210] (0,.75);
	\draw [2label] (0,-4.5) to (0,-2.75);
\end{tikzpicture}
};
\endxy
= [2][3] \;
\xy
(0,0)*{
\begin{tikzpicture}[scale=.2]
	\draw [2label] (0,-4.5) to (0,2.5);
\end{tikzpicture}
};
\endxy 
\;\; , \;\;
\xy
(0,0)*{
\begin{tikzpicture}[scale=.2]
	\draw [3label] (0,.75) to (0,2.5);
	\draw [2label] (0,-2.75) to [out=30,in=330] (0,.75);
	\draw [1label] (0,-2.75) to [out=150,in=210] (0,.75);
	\draw [3label] (0,-4.5) to (0,-2.75);
\end{tikzpicture}
};
\endxy
= [2][3] \;
\xy
(0,0)*{
\begin{tikzpicture}[scale=.2]
	\draw [3label] (0,-4.5) to (0,2.5);
\end{tikzpicture}
};
\endxy 
\;\; , \;\;
\xy 
(0,0)*{
\begin{tikzpicture}[scale=.2, xscale=-1]
	\draw [1label] (-1,-1) to [out=90,in=210] (0,.75);
	\draw [1label] (1,-1) to [out=90,in=330] (0,.75);
	\draw [1label] (3,-1) to [out=90,in=330] (1,2.5);
	\draw [2label] (0,.75) to [out=90,in=210] (1,2.5);
	\draw [3label] (1,2.5) to (1,4.25);
\end{tikzpicture}
};
\endxy
=
\xy 
(0,0)*{
\begin{tikzpicture}[scale=.2]
	\draw [1label] (-1,-1) to [out=90,in=210] (0,.75);
	\draw [1label] (1,-1) to [out=90,in=330] (0,.75);
	\draw [1label] (3,-1) to [out=90,in=330] (1,2.5);
	\draw [2label] (0,.75) to [out=90,in=210] (1,2.5);
	\draw [3label] (1,2.5) to (1,4.25);
\end{tikzpicture}
};
\endxy
\\
\xy
(0,0)*{
\begin{tikzpicture}[scale=.25]
	\draw[2label] (-1,0) to (1,0);
	\draw[1label] (-1,0) to (0,1.732);
	\draw[1label] (1,0) to (0,1.732);
	\draw[2label] (0,1.732) to (0,3.232);
	\draw[3label] (-2.3,-.75) to (-1,0);
	\draw[3label] (2.3,-.75) to (1,0);
\end{tikzpicture}
};
\endxy
= 0
\;\; , \;\;
\xy
(0,0)*{
\begin{tikzpicture}[scale=.55,xscale=-1]
	\draw[2label] (0,-1) to (0,-.5);
	\draw[1label] (0,-.5) to [out=150,in=210] (0,.5);
	\draw[1label] (0,-.5) to (.5,-.25);
	\draw[1label] (.75,-1) to [out=90,in=330] (.5,-.25);
	\draw[2label] (.5,-.25) to  (.5,.25);
	\draw[1label] (.5,.25) to (0,.5);
	\draw[1label] (.5,.25) to [out=30,in=270] (.75,1);
	\draw[2label] (0,.5) to (0,1);
\end{tikzpicture}
};
\endxy
=[3]^2\;
\xy
(0,0)*{
\begin{tikzpicture}[scale=.55,xscale=-1]
	\draw[2label] (0,-1) to (0,1);
	\draw[1label] (.75,-1) to (.75,1);
\end{tikzpicture}
};
\endxy
-\frac{1}{[2]}\;
\xy
(0,0)*{
\begin{tikzpicture}[scale=.55,xscale=-1]
	\draw[2label] (-.25,-1) to [out=90,in=210] (.125,-.4);
	\draw[1label] (.5,-1) to [out=90,in=330] (.125,-.4);
	\draw[1label] (.125,-.4) to (.125,.4);
	\draw[2label] (.125,.4) to [out=150,in=270] (-.25,1);
	\draw[1label] (.125,.4) to [out=30,in=270] (.5,1);
\end{tikzpicture}
};
\endxy
+\frac{[3]^2}{[2]} \;
\xy
(0,0)*{
\begin{tikzpicture}[scale=.55,xscale=-1]
	\draw[2label] (-.25,-1) to [out=90,in=210] (.125,-.4);
	\draw[1label] (.5,-1) to [out=90,in=330] (.125,-.4);
	\draw[3label] (.125,-.4) to (.125,.4);
	\draw[2label] (.125,.4) to [out=150,in=270] (-.25,1);
	\draw[1label] (.125,.4) to [out=30,in=270] (.5,1);
\end{tikzpicture}
};
\endxy 
\\
\xy
(0,0)*{
\begin{tikzpicture}[scale=.3]
	\draw[1label] (-1,0) to (0,1);
	\draw[1label] (1,0) to (0,1);
	\draw[1label] (0,2.5) to (-1,3.5);
	\draw[1label] (0,2.5) to (1,3.5);
	\draw[2label] (0,1) to (0,2.5);
\end{tikzpicture}
};
\endxy
-
\xy
(0,0)*{
\begin{tikzpicture}[scale=.3, rotate=90]
	\draw[1label] (-1,0) to (0,1);
	\draw[1label] (1,0) to (0,1);
	\draw[1label] (0,2.5) to (-1,3.5);
	\draw[1label] (0,2.5) to (1,3.5);
	\draw[2label] (0,1) to (0,2.5);
\end{tikzpicture}
};
\endxy
=[2] \bigg(
\xy
(0,0)*{
\begin{tikzpicture}[scale=.3]
	\draw[1label] (-1,0) to [out=45,in=-45] (-1,2);
	\draw[1label] (1,0) to [out=135,in=-135] (1,2);
\end{tikzpicture}
};
\endxy
-
\xy
(0,0)*{
\begin{tikzpicture}[scale=.3]
	\draw[1label] (-1,0) to [out=45,in=135] (1,0);
	\draw[1label] (-1,2) to [out=-45,in=-135] (1,2);
\end{tikzpicture}
};
\endxy
\bigg)
\;\; , \;\;
\xy
(0,0)*{
\begin{tikzpicture}[scale=.3]
	\draw[1label] (-1,0) to (0,1);
	\draw[2label] (1,0) to (0,1);
	\draw[2label] (0,2.5) to (-1,3.5);
	\draw[1label] (0,2.5) to (1,3.5);
	\draw[1label] (0,1) to (0,2.5);
\end{tikzpicture}
};
\endxy
-
\xy
(0,0)*{
\begin{tikzpicture}[scale=.3,rotate=90,xscale=-1]
	\draw[1label] (-1,0) to (0,1);
	\draw[2label] (1,0) to (0,1);
	\draw[2label] (0,2.5) to (-1,3.5);
	\draw[1label] (0,2.5) to (1,3.5);
	\draw[1label] (0,1) to (0,2.5);
\end{tikzpicture}
};
\endxy
=[3] \Bigg(
\xy
(0,0)*{
\begin{tikzpicture}[scale=.3,rotate=90,xscale=-1]
	\draw[1label] (-1,0) to (0,1);
	\draw[2label] (1,0) to (0,1);
	\draw[2label] (0,2.5) to (-1,3.5);
	\draw[1label] (0,2.5) to (1,3.5);
	\draw[3label] (0,1) to (0,2.5);
\end{tikzpicture}
};
\endxy
-
\xy
(0,0)*{
\begin{tikzpicture}[scale=.3]
	\draw[1label] (-1,0) to (0,1);
	\draw[2label] (1,0) to (0,1);
	\draw[2label] (0,2.5) to (-1,3.5);
	\draw[1label] (0,2.5) to (1,3.5);
	\draw[3label] (0,1) to (0,2.5);
\end{tikzpicture}
};
\endxy
\; \Bigg)
\end{gathered}
\end{equation}
\end{defn}
In these relations, and for the duration, 
we extend our color conventions by denoting $3$-labeled web edges in 
\textcolor{green}{\textbf{green}}.
(Later, we also depict edges with arbitrary labels in thick \textcolor{gray}{\textbf{gray}}.)
Further, we employ the shorthand
\[
\xy
(0,0)*{
\begin{tikzpicture}[scale =.5, smallnodes]
	\draw[2label] (0,0) node[below]{$2$} to [out=90,in=210] (.5,.75);
	\draw[1label] (1,0) node[below]{$1$} to [out=90,in=330] (.5,.75);
	\draw[3label] (.5,.75) to (.5,1.5) node[above]{$3$};
\end{tikzpicture}
};
\endxy 
:=
\frac{1}{[2][3]}
\xy
(0,0)*{
\begin{tikzpicture}[scale=.25]
	\draw[1label] (-1,0) to (1,0);
	\draw[1label] (-1,0) to (0,1.732);
	\draw[2label] (1,0) to (0,1.732);
	\draw[3label] (0,1.732) to (0,3.232);
	\draw[2label] (-2.3,-.75) to (-1,0);
	\draw[1label] (2.3,-.75) to (1,0);
\end{tikzpicture}
};
\endxy
\]

\begin{rem}
The reader will notice that, with the exception of the final relation, 
these relations agree with the $n=3$ case of the relations in \cite{MOSpn}. 
Indeed, as the proof of Theorem \ref{thm:Psi} shows, 
up to rescaling our generating morphisms, 
these relations are the only ones possible between the constituent webs.
Our choice of normalization 
(which differs from that taken in a previous incarnation of this work \cite{LoganThesis})
was greatly influenced by the notes \cite{Ohtsuki}, 
which, in particular, correct a typo in \cite[Lemma 2.7]{MOSpn} that
makes their ``square relation'' agree with the above. 
\end{rem}

\begin{rem}
We record the following useful consequences of the above:
\[
\begin{gathered}
\xy
(0,0)*{
\begin{tikzpicture}[scale =.75, smallnodes]
	\draw[2label] (0,0) circle (.5);
\end{tikzpicture}
};
\endxy
= \frac{[7][8]}{[4]} 
\;\; , \;\;
\xy
(0,0)*{
\begin{tikzpicture}[scale =.75, smallnodes]
	\draw[3label] (0,0) circle (.5);
\end{tikzpicture}
};
\endxy
= -\frac{[6][7][8]}{[2][3][4]}
\\
\xy
(0,0)*{
\begin{tikzpicture}[scale=.2,anchorbase]
	\draw [1label] (0,.75) to (0,2.5);
	\draw [2label] (0,-2.75) to [out=30,in=330] (0,.75);
	\draw [1label] (0,-2.75) to [out=150,in=210] (0,.75);
	\draw [3label] (0,-4.5) to (0,-2.75);
\end{tikzpicture}
};
\endxy
=0
= 
\xy
(0,0)*{
\begin{tikzpicture}[scale=.2,xscale=-1,anchorbase]
	\draw [1label] (0,.75) to (0,2.5);
	\draw [2label] (0,-2.75) to [out=30,in=330] (0,.75);
	\draw [1label] (0,-2.75) to [out=150,in=210] (0,.75);
	\draw [3label] (0,-4.5) to (0,-2.75);
\end{tikzpicture}
};
\endxy
\;\; , \;\;
\xy
(0,0)*{
\begin{tikzpicture}[scale=.2]
	\draw [1label] (0,.75) to (0,2.5);
	\draw [2label] (0,-2.75) to [out=30,in=330] (0,.75);
	\draw [1label] (0,-2.75) to [out=150,in=210] (0,.75);
	\draw [1label] (0,-4.5) to (0,-2.75);
\end{tikzpicture}
};
\endxy
= -[2][7] \;
\xy
(0,0)*{
\begin{tikzpicture}[scale=.2]
	\draw [1label] (0,-4.5) to (0,2.5);
\end{tikzpicture}
};
\endxy
\;\; , \;\;
\xy
(0,0)*{
\begin{tikzpicture}[scale=.2,anchorbase]
	\draw [1label] (0,.75) to (0,2.5);
	\draw [3label] (0,-2.75) to [out=30,in=330] (0,.75);
	\draw [2label] (0,-2.75) to [out=150,in=210] (0,.75);
	\draw [1label] (0,-4.5) to (0,-2.75);
\end{tikzpicture}
};
\endxy
= \frac{[6][7]}{[3]} \;
\xy
(0,0)*{
\begin{tikzpicture}[scale=.2,anchorbase]
	\draw [1label] (0,-4.5) to (0,2.5);
\end{tikzpicture}
};
\endxy  
\;\; , \;\;
\xy
(0,0)*{
\begin{tikzpicture}[scale=.2]
	\draw [2label] (0,.75) to (0,2.5);
	\draw [3label] (0,-2.75) to [out=30,in=330] (0,.75);
	\draw [1label] (0,-2.75) to [out=150,in=210] (0,.75);
	\draw [2label] (0,-4.5) to (0,-2.75);
\end{tikzpicture}
};
\endxy
= -[6] \;
\xy
(0,0)*{
\begin{tikzpicture}[scale=.2]
	\draw [2label] (0,-4.5) to (0,2.5);
\end{tikzpicture}
};
\endxy 
\\
\xy
(0,0)*{
\begin{tikzpicture}[scale=.25]
	\draw[3label] (-1,0) to (1,0);
	\draw[1label] (-1,0) to (0,1.732);
	\draw[1label] (1,0) to (0,1.732);
	\draw[2label] (0,1.732) to (0,3.232);
	\draw[2label] (-2.3,-.75) to (-1,0);
	\draw[2label] (2.3,-.75) to (1,0);
\end{tikzpicture}
};
\endxy
=
\xy
(0,0)*{
\begin{tikzpicture}[scale=.25]
	\draw[1label] (-1,0) to (1,0);
	\draw[1label] (-1,0) to (0,1.732);
	\draw[1label] (1,0) to (0,1.732);
	\draw[2label] (0,1.732) to (0,3.232);
	\draw[2label] (-2.3,-.75) to (-1,0);
	\draw[2label] (2.3,-.75) to (1,0);
\end{tikzpicture}
};
\endxy
\;\; , \;\;
\xy
(0,0)*{
\begin{tikzpicture}[scale=.25]
	\draw[2label] (-1,0) to (1,0);
	\draw[1label] (-1,0) to (0,1.732);
	\draw[1label] (1,0) to (0,1.732);
	\draw[2label] (0,1.732) to (0,3.232);
	\draw[1label] (-2.3,-.75) to (-1,0);
	\draw[1label] (2.3,-.75) to (1,0);
\end{tikzpicture}
};
\endxy
= 
[4]
\xy
(0,0)*{
\begin{tikzpicture}[scale =.5, smallnodes]
	\draw[1label] (0,0)to [out=90,in=210] (.5,.75);
	\draw[1label] (1,0) to [out=90,in=330] (.5,.75);
	\draw[2label] (.5,.75) to (.5,1.5);
\end{tikzpicture}
};
\endxy 
\;\; , \;\;
\xy
(0,0)*{
\begin{tikzpicture}[scale=.25]
	\draw[3label] (-1,0) to (1,0);
	\draw[2label] (-1,0) to (0,1.732);
	\draw[1label] (1,0) to (0,1.732);
	\draw[3label] (0,1.732) to (0,3.232);
	\draw[1label] (-2.3,-.75) to (-1,0);
	\draw[2label] (2.3,-.75) to (1,0);
\end{tikzpicture}
};
\endxy
= [4] \;
\xy
(0,0)*{
\begin{tikzpicture}[scale =.5, smallnodes]
	\draw[1label] (0,0)to [out=90,in=210] (.5,.75);
	\draw[2label] (1,0) to [out=90,in=330] (.5,.75);
	\draw[3label] (.5,.75) to (.5,1.5);
\end{tikzpicture}
};
\endxy
\end{gathered}
\]
\end{rem}

We posit the following extension of Theorem \ref{thm:sp4}.

\begin{conj}\label{conj}
There is an equivalence of ribbon categories $\Web(\spn[6]) \cong \FRep(U_q(\spn[6]))$.
\end{conj}

While Conjecture \ref{conj} remains open at the moment, 
in the subsequent sections of this paper
we prove a number of results concerning $\Web(\spn[6])$ that serve as strong evidence for its validity. 
Explicitly, we show the following:

\begin{thm}\label{thmFull}
The category $\Web(\spn[6])$ is ribbon, and there is a 
full, essentially surjective, braided, monoidal functor
$\Psi: \Web(\spn[6]) \twoheadrightarrow \FRep(U_q(\spn[6]))$.
\end{thm}

Thus, the outstanding portion of Conjecture \ref{conj} is to show that the functor $\Psi$ is faithful, 
or, equivalently, to show 
that for objects $\vec{k},\vec{\ell}$ in $\Web(\spn[6])$ we have
\begin{equation}\label{eq:HomDim}
\dim \big(\Hom_{\Web(\spn[6])}(\vec{k},\vec{\ell}) \big) \leq 
\dim \big(\Hom_{\spn[6]}(\Psi(\vec{k}) , \Psi(\vec{\ell})) \big).
\end{equation}

Denoting by $V$ the standard representation of $U_q(\spn[6])$, 
it is well known that 
$\End_{\spn[6]}(V^{\otimes k})$
is a quotient of the $k$-strand BMW algebra $\mathrm{BMW}_k(r,z)$ \cite{Murakami,BW,Hu}, 
after specializing parameters (see Section \ref{sec:FullAndBMW} below). 
We show the analogue of this result for $\Web(\spn[6])$.

\begin{thm}\label{thmBMW}
There is a homomorphism 
$\BMW_k(-q^{7},q-q^{-1}) \to \End_{\Web(\spn[6])}(1^{\otimes k})$.
\end{thm}

Since all $\Hom$-spaces in $\Web(\spn[6])$ can be identified with subspaces of 
$\End_{\Web(\spn[6])}(1^{\otimes k})$ for some $k$, 
we suspect that equation \eqref{eq:HomDim}, and thus Conjecture \ref{conj}, 
can be deduced by showing that the map in Theorem \ref{thmBMW} is surjective, 
and comparing its kernel to that of 
$\mathrm{BMW}_k \twoheadrightarrow \End_{\spn[6]}(V^{\otimes k})$.

Although we do not prove \eqref{eq:HomDim} in general 
(this will be pursued further in \cite{BELR}), 
we show that it holds for the endomorphism algebra of the monoidal unit, 
and that the left-hand side is indeed finite 
(which is not guaranteed a priori):

\begin{thm}\label{thmEmpty}
$\End_{\Web(\spn[6])}(\varnothing) \cong \C(q)$. 
Further, for all objects $\vec{k}, \vec{\ell}$ in $\Web(\spn[6])$, 
the $\C(q)$-vector space
$\Hom_{\Web(\spn[6])}(\vec{k}, \vec{\ell})$ is finite-dimensional.
\end{thm}

This first statement is equivalent to the fact that every closed web in 
$\Web(\spn[6])$ can be evaluated to an element of $\C(q)$. 
Our proof provides an explicit algorithm for this evaluation.

Another approach to the resolution of Conjecture \ref{conj} is via work of 
Tuba and Wenzl \cite{TubaWenzl}. 
Therein, they classify the semisimple braided monoidal categories whose Grothendieck 
ring agree with that of $\Rep(U_q(\spn))$, in terms of the eigenvalues of the braiding. 
The Grothendieck ring of a semisimple linear category $\mathcal{C}$ is known to agree with its
categorical trace $\Tr(\mathcal{C})$ (also known as its $0^{th}$ Hochschild-Mitchell homology). 
The latter is amenable to diagrammatic methods, 
as it has an interpretation as the corresponding skein module of the annulus.
To this end, we show the following.

\begin{thm}\label{thmTr}
The functor $\Psi$ induces an isomorphism $\Tr(\Web(\spn[6])) \cong \Tr(\Rep(U_q(\spn[6])))$ 
of commutative $\C(q)$-algebras.
\end{thm}

Using the aformentioned results of Tuba-Wenzl, Conjecture \ref{conj} would follow 
from showing that all endomorphism algebras in $\Web(\spn[6])$ are semisimple $\C(q)$-algebras.

As a final piece of supporting evidence for Conjecture \ref{conj}, 
we note that restricting along the inclusion $U_q(\spn[4]) \hookrightarrow U_q(\spn[6])$ 
gives a functor $\Rep(U_q(\spn[6])) \to \Rep(U_q(\spn[4]))$. 
An analogue for $\Web(\spn)$ is established in \cite{LoganThesis} that we state, without proof, here.
Denote the additive closure of a category $\mathcal{C}$ by $\mathcal{C}^\oplus$, 
then we have the following.

\begin{thm}\label{thmRes}
There is a functor $\Web(\spn[6]) \to \Web(\spn[4])^\oplus$.
\end{thm}

\begin{rem}
Our interest in a graphical description for $\Rep(U_q(\spn))$ is partially motivated 
by considerations in low-dimensional topology. 
Indeed, given any simple Lie algebra $\mathfrak{g}$, 
there is a corresponding Reshetikhin-Turaev invariant $P_\mathfrak{g}(\mathcal{L}) \in \Z[q,q^{-1}]$ 
of framed links $\mathcal{L} \subset S^3$, 
defined using the ribbon structure on $\Rep(U_q(\mathfrak{g}))$; see \cite{RT1}. 
A graphical description for $\Rep(U_q(\mathfrak{g}))$ thus provides a graphical description 
for the corresponding link polynomials. 
In type $A$, such descriptions underly generators-and-relations constructions of Khovanov 
and Khovanov-Rozansky link homology via cobordisms and foams \cite{Kh3,BN2,MSV,QR1}. 
In turn, deformations of these foam/cobordism $2$-categories \cite{BNM,RW} 
play a crucial role in the definition and study of 
associated concordance invariants and slice-genus lower bounds \cite{Ras2,Wu3,Lobb,LewLobb}.
As such, we view a graphical description of $\Rep(U_q(\mathfrak{g}))$ as the first step in 
a long program to provide topological applications for link homology theories outside type $A$. 
Such homological invariants of links are known to exist by work of Webster \cite{Webster}.

To this end, we note that 
Theorem \ref{thmEmpty} (or Theorem \ref{thmTr}) implies that
$\Web(\spn[6])$ indeed gives a generators-and-relations, graphical setting for 
the $\spn[6]$ quantum link polynomial. 
To wit, given a link diagram 
(or braid closure, if using Theorem \ref{thmTr}), 
explicit formulae assign to it a $\C(q)$-linear combination of webs, 
which can be evaluated to an element in $\C(q)$ using the aforementioned theorems.
As usual, a skein-theoretic argument shows that this invariant 
actually takes values in $\Z[q,q^{-1}] \subset \C(q)$. 
As such, we believe that a manifestly integral version of Theorem \ref{thmEmpty} should hold.
See Section \ref{sec:Links} for full details.
\end{rem}

\section{The functor from $\Web(\spn[6])$ to $\FRep(U_q(\spn[6]))$}\label{sec:FullAndBMW}

In this section, we define a functor $\Psi:\Web(\spn[6]) \to \FRep(U_q(\spn[6]))$ 
and prove that it is full and essentially surjective.
Along the way, we endow $\Web(\spn[6])$ with the structure of a ribbon category that is compatible with 
the ribbon structure on $\FRep(U_q(\spn[6]))$, 
and establish a relation between $\Web(\spn[6])$ and the Birman-Murakami-Wenzl (BMW) algebra.
As such, we deduce Theorems \ref{thmFull}  and \ref{thmBMW} above.

We begin by defining the functor $\Psi$.
A direct approach -- e.g. by presenting 
$\Web(\spn[6])$ by generators and relations \emph{as a monoidal category}, 
explicitly specifying the images of generating morphisms,
and checking monoidally generating relations --
is possible, but tedious (see \cite{LoganThesis} for the relevant formulae).
To avoid the numerous computations such an approach would require, we instead proceed via an indirect, 
more-conceptual approach that exploits the ribbon structure on $\Rep(U_q(\spn[6]))$.

\begin{thm}\label{thm:Psi}
There exists an essentially surjective functor $\Psi: \Web(\spn[6]) \to \FRep(U_q(\spn[6]))$.
\end{thm}

\begin{proof}
Let $\FRep^{\pm}(U_q(\spn[6]))$ denote the full subcategory of $\Rep(U_q(\spn[6]))$ 
monoidally generated by the fundamental representations and their duals. 
By the coherence theorem for pivotal categories (see e.g. \cite{NS} or \cite{BarWes}),
this is equivalent, as a pivotal category, to a strict pivotal category. 
Further, since we work with the non-standard pivotal structure from Remark \ref{rem:Ribbon},
all non-zero Frobenius-Schur indicators are $+1$. 
This implies that the full subcategory $\FRep(U_q(\spn[6]))$ inherits the structure of a strict pivotal 
category with self duality structure, in the sense of \cite{Sel2}. 
The latter consists of a choice of coherent isomorphism $X \cong X^*$ for all 
objects $X$ in $\FRep(U_q(\spn[6]))$ that is compatible with the pivotal structure.

Next, equation \eqref{eq:TensorDecomp} implies that both 
$\Hom_{\spn[6]}(V \otimes V, W)$ and $\Hom_{\spn[6]}(V \otimes W, U)$
are $1$-dimensional. 
Choosing a non-zero morphism in each of these $\Hom$-spaces, 
it follows (e.g. from the results in \cite{Sel1}) that there exists
a pivotal functor $\Psi$ from the strict pivotal category freely generated by 
self-dual objects $\{1,2,3\}$ and the morphisms
\[
\xy
(0,0)*{
\begin{tikzpicture}[scale =.5, smallnodes]
	\draw[1label] (0,0) node[below]{$1$} to [out=90,in=210] (.5,.75);
	\draw[1label] (1,0) node[below]{$1$} to [out=90,in=330] (.5,.75);
	\draw[2label] (.5,.75) to (.5,1.5) node[above]{$2$};
\end{tikzpicture}
};
\endxy 
\text{ and }
\xy
(0,0)*{
\begin{tikzpicture}[scale =.5, smallnodes]
	\draw[1label] (0,0) node[below]{$1$} to [out=90,in=210] (.5,.75);
	\draw[2label] (1,0) node[below]{$2$} to [out=90,in=330] (.5,.75);
	\draw[3label] (.5,.75) to (.5,1.5) node[above]{$3$};
\end{tikzpicture}
};
\endxy
\]
to $\FRep(U_q(\spn[6]))$
that sends $1 \mapsto V$, $2 \mapsto W$, and $3 \mapsto U$, 
and sends these webs to the chosen morphisms.
By construction, $\Psi$ is essentially surjective.

It remains to show that $\Psi$ descends to the quotient 
obtained by imposing the relations in \eqref{eq:spRel}.
First, note that
\[
\Psi \left(
\xy
(0,0)*{
\begin{tikzpicture}[scale=.2]
	\draw [1label] (0,-2.75) to [out=30,in=0] (0,-.25);
	\draw [1label] (0,-2.75) to [out=150,in=180] (0,-.25);
	\draw [2label] (0,-4.5) to (0,-2.75);
\end{tikzpicture}
};
\endxy
\right)
= 0
\;\; \text{and} \;\;
\Psi \left( \;
\xy
(0,0)*{
\begin{tikzpicture}[scale=.2]
	\draw[2label] (-1,0) to (1,0);
	\draw[1label] (-1,0) to (0,1.732);
	\draw[1label] (1,0) to (0,1.732);
	\draw[2label] (0,1.732) to (0,3.232);
	\draw[3label] (-2.3,-.75) to (-1,0);
	\draw[3label] (2.3,-.75) to (1,0);
\end{tikzpicture}
};
\endxy
\; \right)
= 0
\]
since $\Hom_{\spn[6]}(W,\C(q))$ and $\Hom_{\spn[6]}(U\otimes U,W)$
are both trivial, and that 
\[
\Psi \left(
\xy
(0,0)*{
\begin{tikzpicture}[scale=.2]
	\draw [2label] (0,.75) to (0,2.5);
	\draw [1label] (0,-2.75) to [out=30,in=330] (0,.75);
	\draw [1label] (0,-2.75) to [out=150,in=210] (0,.75);
	\draw [2label] (0,-4.5) to (0,-2.75);
\end{tikzpicture}
};
\endxy
\right)
= \delta \cdot \id_W
\]
for some $\delta \in \C(q)$ since $\Hom_{\spn[6]}(W,W) \cong \C(q)$.
Further, equation \eqref{eq:qDim} implies that 
\begin{equation}\label{eq:circles}
\Psi \left(
\xy
(0,0)*{
\begin{tikzpicture}[scale =.75, smallnodes]
	\draw[1label] (0,0) circle (.5);
\end{tikzpicture}
};
\endxy
\right)
= -\frac{[3][8]}{[4]}
\;\; , \;\;
\Psi \left(
\xy
(0,0)*{
\begin{tikzpicture}[scale =.75, smallnodes]
	\draw[2label] (0,0) circle (.5);
\end{tikzpicture}
};
\endxy
\right)
= \frac{[7][8]}{[4]} 
\;\; , \;\;
\Psi \left(
\xy
(0,0)*{
\begin{tikzpicture}[scale =.75, smallnodes]
	\draw[3label] (0,0) circle (.5);
\end{tikzpicture}
};
\endxy
\right)
= -\frac{[6][7][8]}{[2][3][4]} 
\end{equation}

Equation \eqref{eq:TensorDecomp} implies that the images under $\Psi$ of
\begin{equation}\label{eq:VVbasis}
\xy
(0,0)*{
\begin{tikzpicture}[scale =.35, smallnodes]
	\draw[1label] (0,0) to (0,2.25);
	\draw[1label] (1,0) to (1,2.25);
\end{tikzpicture}
};
\endxy
\;\; , \;\;
\xy
(0,0)*{
\begin{tikzpicture}[scale =.35, smallnodes]
	\draw[1label] (0,0) to [out=90,in=180] (.5,.75);
	\draw[1label] (1,0) to [out=90,in=0] (.5,.75);
	\draw[1label] (.5,1.5) to [out=180,in=270] (0,2.25);
	\draw[1label] (.5,1.5) to [out=0,in=270] (1,2.25);
\end{tikzpicture}
};
\endxy
\;\; , \;\;
\xy
(0,0)*{
\begin{tikzpicture}[scale =.35, smallnodes]
	\draw[1label] (0,0) to [out=90,in=210] (.5,.75);
	\draw[1label] (1,0) to [out=90,in=330] (.5,.75);
	\draw[2label] (.5,.75) to (.5,1.5);
	\draw[1label] (.5,1.5) to [out=150,in=270] (0,2.25);
	\draw[1label] (.5,1.5) to [out=30,in=270] (1,2.25);
\end{tikzpicture}
};
\endxy
\end{equation}
are linearly independent, hence give a basis for $\End_{\spn[6]}(V\otimes V)$. 
This, together with the self-duality structure, imply that there exists 
$\alpha \in \C(q)$ and $\gamma \in \{-1,1\}$ so that
\begin{equation}\label{eq:Switching}
\Psi \Big(
\xy
(0,0)*{
\begin{tikzpicture}[scale=.2, rotate=90]
	\draw[1label] (-1,0) to (0,1);
	\draw[1label] (1,0) to (0,1);
	\draw[1label] (0,2.5) to (-1,3.5);
	\draw[1label] (0,2.5) to (1,3.5);
	\draw[2label] (0,1) to (0,2.5);
\end{tikzpicture}
};
\endxy
\Big)
+ \gamma \cdot
\Psi \left( \;
\xy
(0,0)*{
\begin{tikzpicture}[scale =.35, smallnodes]
	\draw[1label] (0,0) to [out=90,in=210] (.5,.75);
	\draw[1label] (1,0) to [out=90,in=330] (.5,.75);
	\draw[2label] (.5,.75) to (.5,1.5);
	\draw[1label] (.5,1.5) to [out=150,in=270] (0,2.25);
	\draw[1label] (.5,1.5) to [out=30,in=270] (1,2.25);
\end{tikzpicture}
};
\endxy
\; \right)
=
\alpha \cdot
\left(
\Psi \left( \;
\xy
(0,0)*{
\begin{tikzpicture}[scale =.35, smallnodes]
	\draw[1label] (0,0) to (0,2.25);
	\draw[1label] (1,0) to (1,2.25);
\end{tikzpicture}
};
\endxy
\; \right)
+ \gamma \cdot
\Psi \left( \;
\xy
(0,0)*{
\begin{tikzpicture}[scale =.35, smallnodes]
	\draw[1label] (0,0) to [out=90,in=180] (.5,.75);
	\draw[1label] (1,0) to [out=90,in=0] (.5,.75);
	\draw[1label] (.5,1.5) to [out=180,in=270] (0,2.25);
	\draw[1label] (.5,1.5) to [out=0,in=270] (1,2.25);
\end{tikzpicture}
};
\endxy
\; \right)
\right).
\end{equation}
Taking the closure of this relation, 
and applying the above relations, 
we see that
\[
\gamma \delta \frac{[7][8]}{[4]} 
= \alpha \frac{[3][8]}{[4]} \left(\frac{[3][8]}{[4]} - \gamma \right)
\]
so
\begin{equation}\label{eq:SwitchCoeff}
\delta[7]=[3]\alpha\left( \gamma([7]-1)-1\right).
\end{equation}
Here, we use that $\frac{[3][8]}{[4]}=[7]-1$.

Next, recall that $\FRep(U_q(\spn[6]))$ inherits a braiding $\beta$ 
from $\Rep(U_q(\spn[6]))$. 
It follows that
\[
\beta_{V,V} = 
\kappa \cdot \Psi \left( \;
\xy
(0,0)*{
\begin{tikzpicture}[scale =.35, smallnodes]
	\draw[1label] (0,0) to (0,2.25);
	\draw[1label] (1,0) to (1,2.25);
\end{tikzpicture}
};
\endxy
\; \right)
+
\lambda \cdot \Psi \left( \;
\xy
(0,0)*{
\begin{tikzpicture}[scale =.35, smallnodes]
	\draw[1label] (0,0) to [out=90,in=180] (.5,.75);
	\draw[1label] (1,0) to [out=90,in=0] (.5,.75);
	\draw[1label] (.5,1.5) to [out=180,in=270] (0,2.25);
	\draw[1label] (.5,1.5) to [out=0,in=270] (1,2.25);
\end{tikzpicture}
};
\endxy
\; \right)
+
\mu \cdot \Psi \left( \;
\xy
(0,0)*{
\begin{tikzpicture}[scale =.35, smallnodes]
	\draw[1label] (0,0) to [out=90,in=210] (.5,.75);
	\draw[1label] (1,0) to [out=90,in=330] (.5,.75);
	\draw[2label] (.5,.75) to (.5,1.5);
	\draw[1label] (.5,1.5) to [out=150,in=270] (0,2.25);
	\draw[1label] (.5,1.5) to [out=30,in=270] (1,2.25);
\end{tikzpicture}
};
\endxy
\; \right)
\]
for some $\kappa, \lambda,\mu \in \C(q)$. 
The self-duality structure then implies that
\[
\begin{aligned}
\beta_{V,V}^{-1} 
&= 
\kappa \cdot \Psi \left( \;
\xy
(0,0)*{
\begin{tikzpicture}[scale =.35, smallnodes]
	\draw[1label] (0,0) to [out=90,in=180] (.5,.75);
	\draw[1label] (1,0) to [out=90,in=0] (.5,.75);
	\draw[1label] (.5,1.5) to [out=180,in=270] (0,2.25);
	\draw[1label] (.5,1.5) to [out=0,in=270] (1,2.25);
\end{tikzpicture}
};
\endxy
\; \right)
+
\lambda \cdot \Psi \left( \;
\xy
(0,0)*{
\begin{tikzpicture}[scale =.35, smallnodes]
	\draw[1label] (0,0) to (0,2.25);
	\draw[1label] (1,0) to (1,2.25);
\end{tikzpicture}
};
\endxy
\; \right)
+
\mu \cdot \Psi \left( \;
\xy
(0,0)*{
\begin{tikzpicture}[scale=.2, rotate=90]
	\draw[1label] (-1,0) to (0,1);
	\draw[1label] (1,0) to (0,1);
	\draw[1label] (0,2.5) to (-1,3.5);
	\draw[1label] (0,2.5) to (1,3.5);
	\draw[2label] (0,1) to (0,2.5);
\end{tikzpicture}
};
\endxy
\; \right) \\
&= 
(\kappa+\mu \alpha \gamma) \cdot \Psi \left( \;
\xy
(0,0)*{
\begin{tikzpicture}[scale =.35, smallnodes]
	\draw[1label] (0,0) to [out=90,in=180] (.5,.75);
	\draw[1label] (1,0) to [out=90,in=0] (.5,.75);
	\draw[1label] (.5,1.5) to [out=180,in=270] (0,2.25);
	\draw[1label] (.5,1.5) to [out=0,in=270] (1,2.25);
\end{tikzpicture}
};
\endxy
\; \right)
+
(\lambda+\mu \alpha) \cdot \Psi \left( \;
\xy
(0,0)*{
\begin{tikzpicture}[scale =.35, smallnodes]
	\draw[1label] (0,0) to (0,2.25);
	\draw[1label] (1,0) to (1,2.25);
\end{tikzpicture}
};
\endxy
\; \right)
-
\mu\gamma \cdot \Psi \left( \;
\xy
(0,0)*{
\begin{tikzpicture}[scale =.35, smallnodes]
	\draw[1label] (0,0) to [out=90,in=210] (.5,.75);
	\draw[1label] (1,0) to [out=90,in=330] (.5,.75);
	\draw[2label] (.5,.75) to (.5,1.5);
	\draw[1label] (.5,1.5) to [out=150,in=270] (0,2.25);
	\draw[1label] (.5,1.5) to [out=30,in=270] (1,2.25);
\end{tikzpicture}
};
\endxy
\; \right).
\end{aligned}
\]

The equality $\beta_{V,V}^{-1}  \beta_{V,V} = \id_{V\otimes V}$, 
together with linear independence of the images of \eqref{eq:VVbasis}, 
implies that
\begin{equation}\label{eq:R2}
\begin{aligned}
\kappa(\lambda + \mu\alpha) &= 1 \\
\frac{\kappa}{\lambda} + \frac{\lambda+\mu \alpha}{\kappa+\mu \alpha \gamma}
&= \frac{[3][8]}{[4]} \\
\lambda + \mu\alpha &= \gamma(\kappa + \mu\delta).
\end{aligned}
\end{equation}
In particular, we deduce that
\begin{equation}\label{eq:Kappa1}
\kappa\gamma(\kappa + \mu\delta)=1.
\end{equation}

Setting 
\[
\xy
(0,0)*{
\begin{tikzpicture}[scale=.3]
	\draw[1label] (1,-1) to (-1,1);
	\draw[overcross] (-1,-1) to (1,1);
	\draw[1label] (-1,-1) to (1,1);
\end{tikzpicture}
};
\endxy = 
\kappa \;
\xy
(0,0)*{
\begin{tikzpicture}[scale =.35, smallnodes]
	\draw[1label] (0,0) to (0,2.25);
	\draw[1label] (1,0) to (1,2.25);
\end{tikzpicture}
};
\endxy
+
\lambda \;
\xy
(0,0)*{
\begin{tikzpicture}[scale =.35, smallnodes]
	\draw[1label] (0,0) to [out=90,in=180] (.5,.75);
	\draw[1label] (1,0) to [out=90,in=0] (.5,.75);
	\draw[1label] (.5,1.5) to [out=180,in=270] (0,2.25);
	\draw[1label] (.5,1.5) to [out=0,in=270] (1,2.25);
\end{tikzpicture}
};
\endxy
+
\mu \;
\xy
(0,0)*{
\begin{tikzpicture}[scale =.35, smallnodes]
	\draw[1label] (0,0) to [out=90,in=210] (.5,.75);
	\draw[1label] (1,0) to [out=90,in=330] (.5,.75);
	\draw[2label] (.5,.75) to (.5,1.5);
	\draw[1label] (.5,1.5) to [out=150,in=270] (0,2.25);
	\draw[1label] (.5,1.5) to [out=30,in=270] (1,2.25);
\end{tikzpicture}
};
\endxy
\]
we have that\footnote{We follow the convention, as in \cite{ST}, 
that the ribbon element $\nu$ acts as the \emph{negative} full twist.}
\[
\Big( \kappa - \lambda([7]-1) \Big) \id_V = 
\Psi\left(
\xy
(0,0)*{
\begin{tikzpicture}[scale=.5,yscale=-1]
	\draw[1label] (.8,-.4) to [out=180,in=270] (0,1);
	\draw[overcross] (0,-1) to [out=90,in=180] (.8,.4);
	\draw[1label] (0,-1) to [out=90,in=180] (.8,.4);
	\draw[1label] (.8,.4) to [out=0,in=90] (1.1,0) to [out=270,in=0] (.8,-.4);
\end{tikzpicture}
};
\endxy \right)
= \nu|_V \id_V = -q^{-(\omega_1,\omega_1+2\rho)} \id_V = -q^{-7} \id_V
\]
thus 
\begin{equation}\label{eq:Lambda}
\lambda = \frac{\kappa + q^{-7}}{[7]-1}
\end{equation}
Similarly, we compute that 
\[
\nu|_W^{-1} \cdot
\Psi \left( \;
\xy
(0,0)*{
\begin{tikzpicture}[scale =.35, smallnodes]
	\draw[2label] (.5,.75) to (.5,1.5);
	\draw[1label] (.5,1.5) to [out=150,in=270] (0,2.25);
	\draw[1label] (.5,1.5) to [out=30,in=270] (1,2.25);
\end{tikzpicture}
};
\endxy
\; \right)
=
\Psi\left( \;
\xy
(0,0)*{
\begin{tikzpicture}[scale=.35]
	\draw[2label] (.8,-.4) to [out=180,in=270] (0,1);
	\draw[overcross] (0,-1) to [out=90,in=180] (.8,.4);
	\draw[2label] (0,-1) to [out=90,in=180] (.8,.4);
	\draw[2label] (.8,.4) to [out=0,in=90] (1.1,0) to [out=270,in=0] (.8,-.4);
	\draw[1label] (0,1) to [out=150,in=270] (-.5,1.75);
	\draw[1label] (0,1) to [out=30,in=270] (.5,1.75);
\end{tikzpicture}
};
\endxy \right)
=
\Psi \left( \;
\xy
(0,0)*{
\begin{tikzpicture}[scale =.35, smallnodes]
	\draw[2label] (.5,.75) to (.5,1.5);
	\draw[1label] (.5,1.5) to [out=30,in=270] (1,2.25) to [out=90,in=270] (0,3.75);
	\draw[overcross] (0,2.25) to [out=90,in=270] (1,3.75);
	\draw[1label] (.5,1.5) to [out=150,in=270] (0,2.25) to [out=90,in=270] (1,3.75) to [out=90,in=270] (0,5.25);
	\draw[overcross] (0,3.75) to [out=90,in=270] (1,5.25);
	\draw[1label] (0,3.75) to [out=90,in=270] (1,5.25);
	\draw[1label] (1.5,6) to [out=0,in=90] (1.75,5.75) 
		to [out=270,in=0] (1.5,5.5) to [out=180,in=270] (1,6.5);
	\draw[thinovercross] (1,5.25) to [out=90,in=180] (1.5,6);
	\draw[1label] (1,5.25) to [out=90,in=180] (1.5,6);
	\draw[1label] (.5,6) to [out=0,in=90] (.75,5.75) 
		to [out=270,in=0] (.5,5.5) to [out=180,in=270] (0,6.5);
	\draw[thinovercross] (0,5.25) to [out=90,in=180] (0.5,6);
	\draw[1label] (0,5.25) to [out=90,in=180] (.5,6);
\end{tikzpicture}
};
\endxy
\; \right) =
\nu|_V^{-2} (\kappa+\mu\delta)^2 \cdot
\Psi \left( \;
\xy
(0,0)*{
\begin{tikzpicture}[scale =.35, smallnodes]
	\draw[2label] (.5,.75) to (.5,1.5);
	\draw[1label] (.5,1.5) to [out=150,in=270] (0,2.25);
	\draw[1label] (.5,1.5) to [out=30,in=270] (1,2.25);
\end{tikzpicture}
};
\endxy
\; \right)
\]
and since $\nu|_W = q^{-12}$ this gives
\begin{equation}\label{eq:Kappa2}
\kappa+\mu\delta = \pm q^{-1}.
\end{equation}
Equations \eqref{eq:Kappa1} and \eqref{eq:Kappa2} then imply that
\[
\kappa = \pm q.
\]

We further claim that we have $\kappa \neq -q$. 
Indeed, if $\kappa = -q$, then \eqref{eq:Lambda} implies that $\lambda|_{q=1} = 0$, 
$\mu\alpha|_{q=1} = -1$ by \eqref{eq:R2}.
Equation \eqref{eq:R2} also gives that $\mu\delta|_{q=1} = 1-\gamma$. 
Multiplying \eqref{eq:SwitchCoeff} by $\mu$ and evaluating at $q=1$ then gives that 
\[
7(1-\gamma) = -3(6\gamma-1)
\]
which implies that $\gamma=-4/11$, contradicting that $\gamma = \pm1$.

Thus, we have that $\kappa=q$. 
Equation \eqref{eq:Lambda} then implies that 
\[
\lambda = q^{-3} \frac{(q^4+q^{-4})[4]}{[3][8]} = \frac{q^{-3}}{[3]}
\]
so
\[
\mu\alpha = \kappa^{-1} - \lambda = q^{-1} - \frac{q^{-3}}{[3]} = \frac{[2]}{[3]}.
\]
Equation \eqref{eq:R2} this implies that 
$\mu\delta = \gamma q^{-1} - q$, so in particular $\mu\delta|_{q=1} = \gamma -1$. 
Multiplying \eqref{eq:SwitchCoeff} by $\mu$ and evaluating at $q=1$ as above then gives 
\[
7(\gamma - 1) = 2(6 \gamma - 1)
\] 
so $\gamma= -1$.
This in turn implies that $\mu\delta = -[2]$ and thus $\delta = -[3] \alpha$. 

Hence, if we can deduce the value of $\delta$, we will have identified all unknown coefficients, 
and deduced that all relations in \eqref{eq:spRel} that don't involve $3$-labeled edges hold
(We've also already deduced one relation involving a $3$-labeled edge holds.)
Note that we don't expect to explicitly identify $\delta$ at this point; 
indeed, there is some flexibility in our choice for this parameter, 
since changing our choice of 
$
\Psi \left( \;
\xy
(0,0)*{
\begin{tikzpicture}[scale =.35, smallnodes, yscale=-1]
	\draw[2label] (.5,.75) to (.5,1.5);
	\draw[1label] (.5,1.5) to [out=150,in=270] (0,2.25);
	\draw[1label] (.5,1.5) to [out=30,in=270] (1,2.25);
\end{tikzpicture}
};
\endxy
\; \right)
$
will change $\delta$ by the square of an element in $\C(q)$.
We thus can compute $\delta$ by 
choosing $M \in \Hom_{\spn[6]}(V\otimes V, W)$
and explicitly comparing the maps
\[
\Psi \left( \;
\xy
(0,0)*{
\begin{tikzpicture}[scale =.35, smallnodes, yscale=-1]
	\draw[2label] (.5,.75) to (.5,1.5);
	\draw[1label] (.5,1.5) to [out=150,in=270] (0,2.25) to [out=90,in=180] (1.5,4) to [out=0,in=90] (3,2.25);
	\draw[1label] (.5,1.5) to [out=30,in=270] (1,2.25) to [out=90,in=180] (1.5,3) to [out=0,in=90] (2,2.25);
	\draw[2label] (2.5,.75) to (2.5,1.5);
	\draw[1label] (2.5,1.5) to [out=150,in=270] (2,2.25);
	\draw[1label] (2.5,1.5) to [out=30,in=270] (3,2.25);
\end{tikzpicture}
};
\endxy
\; \right)
\text{ and }
\Psi \left( \;
\xy
(0,0)*{
\begin{tikzpicture}[scale =.5, smallnodes]
	\draw[2label] (0,1) to [out=270,in=180] (.5,0) to [out=0,in=270] (1,1);
\end{tikzpicture}
};
\endxy
\; \right)
\]
However, we now note that we have a choice in the value of the latter. 
Indeed, we can rescale the unit morphism for $W$ by any non-zero element in $\C(q)$, 
provided we also rescale its counit by the inverse. 
Hence, we can choose any non-zero value for $\delta$, 
and we let\footnote{Note that the different choice of $\delta = -[2][3]$ is used in \cite{LoganThesis}, 
which follows from the choices of (co)unit used there.} 
$\delta = [2][3]$. 
It then follows that $\alpha= -[2]$ and $\mu = \frac{-1}{[3]}$.

Next, since $\Hom_{\spn[6]}(V^{\otimes 3},U)$ is $1$-dimensional, 
we must have
\[
\Psi\left( \;
\xy
(0,0)*{

};
\endxy
\; \right)
\]
for some $\tau$, and an explicit computation verifies that $\tau \neq 0$.
We thus set
\[
\xy
(0,0)*{
%
};
\endxy
\]
which implies that
\begin{equation}\label{eq:Assoc}
\Psi \left( \;
\xy 
(0,0)*{
%
};
\endxy
\; \right).
\end{equation}
Given this, we then compute that 
\[
\Psi \left(\;
\xy
(0,0)*{
%
};
\endxy 
\; \right)
\]
so $\tau = [2][3]$, as desired.

Since $\End_{\spn[6]}(U)$ is $1$-dimensional, 
we have
\[
\Psi \left(
\xy
(0,0)*{
%
};
\endxy
\right)
= \delta' \cdot \id_U
\]
and we can determine $\delta'$ as before. 
Indeed, since we can rescale the image of a $3$-labeled cup 
(and/or the image of the $(1,2) \to 3$ trivalent vertex), 
we are free to choose any value for $\delta'$, 
and we again set\footnote{In \cite{LoganThesis}, the different choice of $\delta' = -[3]^2$ is used.} 
$\delta' = [2][3]$.

For our final two relations, 
observe that equation \eqref{eq:TensorDecomp}
implies that the images of
\[
\xy
(0,0)*{
%
};
\endxy 
\]
give a basis for $\End_{\spn[6]}(V \otimes W)$.
This implies that 
\begin{equation}\label{eq:12square}
\Psi \left( \;
\xy
(0,0)*{
%
};
\endxy
\; \right)
\end{equation}
for some $a,b,c \in \C(q)$. 
Closing the $1$-labeled strand, 
and using the relations
\[
\Psi \left(
\xy
(0,0)*{
%
};
\endxy
\right)
= -[6] \cdot \id_W
\]
(which themselves follow by observing that 
$\End_{\spn[6]}(V)$ and $\End_{\spn[6]}(U)$ are $1$-dimensional, 
taking a closure, and applying \eqref{eq:circles}), 
we find that
\[
[2]^2 [7] = \frac{[8]}{[4]}a - [2] b + \frac{[6]}{[3]} c.
\]
Similarly, composing \eqref{eq:12square} with 
$
\Psi \left( \;
\xy
(0,0)*{
%
};
\endxy
\; \right)
$
and using the relation
\[
\Psi \left( \;
\xy
(0,0)*{
%
};
\endxy 
\; \right)
\]
(which itself follows by composing \eqref{eq:Switching} with 
$
\Psi \left( \;
\xy
(0,0)*{
%
};
\endxy
\; \right)
$), we obtain the equation
\[
[4]^2 = a - [2][7]b,
\]
and composing with 
\eqref{eq:12square} with 
$
\Psi \left( \;
\xy
(0,0)*{
%
};
\endxy
\; \right)
$
gives
\[
[2]^2[3]^2 = a + [2][3] c.
\]
thus we have that
\[
a=[3]^2 \;\; , \;\; b= \frac{-1}{[2]} \;\; , \;\; c=\frac{[3]^2}{[2]}
\]
as desired.

Lastly, similar glueing/closure arguments show that
\[
\Psi \left( \;
\xy
(0,0)*{
%
};
\endxy
\; \right)
\]
are linearly independent, thus span $\Hom_{\spn[6]}(V\otimes W , W \otimes V)$. 
As above, this implies we must have a relation of the form
\[
\Psi \left( \;
\xy
(0,0)*{
%
};
\endxy
\; \right)
\]
for $x,y,z \in \C(q)$. 
Splitting the $2$-labeled endpoints (by composing with trivalent vertices), 
applying \eqref{eq:Assoc}, and using the self-duality structure, 
we find that $z = \pm1$. 
Further, composing with 
$
\Psi \left( \;
\xy
(0,0)*{
%
};
\endxy
\; \right)
$
and simplifying shows that $y=-zx$, so we have
\begin{equation}\label{eq:OtherSwitching}
\Psi \left( \;
\xy
(0,0)*{
%
};
\endxy
\; \right)
\right).
\end{equation}
Next, splitting the bottom-right $2$-labeled endpoint and glueing it to the bottom-left endpoint, 
then simplifying, gives the equation
\[
x = \frac{[6]}{[2][7] + z[4]}
\]
hence the proof is complete once we verify that $z=1$.

To do so, we first note that 
\[
\Psi \left( \;
\xy
(0,0)*{
%
};
\endxy
\; \right).
\]
which follows from \eqref{eq:12square}.
Using this, we see that composing \eqref{eq:OtherSwitching} 
with
$
\Psi \left( \;
\xy
(0,0)*{
%
};
\endxy
\; \right)
$
and simplifying gives a multiple of \eqref{eq:12square} when $z=1$, 
and a contradictory equation when $z=-1$.
\end{proof}

This proof suggests that we can define a braided structure on $\Web(\spn[6])$ 
by setting
\begin{equation}\label{eq:braiding}
\beta_{1,1} := 
\xy
(0,0)*{
%
};
\endxy
\end{equation}
The proof of Theorem \ref{thm:Psi} shows that we indeed have 
$\beta_{1,1}^{-1} \beta_{1,1} = 
\xy
(0,0)*{
%
};
\endxy \;
$.
Further, naturality implies that we must then set 
\begin{equation}\label{eq:coloredbraiding}
\begin{gathered}
\beta_{1,2} := 
\xy
(0,0)*{
%
};
\endxy
\end{gathered}
\end{equation}
These assignments then determine $\beta_{k,l}^{-1}$ for $k,l \in \{1,2,3\}$ using the pivotal structure. 
Explicit formulae for the crossings in \eqref{eq:coloredbraiding} 
can be found in Section \ref{sec:Links}.

\begin{thm}
The formulae in equations \eqref{eq:braiding} and \eqref{eq:coloredbraiding} endow 
$\Web(\spn[6])$ with the structure of a ribbon category.
\end{thm}
\begin{proof}
To begin, we extend the definition of $\beta$ to all objects 
$\vec{k} = (k_1,\ldots,k_m)$ and $\vec{l} = (l_1,\ldots,l_n)$
in $\Web(\spn[6])$ by setting
\[
\beta_{\vec{k},\vec{l}} :=
\xy
(0,0)*{
\begin{tikzpicture}[scale=.5,smallnodes]
	\draw[glabel] (3,0) node[below]{$l_1$} to [out=90,in=270] (0,3);
	\node at (4,0) {$\cdots$}; 
	\node at (4,3) {$\cdots$};
	\draw[glabel] (5,0) node[below]{$l_n$} to [out=90,in=270] (2,3);
	\draw[overcross] (0,0) to [out=90,in=270] (3,3);
	\draw[glabel] (0,0) node[below]{$k_1$} to [out=90,in=270] (3,3);
	\node at (1,0) {$\cdots$};
	\node at (1,3) {$\cdots$};
	\draw[overcross] (2,0) to [out=90,in=270] (5,3);
	\draw[glabel] (2,0) node[below]{$k_m$} to [out=90,in=270] (5,3);
\end{tikzpicture}
};
\endxy
\]
which also determines $\beta^{-1}$ using the pivotal structure. 
To see that $\beta$ indeed gives a braiding on $\Web(\spn[6])$, 
it suffices to see that $\beta$ is natural (with respect to morphisms in $\Web(\spn[6])$) 
and that the braid relations are satisfied.

The former follows via explicit computations that show that
\begin{equation}\label{eq:vertexslide}
\xy
(0,0)*{
\begin{tikzpicture}[scale=.75]
	\draw[glabel] (.5,-1) to [out=90,in=210] (.25,.25);
	\draw[glabel] (1,-1) to [out=90,in=-30] (.25,.25);
	\draw[glabel] (.25,.25) to [out=90,in=-90] (.25,.75);
	\draw[overcross] (0,-1) to [out=90,in=-90] (1,.75);
	\draw[glabel] (0,-1) to [out=90,in=-90] (1,.75);
\end{tikzpicture}
};
\endxy
=
\xy
(0,0)*{
\begin{tikzpicture}[scale=.75]
	\draw[glabel] (.5,-1) to [out=90,in=210] (.75,-.5);
	\draw[glabel] (1,-1) to [out=90,in=-30] (.75,-.5);
	\draw[glabel] (.75,-.5) to [out=90,in=-90] (.25,.75);
	\draw[overcross] (0,-1) to [out=90,in=-90] (1,.75);
	\draw[glabel] (0,-1) to [out=90,in=-90] (1,.75);
\end{tikzpicture}
};
\endxy
\quad \text{and} \quad
\xy
(0,0)*{
\begin{tikzpicture}[scale=.75]
	\draw[glabel] (0,-1) to [out=90,in=-90] (1,.75);
	\draw[overcross] (.5,-1) to [out=90,in=210] (.25,.25);
	\draw[glabel] (.5,-1) to [out=90,in=210] (.25,.25);
	\draw[overcross] (1,-1) to [out=90,in=-30] (.25,.25);
	\draw[glabel] (1,-1) to [out=90,in=-30] (.25,.25);
	\draw[glabel] (.25,.25) to [out=90,in=-90] (.25,.75);
\end{tikzpicture}
};
\endxy
=
\xy
(0,0)*{
\begin{tikzpicture}[scale=.75]
	\draw[glabel] (0,-1) to [out=90,in=-90] (1,.75);
	\draw[glabel] (.5,-1) to [out=90,in=210] (.75,-.5);
	\draw[glabel] (1,-1) to [out=90,in=-30] (.75,-.5);
	\draw[overcross] (.75,-.5) to [out=90,in=-90] (.25,.75);	
	\draw[glabel] (.75,-.5) to [out=90,in=-90] (.25,.75);
\end{tikzpicture}
};
\endxy
\end{equation}
For example, a computation shows that
\[
\xy
(0,0)*{
\begin{tikzpicture}[scale=.75]
	\draw[1label] (.5,-1) to [out=90,in=210] (.25,.25);
	\draw[1label] (1,-1) to [out=90,in=-30] (.25,.25);
	\draw[2label] (.25,.25) to [out=90,in=-90] (.25,.75);
	\draw[overcross] (0,-1) to [out=90,in=-90] (1,.75);
	\draw[1label] (0,-1) to [out=90,in=-90] (1,.75);
\end{tikzpicture}
};
\endxy
=
\frac{1}{[2]}
\xy
(0,0)*{
\begin{tikzpicture}[scale=.75]
	\draw[1label] (0,-1) to [out=90,in=210] (.5,-.25);
	\draw[1label] (.5,-1) to [out=90,in=210] (.75,-.625);
	\draw[1label] (1,-1) to [out=90,in=-30] (.75,-.625);
	\draw[2label] (.75,-.625) to [out=90,in=330] (.5,-.25);
	\draw[3label] (.5,-.25) to (.5,.25);
	\draw[1label] (.5,.25) to [out=30,in=270](.875,.75);
	\draw[2label] (.5,.25) to [out=150,in=270] (.125,.75);
\end{tikzpicture}
};
\endxy
-\frac{q}{[3]}
\xy
(0,0)*{
\begin{tikzpicture}[scale=.75]
	\draw[1label] (0,-1) to [out=90,in=210] (.25,.25);
	\draw[1label] (.5,-1) to [out=90,in=210] (.75,-.5);
	\draw[1label] (1,-1) to [out=90,in=-30] (.75,-.5);
	\draw[2label] (.75,-.5) to (.75,0);
	\draw[1label] (.75,0) to [out=150,in=-30] (.25,.25);
	\draw[1label] (.75,0) to [out=30,in=-90] (1,.75);
	\draw[2label] (.25,.25) to (.25,.75);
\end{tikzpicture}
};
\endxy
-\frac{q^{-2}}{[2][3]}
\xy
(0,0)*{
\begin{tikzpicture}[scale=.75]
	\draw[1label] (0,-1) to [out=90,in=210] (.5,-.25);
	\draw[1label] (.5,-1) to [out=90,in=210] (.75,-.625);
	\draw[1label] (1,-1) to [out=90,in=-30] (.75,-.625);
	\draw[2label] (.75,-.625) to [out=90,in=330] (.5,-.25);
	\draw[1label] (.5,-.25) to (.5,.25);
	\draw[1label] (.5,.25) to [out=30,in=270](.875,.75);
	\draw[2label] (.5,.25) to [out=150,in=270] (.125,.75);
\end{tikzpicture}
};
\endxy
\]
which gives the first relation in \eqref{eq:vertexslide} in this case 
by composing with 
$
\xy
(0,0)*{
\begin{tikzpicture}[scale =.35, smallnodes]
	\draw[1label] (-1,0) to (-1,2.25);
	\draw[1label] (0,0) to [out=90,in=210] (.5,.75);
	\draw[1label] (1,0) to [out=90,in=330] (.5,.75);
	\draw[2label] (.5,.75) to (.5,1.5);
	\draw[1label] (.5,1.5) to [out=150,in=270] (0,2.25);
	\draw[1label] (.5,1.5) to [out=30,in=270] (1,2.25);
\end{tikzpicture}
};
\endxy
$.

The braid relations then follow from the $1$-labeled case, 
i.e. from the relations
\begin{equation}\label{eq:braidrels}
\xy
(0,0)*{
\begin{tikzpicture} [scale=.5]
	\draw[1label] (1,0) to [out=90,in=270] (0,1.5) to [out=90,in=270] (1,3);
	\draw[overcross] (0,0) to [out=90,in=270] (1,1.5) to [out=90,in=270] (0,3);
	\draw[1label] (0,0) to [out=90,in=270] (1,1.5) to [out=90,in=270] (0,3);
\end{tikzpicture}
};
\endxy
=
\xy
(0,0)*{
\begin{tikzpicture} [scale=.5]
	\draw[1label] (1,0) to (1,3);
	\draw[1label] (0,0) to (0,3);
\end{tikzpicture}
};
\endxy
\quad \text{and} \quad
\xy
(0,0)*{
\begin{tikzpicture}[scale=.55]
	\draw[1label] (.8,-1) to [out=90,in=270] (0,.5);
	\draw[1label] (1.6,-1) to [out=90,in=270] (0,2);
	\draw[overcross] (0,.5) to [out=90,in=270] (.8,2);
	\draw[1label] (0,.5) to [out=90,in=270] (.8,2);
	\draw[overcross] (0,-1) to [out=90,in=270] (1.6,2);
	\draw[1label] (0,-1) to [out=90,in=270] (1.6,2);
\end{tikzpicture}
};
\endxy
=
\xy
(0,0)*{
\begin{tikzpicture}[scale=.55]
	\draw[1label] (1.6,-1) to [out=90,in=270] (0,2);
	\draw[1label] (1.6,.5) to [out=90,in=270] (.8,2);
	\draw[overcross] (.8,-1) to [out=90,in=270] (1.6,.5);
	\draw[1label] (.8,-1) to [out=90,in=270] (1.6,.5);
	\draw[overcross] (0,-1) to [out=90,in=270] (1.6,2);
	\draw[1label] (0,-1) to [out=90,in=270] (1.6,2);
\end{tikzpicture}
};
\endxy
\end{equation}
The first (Reidemeister II) of these relations holds since 
the coefficients in \eqref{eq:braiding} satisfy \eqref{eq:R2}, 
while the second (Reidemeister III) is given as follows:
\[
\begin{aligned}
\xy
(0,0)*{
\begin{tikzpicture}[scale=.55]
	\draw[1label] (.8,-1) to [out=90,in=270] (0,.5);
	\draw[1label] (1.6,-1) to [out=90,in=270] (0,2);
	\draw[overcross] (0,.5) to [out=90,in=270] (.8,2);
	\draw[1label] (0,.5) to [out=90,in=270] (.8,2);
	\draw[overcross] (0,-1) to [out=90,in=270] (1.6,2);
	\draw[1label] (0,-1) to [out=90,in=270] (1.6,2);
\end{tikzpicture}
};
\endxy
&=q
\xy
(0,0)*{
\begin{tikzpicture}[scale=.55]
	\draw[1label] (.8,-1) to [out=90,in=270] (0,1.2);
	\draw[1label] (1.6,-1) to [out=90,in=270] (.8,1.2);
	\draw[overcross] (0,-1) to [out=90,in=270] (1.6,1.2);
	\draw[1label] (0,-1) to [out=90,in=270] (1.6,1.2);
	\draw[1label] (0,1.2) to (0,2);
	\draw[1label] (.8,1.2) to (.8,2);
	\draw[1label] (1.6,1.2) to (1.6,2);
\end{tikzpicture}
};
\endxy
+\frac{q^{-3}}{[3]}
\xy
(0,0)*{
\begin{tikzpicture}[scale=.55]
	\draw[1label] (.8,-1) to [out=90,in=270] (0,.4);
	\draw[1label] (1.6,-1) to [out=90,in=270] (.8,.4);
	\draw[overcross] (0,-1) to [out=90,in=270] (1.6,.8);
	\draw[1label] (0,-1) to [out=90,in=270] (1.6,.8);
	\draw[1label] (0,.4) to [out=90,in=180] (.4,.8);
	\draw[1label] (.8,.4) to [out=90,in=0] (.4,.8);
	\draw[1label] (.4,1.6) to [out=180,in=270] (0,2);
	\draw[1label] (.4,1.6) to [out=0,in=270] (.8,2);
	\draw[1label] (1.6,.8) to (1.6,2);
\end{tikzpicture}
};
\endxy
-\frac{1}{[3]}
\xy
(0,0)*{
\begin{tikzpicture}[scale=.55]
	\draw[1label] (.8,-1) to [out=90,in=270] (0,.4);
	\draw[1label] (1.6,-1) to [out=90,in=270] (.8,.4);
	\draw[overcross] (0,-1) to [out=90,in=270] (1.6,.8);
	\draw[1label] (0,-1) to [out=90,in=270] (1.6,.8);
	\draw[1label] (0,.4) to [out=90,in=210] (.4,.8);
	\draw[1label] (.8,.4) to [out=90,in=-30] (.4,.8);
	\draw[2label] (.4,.8) to (.4,1.6);
	\draw[1label] (.4,1.6) to [out=150,in=270] (0,2);
	\draw[1label] (.4,1.6) to [out=30,in=270] (.8,2);
	\draw[1label] (1.6,.8) to (1.6,2);
\end{tikzpicture}
};
\endxy\\
&=q
\xy
(0,0)*{
\begin{tikzpicture}[scale=.55,rotate=180]
	\draw[1label] (.8,-1) to [out=90,in=270] (0,1.2);
	\draw[1label] (1.6,-1) to [out=90,in=270] (.8,1.2);
	\draw[overcross] (0,-1) to [out=90,in=270] (1.6,1.2);
	\draw[1label] (0,-1) to [out=90,in=270] (1.6,1.2);
	\draw[1label] (0,1.2) to (0,2);
	\draw[1label] (.8,1.2) to (.8,2);
	\draw[1label] (1.6,1.2) to (1.6,2);
\end{tikzpicture}
};
\endxy
+\frac{q^{-3}}{[3]}
\xy
(0,0)*{
\begin{tikzpicture}[scale=.55,rotate=180]
	\draw[1label] (.8,-1) to [out=90,in=270] (0,.4);
	\draw[1label] (1.6,-1) to [out=90,in=270] (.8,.4);
	\draw[overcross] (0,-1) to [out=90,in=270] (1.6,.8);
	\draw[1label] (0,-1) to [out=90,in=270] (1.6,.8);
	\draw[1label] (0,.4) to [out=90,in=180] (.4,.8);
	\draw[1label] (.8,.4) to [out=90,in=0] (.4,.8);
	\draw[1label] (.4,1.6) to [out=180,in=270] (0,2);
	\draw[1label] (.4,1.6) to [out=0,in=270] (.8,2);
	\draw[1label] (1.6,.8) to (1.6,2);
\end{tikzpicture}
};
\endxy
-\frac{1}{[3]}
\xy
(0,0)*{
\begin{tikzpicture}[scale=.55,rotate=180]
	\draw[1label] (.8,-1) to [out=90,in=270] (0,.4);
	\draw[1label] (1.6,-1) to [out=90,in=270] (.8,.4);
	\draw[overcross] (0,-1) to [out=90,in=270] (1.6,.8);
	\draw[1label] (0,-1) to [out=90,in=270] (1.6,.8);
	\draw[1label] (0,.4) to [out=90,in=210] (.4,.8);
	\draw[1label] (.8,.4) to [out=90,in=-30] (.4,.8);
	\draw[2label] (.4,.8) to (.4,1.6);
	\draw[1label] (.4,1.6) to [out=150,in=270] (0,2);
	\draw[1label] (.4,1.6) to [out=30,in=270] (.8,2);
	\draw[1label] (1.6,.8) to (1.6,2);
\end{tikzpicture}
};
\endxy
=
\xy
(0,0)*{
\begin{tikzpicture}[scale=.55]
	\draw[1label] (1.6,-1) to [out=90,in=270] (0,2);
	\draw[1label] (1.6,.5) to [out=90,in=270] (.8,2);
	\draw[overcross] (.8,-1) to [out=90,in=270] (1.6,.5);
	\draw[1label] (.8,-1) to [out=90,in=270] (1.6,.5);
	\draw[overcross] (0,-1) to [out=90,in=270] (1.6,2);
	\draw[1label] (0,-1) to [out=90,in=270] (1.6,2);
\end{tikzpicture}
};
\endxy
\end{aligned}
\]

Finally, a computation shows that
\begin{equation}\label{eq:ribbon}
\xy
(0,0)*{
\begin{tikzpicture}[scale=.5,yscale=-1]
	\draw[1label] (.8,-.4) to [out=180,in=270] (0,1);
	\draw[overcross] (0,-1) to [out=90,in=180] (.8,.4);
	\draw[1label] (0,-1) to [out=90,in=180] (.8,.4);
	\draw[1label] (.8,.4) to [out=0,in=90] (1.1,0) to [out=270,in=0] (.8,-.4);
\end{tikzpicture}
};
\endxy = 
-q^{-7} \;
\xy
(0,0)*{
\begin{tikzpicture}[scale=.5,yscale=-1]
	\draw[1label] (0,-1) to (0,1);
\end{tikzpicture}
};
\endxy
\quad \text{and} \quad
\xy
(0,0)*{
\begin{tikzpicture}[scale=.5]
	\draw[1label] (.8,-.4) to [out=180,in=270] (0,1);
	\draw[overcross] (0,-1) to [out=90,in=180] (.8,.4);
	\draw[1label] (0,-1) to [out=90,in=180] (.8,.4);
	\draw[1label] (.8,.4) to [out=0,in=90] (1.1,0) to [out=270,in=0] (.8,-.4);
\end{tikzpicture}
};
\endxy =
-q^{7} \;
\xy
(0,0)*{
\begin{tikzpicture}[scale=.5,yscale=-1]
	\draw[1label] (0,-1) to (0,1);
\end{tikzpicture}
};
\endxy
\end{equation}
hence
\[
\xy
(0,0)*{
\begin{tikzpicture}[scale=.5]
	\draw[1label] (.8,-.4) to [out=180,in=270] (0,1);
	\draw[overcross] (0,-1) to [out=90,in=180] (.8,.4);
	\draw[1label] (0,-1) to [out=90,in=180] (.8,.4);
	\draw[1label] (.8,.4) to [out=0,in=90] (1.1,0) to [out=270,in=0] (.8,-.4);
	\draw[1label] (0,1) to [out=90,in=180] (.8,2.4);
	\draw[1label] (.8,2.4) to [out=0,in=90] (1.1,2) to [out=270,in=0] (.8,1.6);
	\draw[overcross] (.8,1.6) to [out=180,in=270] (0,3);
	\draw[1label] (.8,1.6) to [out=180,in=270] (0,3);
\end{tikzpicture}
};
\endxy
= \;
\xy
(0,0)*{
\begin{tikzpicture}[scale=.5]
	\draw[1label] (0,-.5) to (0,2.5);
\end{tikzpicture}
};
\endxy
\]
so $\Web(\spn[6])$ is ribbon.
\end{proof}

Further, it's clear from the proof of Theorem \ref{thm:Psi} 
that the functor $\Psi$ defined therein is ribbon.
Thus, Theorem \ref{thmFull} will follow once we have shown that 
$\Psi$ is full. 

To aid in this task, we first establish the relation between 
$\Web(\spn[6])$ and the BMW algebra \cite{BW},\cite{Murakami}. 
We now recall its definition, 
following the conventions from \cite{Hu}.

\begin{defn}
The $k$-strand Birman-Murakami-Wenzl (BMW) algebra 
$\BMW_k(r,z)$ is the unital associative $\C(q)$-algebra generated by 
$e_i,g_i,g_i^{-1}$ for $1\leq i \leq k-1$, with relations:
\begin{enumerate}
	\item $g_i-g_i^{-1}=z(1-e_i)$
	\item $e_i^2=\left(1+\frac{r-r^{-1}}{z}\right)e_i$
	\item $g_ig_{i+1}g_i=g_{i+1}g_ig_{i+1}$ for $1 \leq i \leq k-2$
	\item $g_ig_j=g_jg_i$ for $|i-j|>1$
	\item $e_ie_{i+1}e_i=e_i$ and $e_{i+1}e_ie_{i+1}=e_{i+1}$ for $1 \leq i \leq k-2$
	\item $g_ig_{i+1}e_i=e_{i+1}e_i$ and $g_{i+1}g_ie_{i+1}=e_ie_{i+1}$
		for $1 \leq i \leq k-2$
	\item $e_ig_i=g_ie_i=r^{-1} e_i$
	\item $e_ig_{i+1}e_i=re_i$ and $e_{i+1}g_ie_{i+1}=re_{i+1}$
		for $1 \leq i \leq k-2$.
\end{enumerate}
\end{defn}
It is known that the BMW algebra is ``quantum Schur-Weyl'' dual
to $U_q(\spn)$, see \cite{Hu} and references therein.
In particular, there is a surjective homomorphism 
\begin{equation}\label{eq:BMWmap}
\BMW_k(-q^{7},q-q^{-1}) \to \End_{\spn[6]}(V^{\otimes k}).
\end{equation}
We show a partial analogue for $\Web(\spn[6])$.

\begin{prop}
There is a homomorphism 
$BWM(-q^{7},q-q^{-1}) \to \End_{\Web(\spn[6])}(1^{\otimes k})$
and \eqref{eq:BMWmap} factors through this.
\end{prop}
\begin{proof}
The homomorphism in \eqref{eq:BMWmap} is defined using the braiding and 
(co)unit morphisms in $\Rep(U_q(\spn[6]))$, 
thus it suffices to show that analogous formulae determine a homomorphism 
to $ \End_{\Web(\spn[6])}(1^{\otimes k})$, i.e. that
\[
g_i\mapsto
\xy
(0,0)*{
\begin{tikzpicture}[scale=.75]
	\draw[1label] (-1,0) to (-1,1.5);
	\node at (-.625,.75) {$\dots$};
	\draw[1label] (-.25,0) to (-.25,1.5);
	\draw[1label] (.75,0) to [out=90,in=-90] (0,1.5);
	\draw[overcross] (0,0) to [out=90,in=-90] (.75,1.5);
	\draw[1label] (0,0) to [out=90,in=-90] (.75,1.5);
	\draw[1label] (1,0) to (1,1.5);
	\node at (1.375,.75) {$\dots$};
	\draw[1label] (1.75,0) to (1.75,1.5);
\end{tikzpicture}
};
\endxy
\;\; , \;\;
g_i^{-1}\mapsto
\xy
(0,0)*{
\begin{tikzpicture}[scale=.75,xscale=-1]
    \draw[1label] (-1,0) to (-1,1.5);
    \node at (-.625,.75) {$\dots$};
    \draw[1label] (-.25,0) to (-.25,1.5);
	\draw[1label] (.75,0) to [out=90,in=-90] (0,1.5);
	\draw[overcross] (0,0) to [out=90,in=-90] (.75,1.5);
	\draw[1label] (0,0) to [out=90,in=-90] (.75,1.5);
	\draw[1label] (1,0) to (1,1.5);
	\node at (1.375,.75) {$\dots$};
	\draw[1label] (1.75,0) to (1.75,1.5);
\end{tikzpicture}
};
\endxy
\;\; , \;\;
e_i\mapsto
\xy
(0,0)*{
\begin{tikzpicture}[scale=.75]
	\draw[1label] (-1,0) to (-1,1.5);
	\node at (-.625,.75) {$\dots$};
	\draw[1label] (-.25,0) to (-.25,1.5);
	\draw[1label] (0,0) to [out=90,in=180] (.375,.5);
	\draw[1label] (.375,.5) to [out=0,in=90] (.75,0);
	\draw[1label] (0,1.5) to [out=-90,in=180] (.375,1);
	\draw[1label] (.375,1) to [out=0,in=-90] (.75,1.5);
	\draw[1label] (1,0) to (1,1.5);
	\node at (1.375,.75) {$\dots$};
	\draw[1label] (1.75,0) to (1.75,1.5);
\end{tikzpicture}
};
\endxy
\]
gives a homomorphism. 
(Here, the crossings and caps/cups 
occur between the $i^{th}$ and $(i+1)^{st}$ positions.)
A direct computation shows that this is indeed the case; 
we now summarize the computation.
\begin{enumerate}
\item Equation \eqref{eq:braiding} implies that 
\[
\xy
(0,0)*{
\begin{tikzpicture} [scale=.5]
	\draw[1label] (1,0) to [out=90,in=270] (0,1.5);
	\draw[overcross] (0,0) to [out=90,in=270] (1,1.5);
	\draw[1label] (0,0) to [out=90,in=270] (1,1.5);
\end{tikzpicture}
};
\endxy 
-
\xy
(0,0)*{
\begin{tikzpicture} [scale=.5]
	\draw[1label] (0,0) to [out=90,in=270] (1,1.5);
	\draw[overcross] (1,0) to [out=90,in=270] (0,1.5);
	\draw[1label] (1,0) to [out=90,in=270] (0,1.5);
\end{tikzpicture}
};
\endxy
=
(q - q^{-1}) \;
\xy
(0,0)*{
\begin{tikzpicture}[scale =.4, smallnodes]
	\draw[1label] (0,0) to (0,2.25);
	\draw[1label] (1,0) to (1,2.25);
\end{tikzpicture}
};
\endxy
+
\frac{q^{-3} - q^3}{[3]} \;
\xy
(0,0)*{
\begin{tikzpicture}[scale =.4, smallnodes]
	\draw[1label] (0,0) to [out=90,in=180] (.5,.75);
	\draw[1label] (1,0) to [out=90,in=0] (.5,.75);
	\draw[1label] (.5,1.5) to [out=180,in=270] (0,2.25);
	\draw[1label] (.5,1.5) to [out=0,in=270] (1,2.25);
\end{tikzpicture}
};
\endxy
=
(q-q^{-1}) 
\left( \;
\xy
(0,0)*{
\begin{tikzpicture}[scale =.4, smallnodes]
	\draw[1label] (0,0) to (0,2.25);
	\draw[1label] (1,0) to (1,2.25);
\end{tikzpicture}
};
\endxy
-
\xy
(0,0)*{
\begin{tikzpicture}[scale =.4, smallnodes]
	\draw[1label] (0,0) to [out=90,in=180] (.5,.75);
	\draw[1label] (1,0) to [out=90,in=0] (.5,.75);
	\draw[1label] (.5,1.5) to [out=180,in=270] (0,2.25);
	\draw[1label] (.5,1.5) to [out=0,in=270] (1,2.25);
\end{tikzpicture}
};
\endxy
\; \right).
\]

\item Since $1+ \frac{r-r^{-1}}{z} = 1-[7] = -\frac{[3][8]}{[4]}$, this relation follows by
observing that
$
-\frac{[3][8]}{[4]}
\xy
(0,0)*{
\begin{tikzpicture}[scale =.3, smallnodes]
	\draw[1label] (0,0) to [out=90,in=180] (.5,.75);
	\draw[1label] (1,0) to [out=90,in=0] (.5,.75);
	\draw[1label] (.5,1.5) to [out=180,in=270] (0,2.25);
	\draw[1label] (.5,1.5) to [out=0,in=270] (1,2.25);
\end{tikzpicture}
};
\endxy
$
is an idempotent in $\Web(\spn[6])$.

\item This follows from the Reidemeister III move in \eqref{eq:braidrels}.

\item This holds by a height exchange isotopy in $\Web(\spn[6])$, 
i.e. because this category is monoidal.

\item This holds via (planar) isotopy in $\Web(\spn[6])$, 
i.e. because this category is pivotal.

\item This follows from the Reidemeister II move in \eqref{eq:braidrels}, 
and planar isotopy.

\item This, and relation (8), hold by \eqref{eq:ribbon}.

\end{enumerate}
\end{proof}

We now complete the proof of Theorem \ref{thmFull}, 
by showing the following.

\begin{thm}\label{thm:Full}
The functor $\Psi: \Web(\spn[6]) \to \FRep(U_q(\spn[6]))$ is full.
\end{thm}

\begin{proof}
As noted above, 
the homomorphism $\BMW_k(-q^7,q-q^{-1}) \to \End_{\spn[6]}(V^{\otimes k})$ 
is known to be surjective. 
Since this factors through 
$\Psi:  \End_{\Web(\spn[6])}(1^{\otimes k}) \to  \End_{\spn[6]}(V^{\otimes k})$, 
this latter homomorphism is surjective as well.
Further, the unit/counit morphisms in $\Web(\spn[6])$ and $\FRep(U_q(\spn[6]))$ 
give the following commutative diagram
\[
\begin{tikzcd}
\Hom_{\Web(\spn[6])}(1^{\otimes r}, 1^{\otimes s}) \arrow[r, "\Psi"] \arrow[d, "\cong"]
& \Hom_{\spn[6]}(V^{\otimes r}, V^{\otimes s}) \arrow[d, "\cong"] \\ 
\End_{\Web(\spn[6])}(1^{\otimes \frac{r+s}{2}}) \arrow[r,"\Psi"]
& \End_{\spn[6]}(V^{\otimes \frac{r+s}{2}}) 
\end{tikzcd}
\]
for any $r,s \geq 0$ with $r+s$ even, which gives surjectivity for these $\Hom$-spaces.
Also, note that if $r+s$ is odd, then
\[
\Hom_{\Web(\spn[6])}(1^{\otimes r}, 1^{\otimes s}) = 0 
\]
since the parity of the sum of the entires in the domain and codomain are equal 
for all morphisms in $\Web(\spn[6])$.
Since $\Hom_{\spn[6]}(V^{\otimes r}, V^{\otimes s}) = 0$ for $r+s$ odd as well, 
we've thus shown that 
\[
\Psi: \Hom_{\Web(\spn[6])}(1^{\otimes r}, 1^{\otimes s}) \to \Hom_{\spn[6]}(V^{\otimes r}, V^{\otimes s})
\]
is surjective for all $r,s \geq 0$.

It remains to extend this to the remaining $\Hom$-spaces in $\Web(\spn[6])$. 
Let $\vec{k} = (k_1,\ldots,k_m)$ and $ \vec{l} = (l_1,\ldots,l_n)$ be objects in $\Web(\spn[6])$. 
Consider the webs
\[
\mathcal{W}_b = \bigotimes_{i=1}^m \mathcal{W}_{b,i} : (k_1,\ldots,k_m) \to 1^{\otimes \sum k_i}
\quad
\text{and} 
\quad
\mathcal{W}_t = \bigotimes_{j=1}^n \mathcal{W}_{t,j} : 1^{\otimes \sum l_j} \to (l_1,\ldots,l_n)
\]
defined by
\[
\mathcal{W}_{b,i} = 
\begin{cases}
\xy
(0,0)*{
\begin{tikzpicture}[scale =.4, smallnodes]
	\node at (-.75,0){};
	\draw[1label] (0,0) to (0,1.5);
\end{tikzpicture}
};
\endxy
& \text{if } k_i =1 \\
\xy
(0,0)*{
\begin{tikzpicture}[scale =.5, yscale=-1]
	\node at (-.125,0){};
	\draw[1label] (0,0) to [out=90,in=210] (.5,.75);
	\draw[1label] (1,0) to [out=90,in=330] (.5,.75);
	\draw[2label] (.5,.75) to (.5,1.5) ;
\end{tikzpicture}
};
\endxy
& \text{if } k_i =2 \\
\xy 
(0,0)*{
\begin{tikzpicture}[scale=.2, xscale=-1, yscale=-1]
	\draw [1label] (-1,-1) to [out=90,in=210] (0,.75);
	\draw [1label] (1,-1) to [out=90,in=330] (0,.75);
	\draw [1label] (3,-1) to [out=90,in=330] (1,2.5);
	\draw [2label] (0,.75) to [out=90,in=210] (1,2.5);
	\draw [3label] (1,2.5) to (1,4.25);
	\node at (1,-1.125){};
\end{tikzpicture}
};
\endxy 
& \text{if } k_i = 3
\end{cases}
\quad \text{and} \quad
\mathcal{W}_{t,j} = 
\begin{cases}
\xy
(0,0)*{
\begin{tikzpicture}[scale =.4, smallnodes]
	\node at (-.75,0){};
	\draw[1label] (0,0) to (0,1.5);
\end{tikzpicture}
};
\endxy
& \text{if } l_j =1 \\
\xy
(0,0)*{
\begin{tikzpicture}[scale =.5, smallnodes]
	\node at (-.125,0){};
	\draw[1label] (0,0) to [out=90,in=210] (.5,.75);
	\draw[1label] (1,0) to [out=90,in=330] (.5,.75);
	\draw[2label] (.5,.75) to (.5,1.5) ;
\end{tikzpicture}
};
\endxy
& \text{if } l_j =2 \\
\xy 
(0,0)*{
\begin{tikzpicture}[scale=.2, xscale=-1]
	\draw [1label] (-1,-1) to [out=90,in=210] (0,.75);
	\draw [1label] (1,-1) to [out=90,in=330] (0,.75);
	\draw [1label] (3,-1) to [out=90,in=330] (1,2.5);
	\draw [2label] (0,.75) to [out=90,in=210] (1,2.5);
	\draw [3label] (1,2.5) to (1,4.25);
\end{tikzpicture}
};
\endxy 
& \text{if } l_j = 3
\end{cases}
\]
We then obtain a $\C(q)$-linear map
\[
\Hom_{\Web(\spn[6])}(1^{\otimes \sum k_i} , 1^{\otimes \sum l_j})
\xrightarrow{\mathcal{W}_t \circ ( - ) \circ \mathcal{W}_b} 
\Hom_{\Web(\spn[6])}(\vec{k}, \vec{l})
\]
which is surjective by the relations in the first line of \eqref{eq:spRel}.
This in turn implies that
\[
\Hom_{\spn[6]}(V^{\otimes \sum k_i} , V^{\otimes \sum l_j})
\xrightarrow{\Psi(\mathcal{W}_t) \circ ( - ) \circ \Psi(\mathcal{W}_b)} 
\Hom_{\spn[6]}(\Psi(\vec{k}), \Psi(\vec{l})) 
\]
is surjective as well, and the result then follows from the diagram
\[
\begin{tikzcd}
\Hom_{\Web(\spn[6])}(1^{\otimes \sum k_i} , 1^{\otimes \sum l_j}) 
\arrow[r, "\Psi"] \arrow[d, "\mathcal{W}_t \circ ( - ) \circ \mathcal{W}_b"]
& \Hom_{\spn[6]}(V^{\otimes \sum k_i} , V^{\otimes \sum l_j})  
\arrow[d, "\Psi(\mathcal{W}_t) \circ ( - ) \circ \Psi(\mathcal{W}_b)"] \\ 
\Hom_{\Web(\spn[6])}(\vec{k}, \vec{l}) \arrow[r,"\Psi"]
& \Hom_{\spn[6]}(\Psi(\vec{k}), \Psi(\vec{l})) 
\end{tikzcd}
\]
\end{proof}

\section{Closed webs and ladders}\label{sec:Closed}

In this section, we prove Theorems \ref{thmEmpty} and \ref{thmTr}. 
Save for the second statement in Theorem \ref{thmEmpty}, 
these results can be interpreted topologically as saying that closed webs, 
living in the plane and annulus, respectively, can be expressed in terms of the simplest 
such webs (the empty web and nested essential circles, respectively). 
We begin by discussing our approach, and introducing the requisite structures.

\subsection{Strategy and $\Lad(\gspn[6])$}

In \cite{TubSar}, a web formalism is used to describe morphisms between 
representations of certain type $BCD$ coideal subalgebras of $U_q(\gln)$, 
and to prove quantum analogues of Howe dualities in these types.
However, their results suggest that Howe dualities cannot be used 
(at least in a straightforward way) to give descriptions of $\Rep(U_q(\mathfrak{g}))$ 
in types $BCD$.
Our main observation is that several salient structures used in the 
study of type $A$ webs can be 
exploited independent of their connection to Howe duality.
As such, we abandon Howe duality, and instead attempt to 
parallel these aspects of the type $A$ story.

The Cautis-Kamnitzer-Morrison approach to $\Web(\sln)$ proceeds by considering 
several auxiliary categories that are related to $\Rep(U_q(\sln))$.
Specifically, for fixed $n>0$, they consider the category 
$\Lad(\gln)$ of ``ladder-like'' $\gln$-webs, 
and show that a certain full subcategory $\Lad_m(\gln)$ is equivalent 
to the quotient of the idempotent form $\dot{U}_q(\gln[m])$ by the kernel of the 
map $\dot{U}_q(\gln[m]) \to \Rep(U_q(\gln))$ induced via quantum skew Howe duality. 
(Recall that the latter is the quantum analogue of the duality between $\gln[m]$ and $\gln$ 
induced by their commuting actions on $\largewedge^\bullet(\C^m \otimes \C^n)$.)
Their construction of $\Web(\sln)$ then follows by analyzing the diagram:
\[
\begin{tikzcd}
\Lad(\gln) \arrow[r] \arrow[rd] & \Web(\gln) \arrow[r] \arrow[d] & \Web(\sln) \arrow[d] \\
& \Rep(U_q(\gln)) \arrow[r] & \Rep(U_q(\sln))
\end{tikzcd}
\]
where $\Web(\gln)$ is a $\gln$ analogue of $\Web(\sln)$.
Most-importantly for our considerations, 
there exist explicit proofs of the type $A$ analogues of
Theorems \ref{thmEmpty} and \ref{thmTr} using properties of $\Lad(\gln)$ 
that can be formulated without reference to skew Howe duality; 
see \cite{GLL} and \cite{QR2,QRS}. 

This suggests that we should introduce type $C_3$ versions of 
$\Lad(\gln)$ and $\Web(\gln)$, 
denoted $\Lad(\gspn[6])$ and $\Web(\gspn[6])$, respectively.
These categories admit functors to $\Web(\spn[6])$ 
possessing certain fullness properties, 
but also have a rigidity\footnote{Rigid in the colloquial, non-category theoretic sense.} 
of structure that allows for the proof of Theorems \ref{thmEmpty} and \ref{thmTr}. 

\begin{rem}\label{rem:Web(gsp)}
The category $\Web(\gspn[6])$ 
is not strictly used in this work, 
i.e. all of our proofs bypass it, using only the relation between 
$\Lad(\gspn[6])$ and $\Web(\spn[6])$.
However, $\Web(\gspn[6])$ serves to motivate the definition of $\Lad(\gspn[6])$, 
which otherwise would seem much more ad hoc, so we briefly, 
and informally, discuss it now.

Paralleling the type $A$ case, the category $\Web(\gspn[6])$ is related to the category of 
representations of $\gspn[6]$.
The latter is the Lie algebra of the Lie group
\[
GSp(6) = \{ A \in M_6(\C) \mid A^{T} J A = \lambda_A J \}.
\]
where here $J \in M_n(\C)$ is non-degenerate and skew symmetric, 
and $\lambda_A \in \C$.
The tensor square of the vector representation no longer 
contains the trivial representation as a summand, 
but rather an irreducible $1$-dimensional representation 
corresponding to $A \mapsto \lambda_A$. 
We thus append a new generating objection $0'$ to the generators of 
$\Web(\spn[6])$, and let $\Web(\gspn[6])$ be generated by objects
$\{0',1,2,3\}$ and morphisms 
\[
\xy
(0,0)*{
\begin{tikzpicture}[scale =.5, smallnodes]
	\draw[1label] (0,0) node[below]{$1$} to [out=90,in=210] (.5,.75);
	\draw[1label] (1,0) node[below]{$1$} to [out=90,in=330] (.5,.75);
	\draw[0label] (.5,.75) to (.5,1.5) node[above]{$0'$};
\end{tikzpicture}
};
\endxy
\;\; , \;\;
\xy
(0,0)*{
\begin{tikzpicture}[scale =.5, smallnodes]
	\draw[1label] (0,0) node[below]{$1$} to [out=90,in=210] (.5,.75);
	\draw[1label] (1,0) node[below]{$1$} to [out=90,in=330] (.5,.75);
	\draw[2label] (.5,.75) to (.5,1.5) node[above]{$2$};
\end{tikzpicture}
};
\endxy
\;\; , \;\;
\xy
(0,0)*{
\begin{tikzpicture}[scale =.5, smallnodes]
	\draw[1label] (0,0) node[below]{$1$} to [out=90,in=210] (.5,.75);
	\draw[2label] (1,0) node[below]{$2$} to [out=90,in=330] (.5,.75);
	\draw[3label] (.5,.75) to (.5,1.5) node[above]{$3$};
\end{tikzpicture}
};
\endxy
\]
together with vertical and horizontal reflections thereof.
We then impose analogues of the $\Web(\spn[6])$ relations, e.g.
\begin{equation}\label{eq:gspRel}
\begin{gathered}
\xy
(0,0)*{
\begin{tikzpicture}[scale=.2]
	\draw [0label] (0,.75) to (0,2.5);
	\draw [1label] (0,-2.75) to [out=30,in=330] (0,.75);
	\draw [1label] (0,-2.75) to [out=150,in=210] (0,.75);
	\draw [0label] (0,-4.5) to (0,-2.75);
\end{tikzpicture}
};
\endxy
= -\frac{[3][8]}{[4]}\,
\xy
(0,0)*{
\begin{tikzpicture}[scale=.2]
	\draw [0label] (0,-4.5) to (0,2.5);
\end{tikzpicture}
};
\endxy
\;\; , \;\;
\xy
(0,0)*{
\begin{tikzpicture}[scale=.2]
	\draw [2label] (0,.75) to (0,2.5);
	\draw [1label] (0,-2.75) to [out=30,in=330] (0,.75);
	\draw [1label] (0,-2.75) to [out=150,in=210] (0,.75);
	\draw [0label] (0,-4.5) to (0,-2.75);
\end{tikzpicture}
};
\endxy
=0
\quad , \;\;
\xy
(0,0)*{
\begin{tikzpicture}[scale=.2]
	\draw [2label] (0,.75) to (0,2.5);
	\draw [1label] (0,-2.75) to [out=30,in=330] (0,.75);
	\draw [1label] (0,-2.75) to [out=150,in=210] (0,.75);
	\draw [2label] (0,-4.5) to (0,-2.75);
\end{tikzpicture}
};
\endxy
= [2][3]\,
\xy
(0,0)*{
\begin{tikzpicture}[scale=.2]
	\draw [2label] (0,-4.5) to (0,2.5);
\end{tikzpicture}
};
\endxy
\quad , \;\;
\xy
(0,0)*{
\begin{tikzpicture}[scale=.2]
	\draw [3label] (0,.75) to (0,2.5);
	\draw [2label] (0,-2.75) to [out=30,in=330] (0,.75);
	\draw [1label] (0,-2.75) to [out=150,in=210] (0,.75);
	\draw [3label] (0,-4.5) to (0,-2.75);
\end{tikzpicture}
};
\endxy
= [2][3]\,
\xy
(0,0)*{
\begin{tikzpicture}[scale=.2]
	\draw [3label] (0,-2.75) to (0,2.5);
\end{tikzpicture}
};
\endxy
\quad , \;\;
\xy
(0,0)*{
\begin{tikzpicture}[scale=.45]
	\draw[1label] (0,0) to [out=90,in=210] (.5,1);
	\draw[1label] (1,0) to [out=90,in=330] (.5,1);
	\draw[2label] (.5,1) to [out=90,in=210] (1,2);
	\draw[1label] (1.75,0) to [out=90,in=330] (1,2);
	\draw[3label] (1,2) to (1,3);
\end{tikzpicture}
};
\endxy
=
\xy
(0,0)*{
\begin{tikzpicture}[scale=.45,xscale=-1]
	\draw[1label] (0,0) to [out=90,in=210] (.5,1);
	\draw[1label] (1,0) to [out=90,in=330] (.5,1);
	\draw[2label] (.5,1) to [out=90,in=210] (1,2);
	\draw[1label] (1.75,0) to [out=90,in=330] (1,2);
	\draw[3label] (1,2) to (1,3);
\end{tikzpicture}
};
\endxy
\\
\xy
(0,0)*{
\begin{tikzpicture}[scale=.75]
	\draw[3label] (0,-1) to (0,-.5);
	\draw[2label] (0,-.5) to [out=150,in=210] (0,.5);
	\draw[1label] (0,-.5) to (.5,-.25);
	\draw[1label] (.75,-1) to [out=90,in=330] (.5,-.25);
	\draw[2label] (.5,-.25) to  (.5,.25);
	\draw[1label] (.5,.25) to (0,.5);
	\draw[1label] (.5,.25) to [out=30,in=270] (.75,1);
	\draw[3label] (0,.5) to (0,1);
\end{tikzpicture}
};
\endxy
=[2]^2 [3] \,
\xy
(0,0)*{
\begin{tikzpicture}[scale=.75]
	\draw[3label] (0,-1) to (0,1);
	\draw[1label] (.75,-1) to (.75,1);
\end{tikzpicture}
};
\endxy
- [2] \,
\xy
(0,0)*{
\begin{tikzpicture}[scale=.75]
	\draw[3label] (0,-1) to (0,-.5);
	\draw[2label] (0,-.5) to [out=150,in=210] (0,.5);
	\draw[1label] (0,-.5) to (.5,-.25);
	\draw[1label] (.75,-1) to [out=90,in=330] (.5,-.25);
	\draw[0label] (.5,-.25) to  (.5,.25);
	\draw[1label] (.5,.25) to (0,.5);
	\draw[1label] (.5,.25) to [out=30,in=270] (.75,1);
	\draw[3label] (0,.5) to (0,1);
\end{tikzpicture}
};
\endxy
\quad , \;\;
\xy
(0,0)*{
\begin{tikzpicture}[scale=.75]
	\draw[2label] (0,-1) to (0,-.5);
	\draw[1label] (0,-.5) to [out=150,in=210] (0,.5);
	\draw[1label] (0,-.5) to (.5,-.25);
	\draw[1label] (.75,-1) to [out=90,in=330] (.5,-.25);
	\draw[2label] (.5,-.25) to  (.5,.25);
	\draw[1label] (.5,.25) to (0,.5);
	\draw[1label] (.5,.25) to [out=30,in=270] (.75,1);
	\draw[2label] (0,.5) to (0,1);
\end{tikzpicture}
};
\endxy
=[3]^2\,
\xy
(0,0)*{
\begin{tikzpicture}[scale=.75]
	\draw[2label] (0,-1) to (0,1);
	\draw[1label] (.75,-1) to (.75,1);
\end{tikzpicture}
};
\endxy
- \frac{1}{[2]}\,
\xy
(0,0)*{
\begin{tikzpicture}[scale=.75]
	\draw[2label] (0,-1) to (0,-.5);
	\draw[1label] (0,-.5) to [out=150,in=210] (0,.5);
	\draw[1label] (0,-.5) to (.5,-.25);
	\draw[1label] (.75,-1) to [out=90,in=330] (.5,-.25);
	\draw[0label] (.5,-.25) to  (.5,.25);
	\draw[1label] (.5,.25) to (0,.5);
	\draw[1label] (.5,.25) to [out=30,in=270] (.75,1);
	\draw[2label] (0,.5) to (0,1);
\end{tikzpicture}
};
\endxy
+ \frac{[3]^2}{[2]}
\xy
(0,0)*{
\begin{tikzpicture}[scale=.75]
	\draw[2label] (-.25,-1) to [out=90,in=210] (.125,-.4);
	\draw[1label] (.5,-1) to [out=90,in=330] (.125,-.4);
	\draw[3label] (.125,-.4) to (.125,.4);
	\draw[2label] (.125,.4) to [out=150,in=270] (-.25,1);
	\draw[1label] (.125,.4) to [out=30,in=270] (.5,1);
\end{tikzpicture}
};
\endxy
\\
\xy
(0,0)*{
\begin{tikzpicture}[scale=.75]
	\draw[1label] (-.25,-1.25) to [out=90,in=210] (0,-.875);
	\draw[1label] (.25,-1.25) to [out=90,in=-30] (0,-.875);
	\draw[2label] (0,-.875) to (0,-.5);
	\draw[1label] (0,-.5) to [out=150,in=210] (0,.5);
	\draw[1label] (0,-.5) to (.5,-.25);
	\draw[1label] (.75,-1.25) to [out=90,in=330] (.5,-.25);
	\draw[0label] (.5,-.25) to  (.5,.25);
	\draw[1label] (.5,.25) to (0,.5);
	\draw[1label] (.5,.25) to [out=30,in=270] (.75,1);
	\draw[0label] (0,.5) to (0,1);
\end{tikzpicture}
};
\endxy
-
\xy
(0,0)*{
\begin{tikzpicture}[scale=.75]
	\draw[1label] (0,-1.25) to [out=90,in=210] (.25,.125);
	\draw[1label] (.5,-1.25) to [out=90,in=210] (.75,-.875);
	\draw[1label] (1,-1.25) to [out=90,in=-30] (.75,-.875);
	\draw[2label] (.75,-.875) to (.75,-.125);
	\draw[1label] (.75,-.125) to (.25,.125);
	\draw[1label] (.75,-.125) to [out=30,in=270] (1,1);
	\draw[0label] (.25,.125) to (.25,1);
\end{tikzpicture}
};
\endxy
= [2] \left( \;
\xy
(0,0)*{
\begin{tikzpicture}[scale=.75]
	\draw[1label] (0,-1.25) to [out=90,in=210] (.25,.125);
	\draw[1label] (.5,-1.25) to [out=90,in=210] (.75,-.875);
	\draw[1label] (1,-1.25) to [out=90,in=-30] (.75,-.875);
	\draw[0label] (.75,-.875) to (.75,-.125);
	\draw[1label] (.75,-.125) to (.25,.125);
	\draw[1label] (.75,-.125) to [out=30,in=270] (1,1);
	\draw[0label] (.25,.125) to (.25,1);
\end{tikzpicture}
};
\endxy
-
\xy
(0,0)*{
\begin{tikzpicture}[scale=.75]
	\draw[1label] (-.25,-1.25) to [out=90,in=210] (0,-.875);
	\draw[1label] (.25,-1.25) to [out=90,in=-30] (0,-.875);
	\draw[0label] (0,-.875) to (0,-.5);
	\draw[1label] (0,-.5) to [out=150,in=210] (0,.5);
	\draw[1label] (0,-.5) to (.5,-.25);
	\draw[1label] (.75,-1.25) to [out=90,in=330] (.5,-.25);
	\draw[0label] (.5,-.25) to  (.5,.25);
	\draw[1label] (.5,.25) to (0,.5);
	\draw[1label] (.5,.25) to [out=30,in=270] (.75,1);
	\draw[0label] (0,.5) to (0,1);
\end{tikzpicture}
};
\endxy
\; \right)
\quad , \;\;
\xy
(0,0)*{
\begin{tikzpicture}[scale=.75]
	\draw[0label] (0,-1) to (0,-.5);
	\draw[1label] (0,-.5) to [out=150,in=210] (0,.5);
	\draw[1label] (0,-.5) to (.5,-.25);
	\draw[1label] (.75,-1) to [out=90,in=330] (.5,-.25);
	\draw[0label] (.5,-.25) to  (.5,.25);
	\draw[1label] (.5,.25) to (0,.5);
	\draw[1label] (.5,.25) to [out=30,in=270] (.75,1);
	\draw[0label] (0,.5) to (0,1);
\end{tikzpicture}
};
\endxy
= \;
\xy
(0,0)*{
\begin{tikzpicture}[scale=.75]
	\draw[0label] (0,-1) to (0,1);
	\draw[1label] (.75,-1) to (.75,1);
\end{tikzpicture}
};
\endxy
\end{gathered}
\end{equation}

The astute reader will note that rescaling both of the generating morphisms 
\[
\xy
(0,0)*{
\begin{tikzpicture}[scale =.5, smallnodes]
	\draw[1label] (0,0) node[below]{$1$} to [out=90,in=210] (.5,.75);
	\draw[1label] (1,0) node[below]{$1$} to [out=90,in=330] (.5,.75);
	\draw[0label] (.5,.75) to (.5,1.5) node[above]{$0'$};
\end{tikzpicture}
};
\endxy
\quad \text{and} \quad
\xy
(0,0)*{
\begin{tikzpicture}[scale =.5, smallnodes,yscale=-1]
	\draw[1label] (0,0) node[above]{$1$} to [out=90,in=210] (.5,.75);
	\draw[1label] (1,0) node[above]{$1$} to [out=90,in=330] (.5,.75);
	\draw[0label] (.5,.75) to (.5,1.5) node[below]{$0'$};
\end{tikzpicture}
};
\endxy
\]
by $\sqrt{-1}$ removes all minus signs from the above relations. 
As such, save for a troubling denominator, $\Web(\gspn[6])$ seems ripe for categorification. 
We plan a detailed study of $\Web(\gspn[6])$, in the context of type $C$ link homologies, 
in future work.
Note, however, that the
$\Web(\spn[6])$ relation
\[
\xy
(0,0)*{
\begin{tikzpicture}[scale=.3]
	\draw[1label] (-1,0) to (0,1);
	\draw[2label] (1,0) to (0,1);
	\draw[2label] (0,2.5) to (-1,3.5);
	\draw[1label] (0,2.5) to (1,3.5);
	\draw[1label] (0,1) to (0,2.5);
\end{tikzpicture}
};
\endxy
-
\xy
(0,0)*{
\begin{tikzpicture}[scale=.3,rotate=90,xscale=-1]
	\draw[1label] (-1,0) to (0,1);
	\draw[2label] (1,0) to (0,1);
	\draw[2label] (0,2.5) to (-1,3.5);
	\draw[1label] (0,2.5) to (1,3.5);
	\draw[1label] (0,1) to (0,2.5);
\end{tikzpicture}
};
\endxy
=[3] \Bigg(
\xy
(0,0)*{
\begin{tikzpicture}[scale=.3,rotate=90,xscale=-1]
	\draw[1label] (-1,0) to (0,1);
	\draw[2label] (1,0) to (0,1);
	\draw[2label] (0,2.5) to (-1,3.5);
	\draw[1label] (0,2.5) to (1,3.5);
	\draw[3label] (0,1) to (0,2.5);
\end{tikzpicture}
};
\endxy
-
\xy
(0,0)*{
\begin{tikzpicture}[scale=.3]
	\draw[1label] (-1,0) to (0,1);
	\draw[2label] (1,0) to (0,1);
	\draw[2label] (0,2.5) to (-1,3.5);
	\draw[1label] (0,2.5) to (1,3.5);
	\draw[3label] (0,1) to (0,2.5);
\end{tikzpicture}
};
\endxy
\; \Bigg)
\]
has no analogue in $\Web(\gspn[6])$. 
We would like to impose the relation
\begin{equation}\label{eq:want}
\xy
(0,0)*{
\begin{tikzpicture}[scale=.75]
	\draw[1label] (-.25,-1.25) to [out=90,in=210] (0,-.875);
	\draw[2label] (.25,-1.25) to [out=90,in=-30] (0,-.875);
	\draw[1label] (0,-.875) to (0,-.5);
	\draw[2label] (0,-.5) to [out=150,in=210] (0,.5);
	\draw[1label] (0,-.5) to (.5,-.25);
	\draw[1label] (.75,-1.25) to [out=90,in=330] (.5,-.25);
	\draw[0label] (.5,-.25) to  (.5,.25);
	\draw[2label] (.5,.25) to (0,.5);
	\draw[2label] (.5,.25) to [out=30,in=270] (.75,1);
	\draw[0label] (0,.5) to (0,1);
\end{tikzpicture}
};
\endxy
-
\xy
(0,0)*{
\begin{tikzpicture}[scale=.75]
	\draw[1label] (0,-1.25) to [out=90,in=210] (.25,.125);
	\draw[2label] (.5,-1.25) to [out=90,in=210] (.75,-.875);
	\draw[1label] (1,-1.25) to [out=90,in=-30] (.75,-.875);
	\draw[1label] (.75,-.875) to (.75,-.125);
	\draw[1label] (.75,-.125) to (.25,.125);
	\draw[2label] (.75,-.125) to [out=30,in=270] (1,1);
	\draw[0label] (.25,.125) to (.25,1);
\end{tikzpicture}
};
\endxy
= [3] \left( \;
\xy
(0,0)*{
\begin{tikzpicture}[scale=.75]
	\draw[1label] (0,-1.25) to [out=90,in=210] (.25,.125);
	\draw[2label] (.5,-1.25) to [out=90,in=210] (.75,-.875);
	\draw[1label] (1,-1.25) to [out=90,in=-30] (.75,-.875);
	\draw[3label] (.75,-.875) to (.75,-.125);
	\draw[1label] (.75,-.125) to (.25,.125);
	\draw[2label] (.75,-.125) to [out=30,in=270] (1,1);
	\draw[0label] (.25,.125) to (.25,1);
\end{tikzpicture}
};
\endxy
-
\xy
(0,0)*{
\begin{tikzpicture}[scale=.75]
	\draw[1label] (-.25,-1.25) to [out=90,in=210] (0,-.875);
	\draw[2label] (.25,-1.25) to [out=90,in=-30] (0,-.875);
	\draw[3label] (0,-.875) to (0,-.5);
	\draw[2label] (0,-.5) to [out=150,in=210] (0,.5);
	\draw[1label] (0,-.5) to (.5,-.25);
	\draw[1label] (.75,-1.25) to [out=90,in=330] (.5,-.25);
	\draw[0label] (.5,-.25) to  (.5,.25);
	\draw[2label] (.5,.25) to (0,.5);
	\draw[2label] (.5,.25) to [out=30,in=270] (.75,1);
	\draw[0label] (0,.5) to (0,1);
\end{tikzpicture}
};
\endxy
\; \right)
\end{equation}
as an analogue, 
but this requires vertices which are not (yet) defined in $\Web(\gspn[6])$.
Thus, the above informal definition of $\Web(\gspn[6])$ should be expanded to include them.
(The use of such vertices is also suggested by the following.)
\end{rem}

We now introduce the $\gspn[6]$ ladder category. 
To motivate the definition, note that we would like to define morphisms in this category 
to be $\Web(\gspn[6])$ webs written in the ladder form of \cite{CKM}. 
We would then impose ladder analogues of the relations in $\Web(\gspn[6])$,
e.g.
\begin{equation}\label{eq:problem}
\xy
(0,0)*{
\begin{tikzpicture}[anchorbase, scale=.75]
	\draw[2label] (0,0) to (0,.5);
	\draw[1label] (.75,0) to (.75,1);
	\draw[1label] (0,.5) to (0,2);
	\draw[1label] (0,.5) to (.75,1);
	\draw[2label] (.75,1) to (.75,1.5);
	\draw[1label] (.75,1.5) to (0,2);
	\draw[1label] (.75,1.5) to (.75,2.5);
	\draw[2label] (0,2) to (0,2.5);
\end{tikzpicture}
};
\endxy
=[3]^2 \;
\xy
(0,0)*{
\begin{tikzpicture}[anchorbase, scale=.75]
	\draw[2label] (0,0) to (0,2.5);
	\draw[1label] (.75,0) to (.75,2.5);
\end{tikzpicture}
};
\endxy
-\frac{1}{[2]} \;
\xy
(0,0)*{
\begin{tikzpicture}[anchorbase, scale=.75]
	\draw[2label] (0,0) to (0,.5);
	\draw[1label] (.75,0) to (.75,1);
	\draw[1label] (0,.5) to (0,2);
	\draw[1label] (0,.5) to (.75,1);
	\draw[0label] (.75,1) to (.75,1.5);
	\draw[1label] (.75,1.5) to (0,2);
	\draw[1label] (.75,1.5) to (.75,2.5);
	\draw[2label] (0,2) to (0,2.5);
\end{tikzpicture}
};
\endxy
+\frac{[3]^2}{[2]} \;
\xy
(0,0)*{
\begin{tikzpicture}[anchorbase, scale=.75]
	\draw[2label] (0,0) to (0,1);
	\draw[1label] (.75,0) to (.75,.5);
	\draw[1label] (.75,.5) to (0,1);
	\draw[3label] (0,1) to (0,1.5);
	\draw[2label] (0,1.5) to (0,2.5);
	\draw[1label] (0,1.5) to (.75,2);
	\draw[1label] (.75,2) to (.75,2.5);
\end{tikzpicture}
};
\endxy
\end{equation}
However, we eventually wish to employ PBW
style arguments to prove Theorems \ref{thmEmpty} and \ref{thmTr}, 
where negatively and positively sloped ``ladder rungs'' play the roles of 
the Chevalley generators $E$ and $F$ in the 
PBW theorem for $\sln[2]$, respectively. 
The second term on the right-hand side of relation \eqref{eq:problem} is then problematic. 
However, we note that the ``middle slice'' of this term corresponds to the object $1 \otimes 0'$, 
and, by introducing a new object $1' \cong 1 \otimes 0'$ 
(together with further generating morphisms and relations), 
we can replace this by the ladder
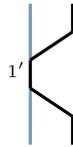
\[
\xy
(0,0)*{
\begin{tikzpicture}[anchorbase, scale=.75]
	\draw[2label] (0,0) to (0,1);
	\draw[1label] (.75,0) to (.75,.5);
	\draw[1label] (.75,.5) to (0,1);
	\draw[1label] (0,1) to node[left, smallnodes, xshift=2pt]{$1'$} (0,1.5);
	\draw[2label] (0,1.5) to (0,2.5);
	\draw[1label] (0,1.5) to (.75,2);
	\draw[1label] (.75,2) to (.75,2.5);
\end{tikzpicture}
};
\endxy
\]
After doing so, relation \eqref{eq:problem} is amenable to PBW style arguments.
Similar considerations, and equation \eqref{eq:want} above, 
suggest that we should also introduce new objects 
$k^{(i)}$ that correspond to the tensor product
$k \otimes \underbrace{0' \otimes \cdots \otimes 0'}_{i}$.

\begin{defn}
The category $\Lad(\gspn[6])$ is the $\C(q)$-linear monoidal category with objects generated by:
\[
\{0^{(i)}, 1^{(j)}, 2^{(s)}, 3^{(t)} \mid i,j,s,t \geq 0\}.
\] 
Defining the \emph{mass} of an object by 
$\mu(k^{(i)}) = k+2i$,
the morphisms are then generated by the 
\emph{rung morphisms}:
\begin{equation}\label{eq:RungDef}
\begin{tikzpicture}[anchorbase,xscale=-1]
	\draw[glabel] (0,0) node[black, below]{$b^{(j)}$} to (0,1.5) node[black, above]{$y^{(t)}$};
	\draw[glabel] (.75,0) node[black, below]{$a^{(i)}$} to (.75,1.5) node[black, above]{$x^{(s)}$};
	\draw[glabel] (0,.5) to node[black, above,xshift=2pt]{$c^{(r)}$} (.75,1);
\end{tikzpicture}
\quad \text{and} \quad
\begin{tikzpicture}[anchorbase]
	\draw[glabel] (0,0) node[black, below]{$b^{(j)}$} to (0,1.5) node[black, above]{$y^{(t)}$};
	\draw[glabel] (.75,0) node[black, below]{$a^{(i)}$} to (.75,1.5) node[black, above]{$x^{(s)}$};
	\draw[glabel] (0,.5) to node[black, above]{$c^{(r)}$} (.75,1);
\end{tikzpicture}
\end{equation}
where the labels satisfy all of the following conditions:
\begin{enumerate}
\item \textbf{Rung mass:} $\mu(c^{(r)}) > 0$; 
\item \textbf{Upper vertex:} exactly one of the following holds for the elements $a,c,x$:
\begin{itemize}
\item two are equal and the other is $0$,
\item two are $1$ and the other is $2$, or
\item all are nonzero and distinct;
\end{itemize}
\item \textbf{Lower vertex:} exactly one of the following holds for the elements $b,c,y$:
\begin{itemize}
\item two are equal and the other is $0$,
\item two are $1$ and the other is $2$, or
\item all are nonzero and distinct;
\end{itemize}
\item \textbf{Mass preservation:}
$\mu\left(a^{(i)}\right) + \mu\left(c^{(r)}\right) = \mu\left(x^{(s)}\right)$ and $\mu\left(b^{(j)}\right) = \mu\left(y^{(t)}\right) + \mu\left(c^{(r)}\right)$.
\end{enumerate}

We adopt terminology from \cite{CKM}. 
Specifically, we will refer to compositions of tensor products of generating morphisms as \emph{ladders},
so morphisms in $\Lad(\gspn[6])$ are $\C(q)$-linear combinations of ladders.
The vertical line segments in ladders are called \emph{uprights}, 
and the segments passing between the uprights as \emph{rungs}.
We refer to the generating ladders in \eqref{eq:RungDef} as 
\emph{$E$-rungs} and \emph{$F$-rungs}, respectively.
Finally, we will write the object (or edge label) 
$\ell^{(k)}$ as $\ell$ with $k$ primes, when $k$ is small,
e.g. $2' = 2^{(1)}$ and $3 = 3^{(0)}$.

The morphisms are subject to relations, 
all of which take the following forms, 
or a reflection thereof. 
Here, we allow some rungs to have zero mass, with the understanding that 
such a rung is simply the corresponding identity morphism:

\begin{itemize}
\item \textbf{Rung explosion:}
If $\mu\left(c^{(r)}\right) > 2$ or $c^{(r)}=2^{(0)}$, we have that
\begin{equation}\label{eq:explosion}
\begin{tikzpicture}[anchorbase]
	\draw[glabel] (0,0) to (0,2);
	\draw[glabel] (.75,0) to (.75,2);
	\draw[glabel] (0,.75) to node[black, above]{$c^{(r)}$} (.75,1.25);
\end{tikzpicture}
=
\sum_i
f_i(q) \;
\begin{tikzpicture}[anchorbase, xscale=1.25]
	\draw[glabel] (0,0) to (0,2);
	\draw[glabel] (.75,0) to (.75,2);
	\draw[glabel] (0,.5) to node[below,black]{$b_i^{(t_i)}$}  (.75,1);
	\draw[glabel] (0,1) to node[above,black]{$a_i^{(s_i)}$} (.75,1.5);
\end{tikzpicture}
\end{equation}
with $f_i \in \C(q)$ and 
$\mu\left(c^{(r)}\right) > 
\max\left(\mu\left(a_i^{(s_i)}\right), \mu\left(b_i^{(t_i)}\right)\right)$.

\item \textbf{Rung swap:} 
For $a^{(s)},b^{(t)} \in \{1, 0'\}$, we have
\begin{equation}\label{eq:swap}
\xy
(0,0)*{
\begin{tikzpicture}[anchorbase,scale=1]
	\draw[glabel] (0,0) to (0,2.5);
	\draw[glabel] (.75,0) to (.75,2.5);
	\draw[glabel] (1.5,0) to (1.5,2.5);
	\draw[glabel] (0,.5) to node[above,black]{$a^{(s)}$} (.75,1);
	\draw[glabel] (1.5,1.5) to node[above,black,yshift=1pt,xshift=1pt]{$b^{(t)}$} (.75,2);
\end{tikzpicture}
};
\endxy
=
\begin{cases}
\displaystyle \sum_i f_i(q)\;
\xy
(0,0)*{
\begin{tikzpicture}[anchorbase,scale=1]
	\draw[glabel] (0,0) to (0,2.5);
	\draw[glabel] (.75,0) to (.75,2.5);
	\draw[glabel] (1.5,0) to (1.5,2.5);
	\draw[1label] (0,1.5) to node[above,black]{$1$} (.75,2);
	\draw[1label] (1.5,.5) to node[above,black,yshift=1pt,xshift=1pt]{$1$} (.75,1);
\end{tikzpicture}
};
\endxy
+ g(q)
\xy
(0,0)*{
\begin{tikzpicture}[anchorbase,scale=1]
	\draw[glabel] (0,0) to (0,2.5);
	\draw[glabel] (.75,0) to (.75,2.5);
	\draw[glabel] (1.5,0) to (1.5,2.5);
	\draw[1label] (0,.75) to node[above]{$1$} (.75,.25);
	\draw[1label] (1.5,.75) to node[above]{$1$} (.75,1.25);
	\draw[0label] (0,1.75) to node[above]{$0'$} (.75,2.25);
\end{tikzpicture}
};
\endxy
& \text{ if } a^{(s)} = 1 = b^{(t)} \\ \\
\hfil
\xy
(0,0)*{
\begin{tikzpicture}[anchorbase,scale=1]
	\draw[glabel] (0,0) to (0,2.5);
	\draw[glabel] (.75,0) to (.75,2.5);
	\draw[glabel] (1.5,0) to (1.5,2.5);
	\draw[glabel] (0,1.5) to node[above,black]{$a^{(s)}$} (.75,2);
	\draw[glabel] (1.5,.5) to node[above,black,yshift=1pt,xshift=1pt]{$b^{(t)}$} (.75,1);
\end{tikzpicture}
};
\endxy
& \text{ if } a^{(s)} = 0' \text{ or } b^{(t)}=0'
\end{cases}
\end{equation}
with $f,g_i\in\C(q)$.

\item \textbf{Square relation:}
For $a^{(s)},b^{(t)} \in \{1, 0'\}$, we have
\begin{equation}\label{eq:square}
\xy
(0,0)*{
\begin{tikzpicture}[anchorbase,scale=1]
	\draw[glabel] (0,0) to (0,2.5);
	\draw[glabel] (.75,0) to (.75,2.5);
	\draw[glabel] (0,.5) to node[below,black]{$b^{(t)}$} (.75,1);
	\draw[glabel] (.75,1.5) to node[above,yshift=1pt,xshift=1pt,black]{$a^{(s)}$} (0,2);
\end{tikzpicture}
};
\endxy
=f(q)\;
\xy
(0,0)*{
\begin{tikzpicture}[anchorbase,scale=1]
	\draw[glabel] (0,0) to (0,2.5);
	\draw[glabel] (.75,0) to (.75,2.5);
\end{tikzpicture}
};
\endxy
+\sum_i g_i(q)\;
\xy
(0,0)*{
\begin{tikzpicture}[anchorbase,scale=1,xscale=-1]
	\draw[glabel] (0,0) to (0,2.5);
	\draw[glabel] (.75,0) to (.75,2.5);
	\draw[glabel] (0,.5) to node[below,black]{$a^{(s)}$} (.75,1);
	\draw[glabel] (.75,1.5) to node[above,black]{$b^{(t)}$} (0,2);
\end{tikzpicture}
};
\endxy
\end{equation}
\end{itemize}
The precise relations (i.e. the values of the above coefficients)
depend on the labels and masses of the rungs and uprights, 
and are recorded in Appendix \ref{App:Rels}. 
There, we continue our edge-coloring conventions, 
denoting $1^{(j)}$- ,$2^{(s)}$-, and $3^{(t)}$-labeled 
ladder edges in \textbf{black}, \textcolor{ourblue}{\textbf{blue}}, and \textcolor{green}{\textbf{green}}. 
We denote $0^{(i)}$-labeled edges in thin, dashed black, 
as in Remark \ref{rem:Web(gsp)}.
All such relations are ``ladderized" versions of relations that hold in $\Web(\spn[6])$. 
\end{defn}
\begin{rem}
We content ourselves here with only displaying the general form of the relations in $\Lad(\gspn[6])$, 
as the specific forms of the relations are only used to show that the functor 
in Proposition \ref{thm:Phi} below is well-defined.
After that, the general forms of the relations presented above suffice for the proofs 
of the remaining results from Section \ref{sec:Theorems}.
\end{rem}

\begin{rem}
Ladders were introduced as a tool in type BCD representation theory in \cite{TubSar}, 
where they are used to prove Howe dualities between quantum groups in these types
and certain coideal subalgebras of $U_q(\gln)$. 
In that setup, ladders describe morphisms between representations of these coideal subalgebras, 
rather than the quantum groups.
Hence, we do not expect a direct relation to our ladder category, except in the classical $q \to 1$ limit. 
It would be interesting to compare our categories in that case; 
however, we note that their edge labels correspond to (skew)symmetric tensors, 
as opposed to fundamental representations.
\end{rem}

\begin{prop}\label{thm:Phi}
There is a monoidal functor $\Phi:\Lad(\gspn[6])\to\Web(\spn[6])$
given by 
\[
\begin{gathered}
\Phi
\left( x^{(i)} \right) = 
\begin{cases}
x & \text{if } x \neq 0 \\
\varnothing & \text{if } x=0
\end{cases} \\
\Phi \left(
\begin{tikzpicture}[anchorbase,xscale=-1]
	\draw[glabel] (0,0) node[black, below]{$b^{(j)}$} to (0,1.5) node[black, above]{$y^{(t)}$};
	\draw[glabel] (.75,0) node[black, below]{$a^{(i)}$} to (.75,1.5) node[black, above]{$x^{(s)}$};
	\draw[glabel] (0,.5) to node[black, above,xshift=2pt]{$c^{(r)}$} (.75,1);
\end{tikzpicture}
\right) =
\xy
(0,0)*{
\begin{tikzpicture}[anchorbase,xscale=-1]
	\draw[glabel] (0,-1.25) node[black,below]{$b$} to (0,-.5);
	\draw[glabel] (0,-.5) to [out=150,in=270] (-.25,.5) node[black,above]{$y$};
	\draw[glabel] (0,-.5) to node[black,above]{$c$} (.5,-.25);
	\draw[glabel] (.75,-1.25) node[black,below]{$a$} to [out=90,in=-30] (.5,-.25);
	\draw[glabel] (.5,-.25) to [out=90,in=270] (.5,.5) node[black,above]{$x$};
\end{tikzpicture}
};
\endxy
\quad \text{and} \quad
\Phi \left(
\begin{tikzpicture}[anchorbase]
	\draw[glabel] (0,0) node[black, below]{$b^{(j)}$} to (0,1.5) node[black, above]{$y^{(t)}$};
	\draw[glabel] (.75,0) node[black, below]{$a^{(i)}$} to (.75,1.5) node[black, above]{$x^{(s)}$};
	\draw[glabel] (0,.5) to node[black, above]{$c^{(r)}$} (.75,1);
\end{tikzpicture}
\right) = 
\xy
(0,0)*{
\begin{tikzpicture}[anchorbase]
	\draw[glabel] (0,-1.25) node[black,below]{$b$} to (0,-.5);
	\draw[glabel] (0,-.5) to [out=150,in=270] (-.25,.5) node[black,above]{$y$};
	\draw[glabel] (0,-.5) to node[black,above]{$c$} (.5,-.25);
	\draw[glabel] (.75,-1.25) node[black,below]{$a$} to [out=90,in=-30] (.5,-.25);
	\draw[glabel] (.5,-.25) to [out=90,in=270] (.5,.5) node[black,above]{$x$};
\end{tikzpicture}
};
\endxy
\end{gathered}
\]
\end{prop}

\begin{proof}
This follows via a direct computation, 
which shows that each of the ladder web relations in Appendix \ref{App:Rels}
is sent by $\Phi$ either to a relation in \eqref{eq:spRel}, 
or to a relation easily deduced from these relations.
\end{proof}

We now show that all webs in $\Web(\spn[6])$ have ``preimages'' in 
$\Lad(\gspn[6])$, hence can be studied using the latter category. To begin, we have:

\begin{lem}\label{lem:ladderComp}
Let $L_1 \in \Hom_{\Lad(\gspn[6])}(\vec{x},\vec{y_c})$ and 
$L_2 \in \Hom_{\Lad(\gspn[6])}(\vec{y_d},\vec{z})$
be ladders such that $\Phi(\vec{y_c}) = \Phi(\vec{y_d})$, 
then there exist ladders
$\tilde{L_1} \in \Hom_{\Lad(\gspn[6])}(\tilde{x},\tilde{y})$ and 
$\tilde{L_2} \in \Hom_{\Lad(\gspn[6])}(\tilde{y},\tilde{z})$
so that $\Phi(\tilde{L_i})=\Phi(L_i)$.
\end{lem}

Informally, this lemma says that if we have ladders 
with composable images under $\Phi$ (but are perhaps not composable themselves), 
then we can find composable ladders with the same images under $\Phi$.

\begin{proof}
We repeatedly post-compose $L_1$ and pre-compose $L_2$ 
with rungs of the form:
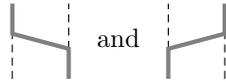
\[
\begin{tikzpicture}[anchorbase]
	\draw[0label] (0,0) to (0,.6);
	\draw[0label] (.75,.4) to (.75,1);
	\draw[glabel] (.75,0) to (.75,.4) to (0,.6) to (0,1);
\end{tikzpicture}
\quad \text{and} \quad
\begin{tikzpicture}[anchorbase,xscale=-1]
	\draw[0label] (0,0) to (0,.6);
	\draw[0label] (.75,.4) to (.75,1);
	\draw[glabel] (.75,0) to (.75,.4) to (0,.6) to (0,1);
\end{tikzpicture}
\]
respectively, to obtain ladders $\hat{L_1}$ and $\hat{L_2}$ with 
the codomain of $\hat{L_1}$ equal to $\hat{y_c}=y_1^{(r_1)}\cdots y_k^{(r_k)} 0^{(i_1)}\cdots 0^{(i_l)}$
and the domain of $\hat{L_2}$ equal to $\hat{y_d}=y_1^{(s_1)}\cdots y_k^{(s_k)} 0^{(j_1)}\cdots 0^{(j_m)}$
with $y_i \neq 0$ for all $i=1,\ldots,k$.
By construction, we have that $\Phi(\hat{L_i}) = \Phi(L_i)$.

Now, if $l \geq m$ then define $\overline{L_2} = \hat{L_2} \otimes \id_{0^{l-m}}$, 
which has domain $\hat{y_d}=y_1^{(s_1)}\cdots y_k^{(s_k)} 0^{(j_1)}\cdots 0^{(j_l)}$
with $j_t = 0$ for $m+1 \leq t \leq l$.
Finally we define $\tilde{L_1}$ and $\tilde{L_2}$ as follows.
For $1 \leq i \leq k$, 
add $|r_i - s_i|$ to the exponent of every label in the $i^{th}$ upright of 
$\hat{L_1}$ if $r_i < s_i$, and in the $i^{th}$ upright of $\overline{L_2}$ otherwise. 
Similarly, for $1 \leq t \leq l$, add $|i_t - j_t|$ to the exponent of every label in the 
$(k+t)^{th}$ upright of $\hat{L_1}$ if $i_t < j_t$, and in the $(k+t)^{th}$ upright of $\overline{L_2}$ otherwise. 
It follows that the codomain of $\tilde{L_1}$ and the domain of $\tilde{L_2}$ are both equal to 
$\tilde{y} = y_1^{(\max(r_1,s_1))} \cdots y_k^{(\max(r_k,s_k))} 0^{(\max(i_1,j_1))} \cdots 0^{(\max(i_l,j_l))}$, 
and $\Phi(\tilde{L_i}) = \Phi(L_i)$ by construction. The case $l \leq m$ can be handled similarly.
\end{proof}

\begin{example}
Let
\[
L_1=
\begin{tikzpicture}[anchorbase]
	\draw[2label] (0,0) to (0,.4);
	\draw[1label] (0,.4) to (0,1.5);
	\draw[0label] (.75,0) to (.75,.6);
	\draw[1label] (.75,.6) to (.75,1);
	\draw[0label] (.75,1) to (.75,1.5);
	\draw[2label] (1.5,0) to (1.5,.8);
	\draw[1label] (1.5,.8) to (1.5,1.5);
	\draw[1label] (0,.4) to (.75,.6);
	\draw[1label] (1.5,.8) to (.75,1);
\end{tikzpicture}
\;\; , \;\;
L_2=
\begin{tikzpicture}[anchorbase]
	\draw[1label] (0,0) to (0,.6);
	\draw[2label] (0,.6) to (0,1);
	\draw[1label] (.75,0) to (.75,.4);
	\draw[0label] (.75,.4) to (.75,1);
	\draw[1label] (.75,.4) to (0,.6);
\end{tikzpicture}
\]
where here we omit the explicit exponents in the (co)domains.
Note that $L_2\circ L_1$ is not defined, 
while $\Phi(L_2)\circ\Phi(L_1)$ is defined. 
Lemma \ref{lem:ladderComp} 
then produces the composable ladders
\[
\tilde{L}_1=
\begin{tikzpicture}[anchorbase]
	\draw[2label] (0,0) to (0,.4);
	\draw[1label] (0,.4) to (0,2);
	\draw[0label] (.75,0) to (.75,.6);
	\draw[1label] (.75,.6) to (.75,1);
	\draw[0label] (.75,1) to (.75,1.7);
	\draw[2label] (1.5,0) to (1.5,.8);
	\draw[1label] (1.5,.8) to (1.5,1.5);
	\draw[1label] (0,.4) to (.75,.6);
	\draw[1label] (1.5,.8) to (.75,1);
	\draw[1label] (1.5,1.5) to (.75,1.7);
	\draw[1label] (.75,1.7) to (.75,2);
	\draw[0label] (1.5,1.5) to (1.5,2);
\end{tikzpicture}
\;\; , \;\;
\tilde{L}_2=
\begin{tikzpicture}[anchorbase]
	\draw[1label] (0,0) to (0,.6);
	\draw[2label] (0,.6) to (0,1);
	\draw[1label] (.75,0) to (.75,.4);
	\draw[0label] (.75,.4) to (.75,1);
	\draw[1label] (.75,.4) to (0,.6);
	\draw[0label] (1.5,0) to (1.5,1);
\end{tikzpicture}
\]
with the same image under $\Phi$ as $L_1$ and $L_2$.
\end{example}

\begin{cor}\label{cor:ladderize}
Let $\mathcal{W}$ be a web in $\Web(\spn[6])$,
then there exists a ladder $L_{\mathcal{W}}$ 
in $\Lad(\gspn[6])$ with $\Phi(L_{\mathcal{W}})=\mathcal{W}$. 
\end{cor}
\begin{proof}
By applying a planar isotopy, we can assume that 
all trivalent vertices and horizontal tangencies in $\mathcal{W}$ occur at distinct heights.
The result then follows by inductively applying Lemma \ref{lem:ladderComp}, 
using the fact that $\Phi$ is monoidal and that 
for all $a,b,c$ we can find $i,j,s,t$ so that 
\[
\Phi \left(
\begin{tikzpicture}[anchorbase, smallnodes,scale=.75]
	\draw[glabel] (0,0) node[below,color=black]{$a^{(i)}$} to (0,.6);
	\draw[glabel] (0,.6) to (0,1) node[above,yshift=-2pt,color=black]{$c^{(s)}$};
	\draw[glabel] (.75,0) node[below,color=black]{$b^{(j)}$} to (.75,.4);
	\draw[0label] (.75,.4) to (.75,1) node[above,yshift=-2pt]{$0^{(t)}$};
	\draw[glabel] (.75,.4) to (0,.6);
\end{tikzpicture}
\right) =
\xy
(0,0)*{
\begin{tikzpicture}[scale =.5,smallnodes]
	\draw[glabel] (0,0) node[below,yshift=2pt,color=black]{$a$} to [out=90,in=210] (.5,.75);
	\draw[glabel] (1,0) node[below,yshift=2pt,color=black]{$b$} to [out=90,in=330] (.5,.75);
	\draw[glabel] (.5,.75) to (.5,1.5) node[above,yshift=-2pt,color=black]{$c$};
\end{tikzpicture}
};
\endxy
\;\; , \;\;
\Phi \left(
\begin{tikzpicture}[anchorbase, smallnodes,scale=.75,yscale=-1]
	\draw[glabel] (0,0) node[above,yshift=-2pt,color=black]{$a^{(s)}$} to (0,.6);
	\draw[glabel] (0,.6) to (0,1) node[below,color=black]{$c^{(i)}$};
	\draw[glabel] (.75,0) node[above,yshift=-2pt,color=black]{$b^{(t)}$} to (.75,.4);
	\draw[0label] (.75,.4) to (.75,1) node[below]{$0^{(j)}$};
	\draw[glabel] (.75,.4) to (0,.6);
\end{tikzpicture}
\right) = 
\xy
(0,0)*{
\begin{tikzpicture}[scale =.5,smallnodes,yscale=-1]
	\draw[glabel] (0,0) node[above,yshift=-2pt,color=black]{$a$} to [out=90,in=210] (.5,.75);
	\draw[glabel] (1,0) node[above,yshift=-2pt,color=black]{$b$} to [out=90,in=330] (.5,.75);
	\draw[glabel] (.5,.75) to (.5,1.5) node[below,yshift=2pt,color=black]{$c$};
\end{tikzpicture}
};
\endxy
\;\; , \;\;
\Phi \left(
\begin{tikzpicture}[anchorbase,smallnodes,scale=.75,yscale=-1]
	\draw[0label] (0,.6) to (0,1.2) node[below]{$0^{(i)}$};
	\draw[0label] (.75,.4) to (.75,1.2) node[below]{$0^{(j)}$};
	\draw[glabel] (0,0) node[above,yshift=-2pt,color=black]{$a^{(s)}$} to (0,.6)
		to (.75,.4) to (.75,0) node[above,yshift=-2pt,color=black]{$a^{(t)}$};
\end{tikzpicture}
\right) =
\xy
(0,0)*{
\begin{tikzpicture}[scale =.75, smallnodes,yscale=-1]
	\draw[glabel] (0,0) node[above,yshift=-2pt,color=black]{$a$} to [out=90,in=180] (.375,.5);
	\draw[glabel] (.375,.5) to [out=0,in=90] (.75,0);
\end{tikzpicture}
};
\endxy
\;\; , \;\;
\Phi \left(
\begin{tikzpicture}[anchorbase,smallnodes,scale=.75]
	\draw[0label] (0,.6) to (0,1.2) node[above,yshift=-2pt]{$0^{(s)}$};
	\draw[0label] (.75,.4) to (.75,1.2) node[above,yshift=-2pt]{$0^{(t)}$};
	\draw[glabel] (0,0) node[below,color=black]{$a^{(i)}$} to (0,.6)
		to (.75,.4) to (.75,0) node[below,color=black]{$a^{(j)}$};
\end{tikzpicture}
\right) =
\xy
(0,0)*{
\begin{tikzpicture}[scale =.75, smallnodes]
	\draw[glabel] (0,0) node[below,yshift=2pt,color=black]{$a$} to [out=90,in=180] (.375,.5);
	\draw[glabel] (.375,.5) to [out=0,in=90] (.75,0);
\end{tikzpicture}
};
\endxy
\]
\end{proof}

\subsection{Closed web evaluation}

In this section we show that every closed web, 
i.e. every endomorphism of the monoidal unit (empty sequence) 
in $\Web(\spn[6])$, is equal to a $\C(q)$-multiple of the empty web.
We begin with the following observation.

\begin{lem}\label{lem:PBW}
Any ladder $L \in \Lad(\gspn[6])$ can be expressed $\C(q)$-linear combination of ladders in which 
no $E$-rungs appear above $F$-rungs.
\end{lem}
\begin{proof}
Repeat use of the run explosion relation \eqref{eq:explosion} 
expresses $L$ as a linear combination of ladders in which 
every rung is $1$- or $0'$-labeled. 
The rung swap relation \eqref{eq:swap} and square relation \eqref{eq:square} can then 
be inductively used to express each ladder in terms of ladders 
in which fewer $F$-rungs appear below $E$-rungs.
\end{proof}

\begin{thm}\label{thm:Empty}
$\End_{\Web(\spn[6])}(\varnothing)=\C(q)$.
\end{thm}

\begin{proof}
Existence of the functor 
$\Psi: \Web(\spn[6]) \to \FRep(U_q(\spn[6]))$
implies that 
\[
\dim_{\C(q)}\left( \End_{\Web(\spn[6])}(\varnothing) \right) \geq 1,
\] 
thus it suffices to show the opposite inequality, i.e. that every closed web 
$\mathcal{W} \in \End_{\Web(\spn[6])}(\varnothing)$ is equal to a multiple of the empty web.

Given such $\mathcal{W}$, 
Corollary \ref{cor:ladderize} gives a ladder $\tilde{\mathcal{W}}$ 
with $\Phi(\tilde{\mathcal{W}}) = \mathcal{W}$, 
and further, by pre- and post-composing with $0^{(i)}$-labeled rungs, 
we can assume that 
\[
\tilde{\mathcal{W}} \in \End_{\Lad(\gspn[6])}(0'^{\otimes k}\otimes 0^{\otimes\ell})
\] 
where 
$0'^{\otimes k} = \overbrace{0' \cdots 0'}^k$
and $0^{\otimes\ell} = \overbrace{0 \cdots 0}^l$.
We induct on the minimal such $k$ for which such $\tilde{\mathcal{W}}$ exists, 
noting that all morphisms in 
\[
\End_{\Lad(\gspn[6])}(0 \cdots 0)
\]
are multiples of the identity (i.e. the empty ladder).

For the inductive step, 
suppose that $\tilde{\mathcal{W}} \in \End_{\Lad(\gspn[6])}(0'^{\otimes k}\otimes 0^{\otimes\ell})$ 
for $k \geq 1$. 
Lemma \ref{lem:PBW} then implies that 
\[
\tilde{\mathcal{W}} = \sum_{i=1}^m f_i(q) \tilde{\mathcal{W}}_i
\]
where, in each $\tilde{\mathcal{W}}_i$, no $E$-rungs appear above any $F$-rungs.
Since the first entry of both the domain and codomain are $0'$, 
this implies that every label on the left-most upright of $\tilde{\mathcal{W}}_i$
takes the form $k^{(i)}$ for $i>0$.

Consider the web
$\tilde{\mathcal{W}}^{(-1,\vec{0})}_i \in 
\End_{\Lad(\gspn[6])}(0\otimes 0'^{\otimes {k-1}}\otimes 0^{\otimes\ell})$
that is obtained from the web $\tilde{\mathcal{W}}_i$ 
by replacing every label $k^{(i)}$ on the left-most upright by the label $k^{(i-1)}$. 
Finally, let 
\[
\tilde{\tilde{W_i}} \in \End_{\Lad(\gspn[6])}(0'^{\otimes k-1}\otimes 0^{\otimes\ell+1})
\]
be obtained from $\tilde{\mathcal{W}}^{(-1,\vec{0})}_i$ by repeatedly pre- and post-composing 
with rungs of the form 
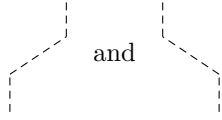
\[
\begin{tikzpicture}[anchorbase]
	\draw[0label] (0,0) to (0,.5);
	\draw[0label] (.75,1) to (.75,1.5);
	\draw[0label] (0,.5) to (.75,1);
\end{tikzpicture}
\quad \text{and} \quad
\begin{tikzpicture}[anchorbase,xscale=-1]
	\draw[0label] (0,0) to (0,.5);
	\draw[0label] (.75,1) to (.75,1.5);
	\draw[0label] (0,.5) to (.75,1);
\end{tikzpicture}
\]
respectively. 
Note that 
$\Phi(\tilde{\tilde{W_i}}) = \Phi(\tilde{\mathcal{W}}^{(-1,\vec{0})}_i )
= \Phi(\tilde{\mathcal{W}}_i)$.
We then have 
\[
\mathcal{W} = \Phi(\tilde{W}) = \sum_{i=1}^m \Phi(\tilde{\mathcal{W}}_i)
=\sum_{i=1}^m \Phi(\tilde{\tilde{W_i}}) \in \C(q)
\]
as desired.
\end{proof}

Theorem \ref{thm:Empty} can be equivalently formulated as saying that
the linear map
\[
\End_{\Lad(\gspn[6])}(\varnothing) \to \End_{\Web(\spn[6])}(\varnothing)
\]
is surjective. 
In fact, a variation of the above proof gives the following.

\begin{prop}\label{prop:Surjective}
Let $k_i \in \{1,2,3\}$ for $1 \leq i \leq m$, then 
\[
\End_{\Lad(\gspn[6])}(k_1\cdots k_m) \to \End_{\Web(\spn[6])}(k_1\cdots k_m)
\]
is surjective.
\end{prop}

\begin{proof}
The argument closely parallels the proof of Theorem \ref{thm:Empty}.
Let $\mathcal{W} \in \End_{\Web(\spn[6])}(k_1\cdots k_m)$ be a web, 
then we can find $\tilde{\mathcal{W}} \in \End_{\Lad(\gspn[6])}( 0'^{\otimes k}\otimes k_1\cdots k_m \otimes 0^{\otimes\ell})$ 
so that $\Phi(\tilde{\mathcal{W}}) = \mathcal{W}$. 
Again using Lemma \ref{lem:PBW}, we can write 
\[
\tilde{\mathcal{W}} = \sum_{i=1}^m f_i(q) \tilde{\mathcal{W}}_i
\]
with no $E$-rungs appearing above any $F$-rungs in each $\tilde{\mathcal{W}}_i$.
The same procedure as above then implies that 
\[
\mathcal{W} =  \sum_{i=1}^m f_i(q) \Phi(\tilde{\tilde{\mathcal{W}}}_i)
\]
with each $\tilde{\tilde{\mathcal{W}}}_i \in \End_{\Lad(\gspn[6])}( 0'^{\otimes k-1}\otimes k_1\cdots k_m \otimes 0^{\otimes \ell+1})$. 
Repeating this, we find that
\[
\mathcal{W} = \sum_{j=1}^p g_j(q) \Phi(\widehat{\mathcal{W}}_j)
\]
with each $\widehat{\mathcal{W}}_j \in \End_{\Lad(\gspn[6])}(k_1\cdots k_m \otimes 0^{\otimes r})$.
Further, again using Lemma \ref{lem:PBW}, we can assume that in each $\widehat{\mathcal{W}}_j$ 
has no $E$-rungs appearing above any $F$-rungs. 
However, this then implies that the webs $\widehat{\mathcal{W}}_j$ take the form
\[
\widehat{\mathcal{W}}_j = \widehat{\widehat{\mathcal{W}}}_j \otimes \id_{0^{\otimes r}}
\]
with $\widehat{\widehat{\mathcal{W}}}_j \in \End_{\Lad(\gspn[6])}(k_1\cdots k_m)$, 
and the result follows since
\[
\mathcal{W} = \sum_{j=1}^p g_j(q) \Phi(\widehat{\widehat{\mathcal{W}}}_j).
\]
\end{proof}

Proposition \ref{prop:Surjective} suggests a clear strategy for the resolution 
of Conjecture \ref{conj}: we should deduce the bound in equation \eqref{eq:HomDim} 
by finding explicit bases for the endomophism algebras in $\Lad(\gspn[6])$, 
and comparing dimensions to those in $\Rep(U_q(\spn[6]))$. 
This approach will be pursued in follow-up work \cite{BELR}.

\begin{cor}
For all objects $\vec{k},\vec{\ell}$ in $\Web(\spn[6])$, 
the $\C(q)$-vector space $\Hom_{\Web(\spn[6])}(\vec{k},\vec{\ell})$ 
is finite-dimensional.
\end{cor}

\begin{proof}
Recall from the proof of Theorem \ref{thm:Full} that there is a surjective linear map
\[
\Hom_{\Web(\spn[6])}(1^{\otimes \sum k_i} , 1^{\otimes \sum \ell_j})
\xrightarrow{\mathcal{W}_t \circ ( - ) \circ \mathcal{W}_b} 
\Hom_{\Web(\spn[6])}(\vec{k}, \vec{\ell}).
\]
and an isomorphism
\[
\Hom_{\Web(\spn[6])}(1^{\otimes \sum k_i} , 1^{\otimes \sum \ell_j}) \cong 
\End_{\Web(\spn[6])}(1^{\otimes \frac{1}{2}(\sum k_i + \sum \ell_j)}).
\]
The result then follows from Proposition \ref{prop:Surjective}, 
since Lemma \ref{lem:PBW} shows that 
$\End_{\Lad(\gspn[6])}(1^{\otimes \frac{1}{2}(\sum k_i + \sum \ell_j)})$ is spanned by ladders 
in which no $E$-rungs appear above any $F$-rungs. 
It is easy to see that there are only finitely many such ladders.
\end{proof}

\subsection{Decategorification of $\Web(\spn[6])$}

In this section, we compute a ``decategorification'' of $\Web(\spn[6])$. 
Typically, the decategorification of an abelian or additive category $\mathcal{C}$ is 
taken to be the Grothendieck group $K_0(\mathcal{C})$, 
i.e. the quotient of the free abelian group on 
the isomorphism classes $[X]$ of objects $X \in Ob(\mathcal{C})$
by relations corresponding to an appropriate notion of exact sequence.
Unfortunately, computing $K_0$ for (the additive closure of) diagrammatic categories 
such as $\Web(\spn[6])$ is typically a difficult endeavor, 
requiring detailed knowledge of the endomorphism algebras therein.

In \cite{BGHL}, the following is proposed as an alternative notion of 
decategorification.

\begin{defn}
Let $\K$ be a field and let $\mathcal{C}$ be a $\K$-linear category.
The categorical trace of $\mathcal{C}$ is the $\K$-vector space
\[
\Tr(\mathcal{C}):= 
\bigoplus_{X\in Ob(\mathcal{C})} \End_{\mathcal{C}}(X) \Big/ \big( fg-gf \big)
\]
where $f\in \Hom_{\mathcal{C}}(X,Y)$ and 
$g\in \Hom_{\mathcal{C}}(Y,X)$ range over all $X,Y\in Ob(\mathcal{C})$.
\end{defn}

We will denote the equivalence class in $\Tr(\mathcal{C})$ of an
endomorphism $h \in \End_{\mathcal{C}}(X)$ by $\tr(h)$.
We now record some standard facts about the categorical trace:
\begin{itemize}

\item $\Tr(\operatorname{Kar}(\mathcal{C})) 
= \Tr(\mathcal{C})$ where $\operatorname{Kar}(\mathcal{C})$ 
denotes the Karoubi (i.e. idempotent) completion of $\mathcal{C}$.

\item If $\mathcal{C}$ is additive linear, 
there is a generalized Chern character map 
$\mathcal{X}: K_0(\mathcal{C}) \to \Tr(\mathcal{C})$ defined by 
$\mathcal{X}([X]) = \tr(\id_X)$.

\item If $\mathcal{C}$ is semisimple, then $\mathcal{X}$ induces an isomorphism
$\K \otimes_{\Z} K_0(\mathcal{C}) \to \Tr(\mathcal{C})$.

\item If $\mathcal{C}$ is (braided) monoidal, 
then $\Tr(\mathcal{C})$ is a (commutative) algebra, with product induced via tensor product, 
and $\mathcal{X}$ is an algebra homomorphism.
\end{itemize}

Combining these facts, we observe that
\begin{equation}\label{eq:TrIso}
\Tr(\FRep(U_q(\spn[6]))) \cong \Tr(\Rep(U_q(\spn[6]))) \cong 
\C(q) \otimes_{\Z} K_0(\Rep(U_q(\spn[6])))
\cong \C(q)[\chi_1, \chi_2,\chi_3]
\end{equation}
where we have used the standard identification of 
the representation ring $K_0(\Rep(U_q(\spn[6])))$ with polynomials 
in the classes (or characters) $\chi_i$ of the fundamental representations. 
We now prove the analogue of this result for $\Web(\spn[6])$
using the fact that the categorical trace is well-suited to diagrammatically 
presented pivotal categories. 
Indeed, in this case it can be identified with the vector space of all annular closures of 
endomorphisms, modulo isotopy in the annulus and local relations 
(now applied in any disk in the annulus). 
We employ this identification in the following.

\begin{thm}\label{thm:traceFRep}
There is an isomorphism of algebras
\[ 
\tau: \C(q)[\chi_1, \chi_2,\chi_3] \to \Tr(\Web(\spn[6]))
\]
determined by $\tau(\chi_i) = \tr(\id_i)$
\end{thm}

\begin{proof}
Since $\Tr(\Web(\spn[6]))$ is a commutative algebra, 
the assignment $\tau(\chi_i) = \tr(\id_i)$ determines an algebra homomorphism 
$\tau: \C(q)[\chi_1, \chi_2,\chi_3] \to \Tr(\Web(\spn[6]))$.
Further, injectivity of $\tau$ follows from the commutative diagram
\[
\begin{tikzcd}
\C(q)[\chi_1, \chi_2,\chi_3] \ar[r,"\tau"] \ar[d,"\cong"] & \Tr(\Web(\spn[6])) \ar[d,"\Tr(\Psi)"] \\
K_0(\Rep(U_q(\spn[6]))) \ar[r,"\cong"]& \Tr(\FRep(U_q(\spn[6])))
\end{tikzcd}
\]
where the composition of the indicated isomorphisms is the inverse of the isomorphism 
given in \eqref{eq:TrIso}.

It remains to show surjectivity, 
i.e. that every class in $\Tr(\Web(\spn[6]))$ can be expressed as 
a linear combination of classes of identity morphisms.
First, note that the functor 
$\Phi: \Lad(\gspn[6]) \to \Web(\spn[6])$ induces a homomorphism 
$\Tr(\Phi): \Tr(\Lad(\gspn[6])) \to \Tr(\Web(\spn[6]))$ 
and Corollary \ref{cor:ladderize} implies that $\Tr(\Phi)$ is surjective. 
It thus suffices to show that every element in $\Tr(\Lad(\gspn[6]))$ 
can be expressed as a linear combination of classes of identity morphisms. 
To do so, we adapt the argument from \cite[Theorem 3.2]{QRS} to our setting. 

To this end, let $L \in \End_{\Lad(\gspn[6])}(a_1^{i_1}\cdots a_m^{i_m})$ be a ladder. 
Using the rung explosion relation \eqref{eq:explosion}, we can assume that all rungs 
are $1$- and $0'$-labeled, and further we can assume that all rungs appear at distinct heights.
The element $\tr(L)$ then corresponds to a ``cyclic sequence'' 
$\vec{\mu}_1, \ldots,\vec{\mu}_{\#r(L)}$ of tuples $\vec{\mu}_i \in \N^m$
that is obtained by taking a horizontal slice in between the rungs of $\tr(L)$, 
and recording the masses of the corresponding labels. 
Here, $\#r(L)$ denotes the 
number\footnote{If $\#r(L)=0$, the cyclic sequence only has one entry, 
thus we slightly abuse notation.} 
of rungs in $L$. 
For example, 
taking a slice at the ``bottom'' (or equivalently, the ``top'') of $\tr(L)$ gives the tuple
\[
\vec{\mu} = (\mu(a_1^{i_1}), \ldots, \mu(a_m^{i_m}))
\]
Note that the sum of the entries in the tuple $\vec{\mu}_i$ is independent 
of the choice of $1 \leq i \leq \#r(L)$, hence determines an invariant of 
$\tr(L)$ that we denote by $\Tot_{\mu}(L)$.

Now, suppose $\#r(L) \geq 1$ 
and consider a tuple in $\vec{\mu}_1, \ldots,\vec{\mu}_{\#r(L)}$ 
that is minimal with respect to the lexicographic order on $\N^m$.
It follows that the rung immediately ``below'' the corresponding slice 
in $\tr(L)$ is an $F$-rung, 
and the rung immediately ``above'' is an $E$-rung. 
We thus can apply a rung swap \eqref{eq:swap} or square \eqref{eq:square} 
relation to express $\tr(L)$ as a linear combination of classes of ladders 
in which the minimal tuple is strictly larger than that in $\tr(L)$, 
or in which it appears a fewer number of times.

Repeat application of this procedure expresses $\tr(L)$ 
as a linear combination of ladders without rungs, 
i.e. identity ladders. This follows from the following observations:
\begin{itemize}
\item We can apply the above procedure, provided $\#r(L) \geq 1.$
\item If $\#r(L)=0$, then $L$ is an identity ladder.
\item For a fixed value of $\Tot_{\mu}(L)$ 
(which remains constant when we apply rung swap and square relations), 
for minimal $\vec{\mu}_i$ sufficiently large, we must have $\#r(L)=0$.
\end{itemize}
\end{proof}

\section{$\spn[6]$ link invariants}\label{sec:Links}

Our results thus far assemble to give an explicit description of the colored 
$U_q(\spn[6])$ invariant of framed links. 
Recall that in the uncolored case, 
this link invariant can be recovered skein-theoretically as an 
appropriate evaluation of the $2$-variable Kauffman polynomial \cite{Kauff2}, 
or using the $n=3$ case of the state-sum model from \cite{MOSpn}. 
We emphasize that our construction has the following important features, 
that are not present in those formulations:
\begin{itemize}
	\item It describes the colored $U_q(\spn[6])$ link invariant, 
where link components are colored by fundamental representations. 
Further, there should exist Jones-Wenzl-like recursions for highest weight 
projectors in $\Web(\spn[6])$ that extend this invariant to links with 
components colored by arbitrary irreducible representations.
	\item It is local, i.e. it assigns a linear combination of $\spn[6]$ webs to 
tangles, as well as links. The invariant of a link can thus be computed 
via the ``divide and conquer'' approach described in \cite{BN3}, 
i.e. by splitting the link into constituent tangles, computing and simplifying 
the invariant of these tangles, then assembling the link invariant from 
these constituent pieces. See \emph{loc. cit.} for a discussion of 
the efficiency of this approach.
\end{itemize}

We now describe the link invariant. 
First, straightforward (but tedious!) computations show that the crossing formulae 
from equations \eqref{eq:braiding} and \eqref{eq:coloredbraiding}
are explicitly given as follows:

\begin{equation}\label{eq:allcross}
\begin{gathered}
\xy
(0,0)*{
\begin{tikzpicture}[scale=.4]
	\draw[1label] (1,-1) to (-1,1);
	\draw[overcross] (-1,-1) to (1,1);
	\draw[1label] (-1,-1) to (1,1);
\end{tikzpicture}
};
\endxy
=q\,
\xy
(0,0)*{
\begin{tikzpicture}[scale=.4]
	\draw[1label] (-1,-1) to [out=45,in=-90] (-.5,0);
	\draw[1label] (-.5,0) to [out=90,in=-45] (-1,1);
	\draw[1label] (1,-1) to [out=135,in=-90] (.5,0);
	\draw[1label] (.5,0) to [out=90,in=-135] (1,1);
\end{tikzpicture}
};
\endxy
+\frac{q^{-3}}{[3]}\,
\xy
(0,0)*{
\begin{tikzpicture}[scale=.4,rotate=90]
	\draw[1label] (-1,-1) to [out=45,in=-90] (-.5,0);
	\draw[1label] (-.5,0) to [out=90,in=-45] (-1,1);
	\draw[1label] (1,-1) to [out=135,in=-90] (.5,0);
	\draw[1label] (.5,0) to [out=90,in=-135] (1,1);
\end{tikzpicture}
};
\endxy
-\frac{1}{[3]}\,
\xy
(0,0)*{
\begin{tikzpicture}[scale=.4]
	\draw[1label] (-1,-1) to (0,-.5);
	\draw[1label] (1,-1) to (0,-.5);
	\draw[2label] (0,-.5) to (0,.5);
	\draw[1label] (0,.5) to (-1,1);
	\draw[1label] (0,.5) to (1,1);
\end{tikzpicture}
};
\endxy\\
\xy
(0,0)*{
\begin{tikzpicture}[scale=.4]
	\draw[2label] (1,-1) to (-1,1);
	\draw[overcross] (-1,-1) to (1,1);
	\draw[1label] (-1,-1) to (1,1);
\end{tikzpicture}
};
\endxy
=
\frac{1}{[2]}\,
\xy
(0,0)*{
\begin{tikzpicture}[scale=.4]
	\draw[1label] (-1,-1) to (0,-.5);
	\draw[2label] (1,-1) to (0,-.5);
	\draw[3label] (0,-.5) to (0,.5);
	\draw[2label] (0,.5) to (-1,1);
	\draw[1label] (0,.5) to (1,1);
\end{tikzpicture}
};
\endxy
- \frac{q}{[3]}\,
\xy
(0,0)*{
\begin{tikzpicture}[scale=.4,rotate=90]
	\draw[2label] (-1,-1) to (0,-.5);
	\draw[1label] (1,-1) to (0,-.5);
	\draw[1label] (0,-.5) to (0,.5);
	\draw[1label] (0,.5) to (-1,1);
	\draw[2label] (0,.5) to (1,1);
\end{tikzpicture}
};
\endxy
-\frac{q^{-2}}{[2][3]}\,
\xy
(0,0)*{
\begin{tikzpicture}[scale=.4]
	\draw[1label] (-1,-1) to (0,-.5);
	\draw[2label] (1,-1) to (0,-.5);
	\draw[1label] (0,-.5) to (0,.5);
	\draw[2label] (0,.5) to (-1,1);
	\draw[1label] (0,.5) to (1,1);
\end{tikzpicture}
};
\endxy
\quad , \quad
\xy
(0,0)*{
\begin{tikzpicture}[scale=.4]
	\draw[1label] (1,-1) to (-1,1);
	\draw[overcross] (-1,-1) to (1,1);
	\draw[2label] (-1,-1) to (1,1);
\end{tikzpicture}
};
\endxy
=
\frac{1}{[2]}\,
\xy
(0,0)*{
\begin{tikzpicture}[scale=.4,xscale=-1]
	\draw[1label] (-1,-1) to (0,-.5);
	\draw[2label] (1,-1) to (0,-.5);
	\draw[3label] (0,-.5) to (0,.5);
	\draw[2label] (0,.5) to (-1,1);
	\draw[1label] (0,.5) to (1,1);
\end{tikzpicture}
};
\endxy
- \frac{q}{[3]}\,
\xy
(0,0)*{
\begin{tikzpicture}[scale=.4,rotate=90,xscale=-1]
	\draw[2label] (-1,-1) to (0,-.5);
	\draw[1label] (1,-1) to (0,-.5);
	\draw[1label] (0,-.5) to (0,.5);
	\draw[1label] (0,.5) to (-1,1);
	\draw[2label] (0,.5) to (1,1);
\end{tikzpicture}
};
\endxy
-\frac{q^{-2}}{[2][3]}\,
\xy
(0,0)*{
\begin{tikzpicture}[scale=.4,xscale=-1]
	\draw[1label] (-1,-1) to (0,-.5);
	\draw[2label] (1,-1) to (0,-.5);
	\draw[1label] (0,-.5) to (0,.5);
	\draw[2label] (0,.5) to (-1,1);
	\draw[1label] (0,.5) to (1,1);
\end{tikzpicture}
};
\endxy \\
\xy
(0,0)*{
\begin{tikzpicture}[scale=.4]
	\draw[2label] (1,-1) to (-1,1);
	\draw[overcross] (-1,-1) to (1,1);
	\draw[2label] (-1,-1) to (1,1);
\end{tikzpicture}
};
\endxy
=\frac{q^4}{[3]}
\xy
(0,0)*{
\begin{tikzpicture}[scale=.4]
	\draw[2label] (-1,-1) to [out=45,in=-90] (-.5,0);
	\draw[2label] (-.5,0) to [out=90,in=-45] (-1,1);
	\draw[2label] (1,-1) to [out=135,in=-90] (.5,0);
	\draw[2label] (.5,0) to [out=90,in=-135] (1,1);
\end{tikzpicture}
};
\endxy
+\frac{q^{-4}}{[3]}
\xy
(0,0)*{
\begin{tikzpicture}[scale=.4,rotate=90]
	\draw[2label] (-1,-1) to [out=45,in=-90] (-.5,0);
	\draw[2label] (-.5,0) to [out=90,in=-45] (-1,1);
	\draw[2label] (1,-1) to [out=135,in=-90] (.5,0);
	\draw[2label] (.5,0) to [out=90,in=-135] (1,1);
\end{tikzpicture}
};
\endxy
-\frac{q}{[3]}
\xy
(0,0)*{
\begin{tikzpicture}[scale=.4,rotate=90]
	\draw[2label] (-1,-1) to (0,-.5);
	\draw[2label] (1,-1) to (0,-.5);
	\draw[2label] (0,-.5) to (0,.5);
	\draw[2label] (0,.5) to (-1,1);
	\draw[2label] (0,.5) to (1,1);
\end{tikzpicture}
};
\endxy
-\frac{q^{-1}}{[3]}
\xy
(0,0)*{
\begin{tikzpicture}[scale=.4]
	\draw[2label] (-1,-1) to (0,-.5);
	\draw[2label] (1,-1) to (0,-.5);
	\draw[2label] (0,-.5) to (0,.5);
	\draw[2label] (0,.5) to (-1,1);
	\draw[2label] (0,.5) to (1,1);
\end{tikzpicture}
};
\endxy
+\frac{1}{[3]^2}
\xy
(0,0)*{
\begin{tikzpicture}[scale=.4]
	\draw[2label] (-1,-1) to (-.5,-.5);
	\draw[2label] (-1,1) to (-.5,.5);
	\draw[2label] (1,-1) to (.5,-.5);
	\draw[2label] (1,1) to (.5,.5);
	\draw[1label] (-.5,-.5) to (.5,-.5);
	\draw[1label] (-.5,-.5) to (-.5,.5);
	\draw[1label] (.5,-.5) to (.5,.5);
	\draw[1label] (-.5,.5) to (.5,.5);
\end{tikzpicture}
};
\endxy\\
\xy
(0,0)*{
\begin{tikzpicture}[scale=.4]
	\draw[3label] (1,-1) to (-1,1);
	\draw[overcross] (-1,-1) to (1,1);
	\draw[1label] (-1,-1) to (1,1);
\end{tikzpicture}
};
\endxy
=
\frac{q}{[2]}\,
\xy
(0,0)*{
\begin{tikzpicture}[scale=.4,rotate=90]
	\draw[3label] (-1,-1) to (0,-.5);
	\draw[1label] (1,-1) to (0,-.5);
	\draw[2label] (0,-.5) to (0,.5);
	\draw[1label] (0,.5) to (-1,1);
	\draw[3label] (0,.5) to (1,1);
\end{tikzpicture}
};
\endxy
+
\frac{q^{-1}}{[2]}\,
\xy
(0,0)*{
\begin{tikzpicture}[scale=.4]
	\draw[1label] (-1,-1) to (0,-.5);
	\draw[3label] (1,-1) to (0,-.5);
	\draw[2label] (0,-.5) to (0,.5);
	\draw[3label] (0,.5) to (-1,1);
	\draw[1label] (0,.5) to (1,1);
\end{tikzpicture}
};
\endxy
\quad , \quad
\xy
(0,0)*{
\begin{tikzpicture}[scale=.4]
	\draw[1label] (1,-1) to (-1,1);
	\draw[overcross] (-1,-1) to (1,1);
	\draw[3label] (-1,-1) to (1,1);
\end{tikzpicture}
};
\endxy
=
\frac{q}{[2]}\,
\xy
(0,0)*{
\begin{tikzpicture}[scale=.4,rotate=90,xscale=-1]
	\draw[3label] (-1,-1) to (0,-.5);
	\draw[1label] (1,-1) to (0,-.5);
	\draw[2label] (0,-.5) to (0,.5);
	\draw[1label] (0,.5) to (-1,1);
	\draw[3label] (0,.5) to (1,1);
\end{tikzpicture}
};
\endxy
+\frac{q^{-1}}{[2]}\,
\xy
(0,0)*{
\begin{tikzpicture}[scale=.4,xscale=-1]
	\draw[1label] (-1,-1) to (0,-.5);
	\draw[3label] (1,-1) to (0,-.5);
	\draw[2label] (0,-.5) to (0,.5);
	\draw[3label] (0,.5) to (-1,1);
	\draw[1label] (0,.5) to (1,1);
\end{tikzpicture}
};
\endxy\\
\xy
(0,0)*{
\begin{tikzpicture}[scale=.4]
	\draw[3label] (1,-1) to (-1,1);
	\draw[overcross] (-1,-1) to (1,1);
	\draw[2label] (-1,-1) to (1,1);
\end{tikzpicture}
};
\endxy
=
\frac{q^2}{[2]}
\xy
(0,0)*{
\begin{tikzpicture}[scale=.4,rotate=90]
	\draw[3label] (-1,-1) to (0,-.5);
	\draw[2label] (1,-1) to (0,-.5);
	\draw[1label] (0,-.5) to (0,.5);
	\draw[2label] (0,.5) to (-1,1);
	\draw[3label] (0,.5) to (1,1);
\end{tikzpicture}
};
\endxy
+\frac{q^{-2}}{[2]}
\xy
(0,0)*{
\begin{tikzpicture}[scale=.4]
	\draw[2label] (-1,-1) to (0,-.5);
	\draw[3label] (1,-1) to (0,-.5);
	\draw[1label] (0,-.5) to (0,.5);
	\draw[3label] (0,.5) to (-1,1);
	\draw[2label] (0,.5) to (1,1);
\end{tikzpicture}
};
\endxy
+\frac{1}{[2][3]}
\xy
(0,0)*{
\begin{tikzpicture}[scale=.4]
	\draw[2label] (-1,-1) to (-.5,-.5);
	\draw[3label] (-1,1) to (-.5,.5);
	\draw[3label] (1,-1) to (.5,-.5);
	\draw[2label] (1,1) to (.5,.5);
	\draw[2label] (-.5,-.5) to (.5,-.5);
	\draw[2label] (-.5,-.5) to (-.5,.5);
	\draw[1label] (.5,-.5) to (.5,.5);
	\draw[1label] (-.5,.5) to (.5,.5);
\end{tikzpicture}
};
\endxy
\quad , \quad
\xy
(0,0)*{
\begin{tikzpicture}[scale=.4]
	\draw[2label] (1,-1) to (-1,1);
	\draw[overcross] (-1,-1) to (1,1);
	\draw[3label] (-1,-1) to (1,1);
\end{tikzpicture}
};
\endxy
=
\frac{q^2}{[2]}
\xy
(0,0)*{
\begin{tikzpicture}[scale=.4,rotate=90,xscale=-1]
	\draw[3label] (-1,-1) to (0,-.5);
	\draw[2label] (1,-1) to (0,-.5);
	\draw[1label] (0,-.5) to (0,.5);
	\draw[2label] (0,.5) to (-1,1);
	\draw[3label] (0,.5) to (1,1);
\end{tikzpicture}
};
\endxy
+\frac{q^{-2}}{[2]}
\xy
(0,0)*{
\begin{tikzpicture}[scale=.4,xscale=-1]
	\draw[2label] (-1,-1) to (0,-.5);
	\draw[3label] (1,-1) to (0,-.5);
	\draw[1label] (0,-.5) to (0,.5);
	\draw[3label] (0,.5) to (-1,1);
	\draw[2label] (0,.5) to (1,1);
\end{tikzpicture}
};
\endxy
+\frac{1}{[2][3]}
\xy
(0,0)*{
\begin{tikzpicture}[scale=.4,xscale=-1]
	\draw[2label] (-1,-1) to (-.5,-.5);
	\draw[3label] (-1,1) to (-.5,.5);
	\draw[3label] (1,-1) to (.5,-.5);
	\draw[2label] (1,1) to (.5,.5);
	\draw[2label] (-.5,-.5) to (.5,-.5);
	\draw[2label] (-.5,-.5) to (-.5,.5);
	\draw[1label] (.5,-.5) to (.5,.5);
	\draw[1label] (-.5,.5) to (.5,.5);
\end{tikzpicture}
};
\endxy\\
\xy
(0,0)*{
\begin{tikzpicture}[scale=.4]
	\draw[3label] (1,-1) to (-1,1);
	\draw[overcross] (-1,-1) to (1,1);
	\draw[3label] (-1,-1) to (1,1);
\end{tikzpicture}
};
\endxy
=q^3
\xy
(0,0)*{
\begin{tikzpicture}[scale=.4]
	\draw[3label] (-1,-1) to [out=45,in=-90] (-.5,0);
	\draw[3label] (-.5,0) to [out=90,in=-45] (-1,1);
	\draw[3label] (1,-1) to [out=135,in=-90] (.5,0);
	\draw[3label] (.5,0) to [out=90,in=-135] (1,1);
\end{tikzpicture}
};
\endxy
+q^{-3}
\xy
(0,0)*{
\begin{tikzpicture}[scale=.4,rotate=90]
	\draw[3label] (-1,-1) to [out=45,in=-90] (-.5,0);
	\draw[3label] (-.5,0) to [out=90,in=-45] (-1,1);
	\draw[3label] (1,-1) to [out=135,in=-90] (.5,0);
	\draw[3label] (.5,0) to [out=90,in=-135] (1,1);
\end{tikzpicture}
};
\endxy
+\frac{q}{[2]^2}
\xy
(0,0)*{
\begin{tikzpicture}[scale=.4]
	\draw[3label] (-1,-1) to (-.5,-.5);
	\draw[3label] (-1,1) to (-.5,.5);
	\draw[3label] (1,-1) to (.5,-.5);
	\draw[3label] (1,1) to (.5,.5);
	\draw[1label] (-.5,-.5) to (.5,-.5);
	\draw[2label] (-.5,-.5) to (-.5,.5);
	\draw[2label] (.5,-.5) to (.5,.5);
	\draw[1label] (-.5,.5) to (.5,.5);
\end{tikzpicture}
};
\endxy
+\frac{q^{-1}}{[2]^2}
\xy
(0,0)*{
\begin{tikzpicture}[scale=.4]
	\draw[3label] (-1,-1) to (-.5,-.5);
	\draw[3label] (-1,1) to (-.5,.5);
	\draw[3label] (1,-1) to (.5,-.5);
	\draw[3label] (1,1) to (.5,.5);
	\draw[2label] (-.5,-.5) to (.5,-.5);
	\draw[1label] (-.5,-.5) to (-.5,.5);
	\draw[1label] (.5,-.5) to (.5,.5);
	\draw[2label] (-.5,.5) to (.5,.5);
\end{tikzpicture}
};
\endxy
\end{gathered}
\end{equation}
In these formulae, we use a new trivalent vertex defined as follows:
\[
\xy
(0,0)*{
\begin{tikzpicture}[scale =.5, smallnodes]
	\draw[2label] (0,0)to [out=90,in=210] (.5,.75);
	\draw[2label] (1,0) to [out=90,in=330] (.5,.75);
	\draw[2label] (.5,.75) to (.5,1.5);
\end{tikzpicture}
};
\endxy 
= \frac{1}{[3]}\,
\xy
(0,0)*{
\begin{tikzpicture}[scale=.25]
	\draw[1label] (-1,0) to (1,0);
	\draw[1label] (-1,0) to (0,1.732);
	\draw[1label] (1,0) to (0,1.732);
	\draw[2label] (0,1.732) to (0,3.232);
	\draw[2label] (-2.3,-.75) to (-1,0);
	\draw[2label] (2.3,-.75) to (1,0);
\end{tikzpicture}
};
\endxy
\]

Now, suppose that $\mathcal{L} \subset S^3$ is a framed link 
with components colored by elements in $\{1,2,3\}$, 
let $\mathcal{D}_{\mathcal{L}}$ be any diagram for $\mathcal{L}$, 
and let $P_{\spn[6]}(\mathcal{L})$ be the element of $\C(q)$ obtained by applying 
the formulae in \eqref{eq:allcross} to $\mathcal{D}_{\mathcal{L}}$, 
and evaluating the closed webs using Theorem \ref{thm:Empty}.

\begin{prop}
$P_{\spn[6]}(\mathcal{L})$ is an invariant of framed colored links 
(i.e. is independent of the choice of diagram), 
and is equal to the $U_q(\spn[6])$ Reshetikhin-Turaev invariant.
\end{prop}
\begin{proof}
Via the Reshetikhin-Turaev construction \cite{RT1},
the link diagram $\mathcal{D}_{\mathcal{L}}$ determines an 
endomorphism of the trivial representation $\C(q)$ in $\Rep(U_q(\spn[6]))$, 
and this scalar is independent of the choice of diagram. 
The result then follows since $P_{\spn[6]}(\mathcal{L})$ is equal to 
the endomorphism of the monoidal unit $\varnothing$ in $\Web(\spn[6])$ 
determined by $\mathcal{D}_{\mathcal{L}}$, 
which is taken to the Reshetikhin-Turaev invariant via the isomorphism
\[
\End_{\Web(\spn[6])}(\varnothing) \xrightarrow{\Psi}
\End_{\spn[6]}(\C(q))
\]
\end{proof}

\begin{rem}
Suppose that $\mathcal{L}$ is $1$-colored.
The invariant $P_{\spn[6]}(\mathcal{L})$
can then be computed from a framed link diagram 
using the skein relation
\[
\xy
(0,0)*{
\begin{tikzpicture}[scale=.4]
	\draw[1label] (1,-1) to (-1,1);
	\draw[overcross] (-1,-1) to (1,1);
	\draw[1label] (-1,-1) to (1,1);
\end{tikzpicture}
};
\endxy -
\xy
(0,0)*{
\begin{tikzpicture}[scale=.4,rotate=90]
	\draw[1label] (1,-1) to (-1,1);
	\draw[overcross] (-1,-1) to (1,1);
	\draw[1label] (-1,-1) to (1,1);
\end{tikzpicture}
};
\endxy
=
(q-q^{-1}) \;
\Bigg( \;
\xy
(0,0)*{
\begin{tikzpicture}[scale=.4]
	\draw[1label] (-1,-1) to [out=45,in=-90] (-.5,0);
	\draw[1label] (-.5,0) to [out=90,in=-45] (-1,1);
	\draw[1label] (1,-1) to [out=135,in=-90] (.5,0);
	\draw[1label] (.5,0) to [out=90,in=-135] (1,1);
\end{tikzpicture}
};
\endxy
-
\xy
(0,0)*{
\begin{tikzpicture}[scale=.4,rotate=90]
	\draw[1label] (-1,-1) to [out=45,in=-90] (-.5,0);
	\draw[1label] (-.5,0) to [out=90,in=-45] (-1,1);
	\draw[1label] (1,-1) to [out=135,in=-90] (.5,0);
	\draw[1label] (.5,0) to [out=90,in=-135] (1,1);
\end{tikzpicture}
};
\endxy \;
\Bigg)
\]
together with the relations
\[
\xy
(0,0)*{
\begin{tikzpicture}[scale=.5,yscale=-1]
	\draw[1label] (.8,-.4) to [out=180,in=270] (0,1);
	\draw[overcross] (0,-1) to [out=90,in=180] (.8,.4);
	\draw[1label] (0,-1) to [out=90,in=180] (.8,.4);
	\draw[1label] (.8,.4) to [out=0,in=90] (1.1,0) to [out=270,in=0] (.8,-.4);
\end{tikzpicture}
};
\endxy = 
-q^{-7} \;
\xy
(0,0)*{
\begin{tikzpicture}[scale=.5,yscale=-1]
	\draw[1label] (0,-1) to (0,1);
\end{tikzpicture}
};
\endxy
\quad , \quad
\xy
(0,0)*{
\begin{tikzpicture}[scale=.5]
	\draw[1label] (.8,-.4) to [out=180,in=270] (0,1);
	\draw[overcross] (0,-1) to [out=90,in=180] (.8,.4);
	\draw[1label] (0,-1) to [out=90,in=180] (.8,.4);
	\draw[1label] (.8,.4) to [out=0,in=90] (1.1,0) to [out=270,in=0] (.8,-.4);
\end{tikzpicture}
};
\endxy =
-q^{7} \;
\xy
(0,0)*{
\begin{tikzpicture}[scale=.5,yscale=-1]
	\draw[1label] (0,-1) to (0,1);
\end{tikzpicture}
};
\endxy
\quad , \quad
\xy
(0,0)*{
\begin{tikzpicture}[scale =.75, smallnodes]
	\draw[1label] (0,0) circle (.5);
\end{tikzpicture}
};
\endxy
= -\frac{[3][8]}{[4]} = 1-[7]
\]
This implies that $P_{\spn[6]}(\mathcal{L}) \in \Z[q,q^{-1}]$. 
\end{rem}

We also obtain a refined invariant for framed links in the solid torus/thickened annulus. 
Indeed, the formulae in \eqref{eq:allcross} assign an 
element $P_{\spn[6]}^{\mathcal{A}}(\mathcal{L})$
of the $\Web(\spn[6])$ skein algebra of the annulus to the diagram of any such link. 
Theorem \ref{thm:traceFRep} shows that this skein algebra 
is isomorphic to $\C(q)[\chi_1,\chi_2,\chi_3]$. 
Since $\Web(\spn[6])$ is ribbon, we immediately have the following.

\begin{prop}
$P_{\spn[6]}^{\mathcal{A}}(\mathcal{L}) \in \C(q)[\chi_1,\chi_2,\chi_3]$ is an invariant 
of framed annular links.
\end{prop}

\appendix

\begingroup
\allowdisplaybreaks[4]
\renewcommand{\arraystretch}{.25}

\section{Relations in $\Lad(\gspn[6])$}\label{App:Rels}

The relations in $\Lad(\gspn[6])$ are as follows, 
together with those obtained from these via horizontal and vertical reflection. 

\begin{itemize}[leftmargin=*]
\item \textbf{Rung explosion:}
If $r \geq 1$, we have
\[
\xy
(0,0)*{

};
\endxy
\]
and for $r=0$ and $c=2$ or $3$, 
the relations are recorded in the following tables:
\[
%
};
\endxy
\]
and the values of 
\[
\xy
(0,0)*{
%
};
\endxy
\]
are recorded in the following tables:
\[
\begin{tabular}{| c | c | c | c | c | c |}
\hline
& & & & & \\
&
$(y,b,z)=(0,1,0)$
&
$(0,1,2)$
&
$(1,0,1)$
&
$(1,2,1)$
&
$(1,2,3)$ \\
& & & & & \\
\hline
& & & & & \\
$(x,a,w)=(0,1,0)$
&
$
\frac{-[3][8]}{[4]}
\xy
(0,0)*{
\begin{tikzpicture}[anchorbase,scale=.75]
	\draw[0label] (0,0) to (0,2.5);
	\draw[0label] (.75,0) to (.75,2.5);
\end{tikzpicture}
};
\endxy
$
&
$
0
$
&
$
\xy
(0,0)*{
\begin{tikzpicture}[anchorbase,scale=.75]
	\draw[0label] (0,0) to (0,2.5);
	\draw[1label] (.75,0) to (.75,2.5);
\end{tikzpicture}
};
\endxy
$
&
$
-[2][7]
\xy
(0,0)*{
\begin{tikzpicture}[anchorbase,scale=.75]
	\draw[0label] (0,0) to (0,2.5);
	\draw[1label] (.75,0) to (.75,2.5);
\end{tikzpicture}
};
\endxy
$
&
$
0
$
\\
& & & & & \\
\hline
& & & & & \\
\hline
\end{tabular}
\]

\[
\begin{tabular}{| c | c | c | c | c | c |}
\hline
& & & & & \\
&
$(y,b,z)=(2,1,0)$
&
$(2,1,2)$
&
$(2,3,2)$
&
$(3,2,1)$
&
$(3,2,3)$ \\
& & & & & \\
\hline 
& & & & & \\
$(x,a,w)=(0,1,0)$
&
$
0
$
&
$
[2][3]
\xy
(0,0)*{
\begin{tikzpicture}[anchorbase,scale=.75]
	\draw[0label] (0,0) to (0,2.5);
	\draw[2label] (.75,0) to (.75,2.5);
\end{tikzpicture}
};
\endxy
$
&
$
-[6]
\xy
(0,0)*{
\begin{tikzpicture}[anchorbase,scale=.75]
	\draw[0label] (0,0) to (0,2.5);
	\draw[2label] (.75,0) to (.75,2.5);
\end{tikzpicture}
};
\endxy
$
&
$
0
$
&
$
[2][3]
\xy
(0,0)*{
\begin{tikzpicture}[anchorbase,scale=.75]
	\draw[0label] (0,0) to (0,2.5);
	\draw[3label] (.75,0) to (.75,2.5);
\end{tikzpicture}
};
\endxy
$
\\
& & & & & \\
\hline
& & & & & \\
\hline
\end{tabular}
\]
\end{itemize}

\bibliographystyle{plain}

%
\end{document}